\documentclass[a4paper]{amsart}
\usepackage{amsmath}
\usepackage{amssymb,amsfonts,amsthm}
\usepackage[english]{babel}
\usepackage{fancyhdr}
\usepackage[utf8]{inputenc}
\usepackage{csquotes}
\usepackage{mathrsfs}
\usepackage{setspace}
\usepackage{cite}

\usepackage{framed}
\usepackage{geometry}
\usepackage[mathcal]{eucal}
\usepackage{dsfont}
\usepackage{bm}
\usepackage{algorithm2e}
\usepackage{hyperref}
\usepackage{harpoon}

\usepackage{pgf,tikz,tikz-cd}
\newsavebox{\pullback}
\sbox\pullback{%
\begin{tikzpicture}%
\draw (0,0) -- (1ex,0ex);%
\draw (1ex,0ex) -- (1ex,1ex);%
\end{tikzpicture}}
\usetikzlibrary{positioning, arrows.meta, decorations.text, decorations.markings}
\usetikzlibrary{shapes}
\usetikzlibrary{graphs,calc}
\usetikzlibrary{patterns}
\usetikzlibrary{positioning}
\usetikzlibrary{matrix,arrows.meta}
\usetikzlibrary{arrows}
\usetikzlibrary{calc,positioning}

\newcommand{\nocontentsline}[3]{}
\let\origcontentsline\addcontentsline
\newcommand\stoptoc{\let\addcontentsline\nocontentsline}
\newcommand\resumetoc{\let\addcontentsline\origcontentsline}

\newtheorem{thm}{Theorem}[subsection]
\newtheorem{prop}[thm]{Proposition}
\newtheorem{lem}[thm]{Lemma}
\newtheorem{cor}[thm]{Corollary}
\newcommand{\PrL}{\mathcal{P}\mathrm{r}^\mathrm{L}}
\newcommand{\+}{\mkern1mu}
\newcommand{\PrLotimes}{\mathcal{P}\mathrm{r}^{\mathrm{L},\otimes}}

\newcommand{\Mod}{\mathcal{M}\mathrm{od}}
\newcommand{\CAlg}{\mathrm{CAlg}}
\newcommand{\colim}{\mathrm{colim}}
\newcommand{\cofib}{\mathrm{cofib}}
\newcommand{\Spc}{\mathcal{S}\mathrm{pc}}
\newcommand{\Spt}{\mathcal{S}\mathrm{pt}}
\newcommand{\DAlog}{\mathcal{DA}^\mathrm{log}}

\newcommand{\SHlog}{\mathcal{SH}^\mathrm{log}}
\newcommand{\Catinf}{\mathcal{C}\mathrm{at}_\infty}
\newcommand{\Sh}{\mathcal{S}\mathrm{h}}
\newcommand{\Baar}{\mathrm{Bar}}
\newcommand{\mpspace}[3]{\+_{#1}^{}\mathrm{P}^{#2}_{#3}X}
\newcommand{\pBar}[3]{}

\theoremstyle{definition}
\newtheorem{defi}[thm]{Definition}
\newtheorem{rem}[thm]{Remark} 
\newtheorem{exmp}[thm]{Example}

\usepackage{tikz-cd}
\usetikzlibrary{decorations.markings}
\makeatletter

\tikzcdset{scale cd/.style={every label/.append style={scale=#1},
    cells={nodes={scale=#1}}}}
\tikzcdset{
open/.code={\tikzcdset{hook, circled};},
closed/.code={\tikzcdset{hook, slashed};},
open'/.code={\tikzcdset{hook', circled};},
closed'/.code={\tikzcdset{hook', slashed};},
circled/.code={\tikzcdset{markwith={\draw (0,0) circle (.375ex);}};},
slashed/.code={\tikzcdset{markwith={\draw[-] (-.4ex,-.4ex) -- (.4ex,.4ex);}};},
markwith/.code={
\pgfutil@ifundefined{tikz@library@decorations.markings@loaded}%
{\pgfutil@packageerror{tikz-cd}{You need to say %
\string\usetikzlibrary{decorations.markings} to use arrow with markings}{}}{}%
\pgfkeysalso{/tikz/postaction={/tikz/decorate,
/tikz/decoration={
markings,
mark = at position 0.5 with
{#1}}}}},
}
\makeatother

\title{The motivic fundamental groupoid at tangential basepoints}
\author{Sofian Tur-Dorvault}
\address{Institut Montpelliérain Alexander Grothendieck, Université de Montpellier, CNRS, Montpellier, France}
\email{sofian.tur-dorvault@umontpellier.fr}
\begin{document}

\begin{abstract}
   We give a general construction of the motivic fundamental groupoid at tangential basepoints, extending previous works of P. Deligne, A. B. Goncharov, and M. Levine, which were limited to ordinary basepoints or to specific varieties. Given a smooth variety over a field endowed with a simple normal crossings divisor, we encode its tangential basepoints using the language of logarithmic geometry. Building on the recent construction by F. Binda, D. Park, and P. A. Østvær of a stable $\infty$-category of $\mathbb{A}^1$-invariant logarithmic motives and its comparison with the usual category of motives, we define in a functorial manner the associated motivic pointed path spaces. In the presence of a motivic $t$-structure, truncating yields the motivic fundamental groupoid. In general, we construct Betti and de Rham realization functors for logarithmic motives (linearizing the construction of F. Binda, D. Park and P. A. Østvær for the Betti case) and we show that the periods of the motivic fundamental groupoid are given by regularized iterated integration of logarithmic differential $1$-forms, thus yielding a general version of Chen's theorem with tangential basepoints.
 
\end{abstract}
\maketitle
\tableofcontents
\newpage

\maketitle

\section{Introduction}
Let $U$ be a smooth scheme of finite type over $k$ a subfield of $\mathbb{C}$, and $x,y$ be rational points on $U$. Denote by $\pi_1^\mathrm{un}(U(\mathbb{C}),x,y)$ the prounipotent completion of the fundamental torsor of $U(\mathbb{C})$ at $x,y$. P. Deligne conjectured in \cite{10.1007/978-1-4613-9649-9_3} the existence of a motivic incarnation of $\pi_1^\mathrm{un}(U(\mathbb{C}),x,y)$, and in joint work with A. B. Goncharov, he showed in \cite{ASENS_2005_4_38_1_1_0} that $\mathcal{O}(\pi_1^\mathrm{un}(U(\mathbb{C}),x,y))$ arises as the degree-zero homology of the Betti realization of an algebra object in Voevodsky's triangulated category of mixed motives called the \textbf{motivic pointed path space of $X$ based at $x,y$}. In the presence of a motivic $t$-structure (for instance if $U$ is of mixed Tate type over a field satisfying Beilinson--Soulé's conjecture), this construction yields the motivic fundamental torsor $\pi_1^m(U,x,y)$ whose Betti realization is $\pi_1^\mathrm{un}(U(\mathbb{C}),x,y)$.

However, many interesting varieties have no rational points, or more generally no section over their base scheme, as is the case for the moduli spaces of curves of genus $g$ with $n$ marked points $\mathcal{M}_{g,n}$ viewed as $\mathbb{Z}$-schemes. In this situation, one must replace rational basepoints by \textbf{tangential basepoints}, that is, points on a smooth compactification with simple normal crossings divisor, equipped with a nonzero tangent vector normal to each component of the divisor. In certain specific cases, most notably for $\mathcal{M}_{0,4} \simeq \mathbb{P}^1_\mathbb{Z}\setminus \{0,1,\infty\}$, Deligne and Goncharov established \cite[Theorem 4.4]{ASENS_2005_4_38_1_1_0} the existence of the motivic fundamental torsor at tangential basepoints whose realizations were previously defined in \cite{10.1007/978-1-4613-9649-9_3}. However, their proof is transcendental in nature and does not provide an explicit construction.

\smallskip

In this article we generalize this work by giving a construction of the motivic pointed path space based at tangential basepoints on any smooth scheme over a field endowed with a simple normal crossings divisor.
More specifically, let $\mathcal{DA}(k,\Lambda)$ be the stable $\infty$-category of motives over a field $k$ with $\Lambda$ coefficients as defined by J. Ayoub \cite{Ayoub_2022}, D.-C. Cisinski and F. Déglise \cite{Cisinski_2019} and others, and $\mathrm{R}_\mathrm{B},\mathrm{R}_\mathrm{dR}$ its Betti and de Rham realization functors (with $\Lambda = \mathbb{Q},\; k \subset \mathbb{C}$ and $\Lambda = k, \; k$ of characteristic $0$ respectively). The main theorem we will prove is the following:

\begin{thm}[Theorem \ref{thm Betti real pi}, Theorem \ref{deRhamreal}, Theorem \ref{end} Definition \ref{pathspace}] \label{main}
    Let $X$ be a smooth scheme of finite type over an arbitrary field $k$, endowed with a simple normal crossings divisor $D$ and $\mathbf{x},\mathbf{y}$ be tangential basepoints for the pair $(X,D)$ (definition \ref{tangent}).
     \begin{itemize}
        \item We construct an algebra object $\displaystyle\,_\mathbf{x}^{}\mathrm{P}^m_\mathbf{y} (X\setminus D) \in \mathcal{DA}(k,\Lambda)$, the \textbf{motivic pointed path space of $X\setminus D$ based at} $\mathbf{x},\mathbf{y}$, such that the family $\{\+_\mathbf{x}^{}\mathrm{P}^m_\mathbf{y} (X\setminus D)\}_{\mathbf{x},\mathbf{y}}$ has the structure dual to that of a groupoid, in the sense of theorem \ref{bar hopf}.
        \item If $k \subset \mathbb{C}$ and $\Lambda = \mathbb{Q}$, its Betti realization satisfies \begin{equation*}
            \mathrm{H}^0(\mathrm{R}_\mathrm{B}\+_\mathbf{x}^{}\mathrm{P}^m_\mathbf{y}(X\setminus D)) \simeq \mathcal{O}(\pi_1^\mathrm{un}(X(\mathbb{C})\setminus D(\mathbb{C}),\mathbf{x},\mathbf{y}))
        \end{equation*} the prounipotent completion of the fundamental torsor of $X(\mathbb{C})\setminus D(\mathbb{C})$ at $\mathbf{x},\mathbf{y}$ in the sense of definition \ref{enfin}.
        \item If $k$ is of characteristic $0$, and $\Lambda = k$, its de Rham realization is described by $$\mathrm{R}_\mathrm{dR}\+_\mathbf{x}^{}\mathrm{P}^m_\mathbf{y}(X\setminus D) \simeq \+_\mathbf{x}\Baar_\mathbf{y}\mathbf{R}\Gamma\big(X,\Omega_{X/k}(\mathrm{log}D)\big),$$ the bar construction on the logarithmic de Rham complex of $(X,D)$ pointed at $\mathbf{x},\mathbf{y}$ in the sense of definition \ref{larefimportante}.
        \item The periods associated with the comparison isomorphism \begin{equation}\label{2}
            \mathrm{H}^0(\mathrm{R}_\mathrm{B}\+_\mathbf{x}^{}\mathrm{P}^m_\mathbf{y}(X\setminus D))\otimes_\mathbb{Q} \mathbb{C} \simeq \mathrm{H}^0(\mathrm{R}_\mathrm{dR}\+_\mathbf{x}^{}\mathrm{P}^m_\mathbf{y}(X\setminus D)) \otimes_k \mathbb{C}
        \end{equation} are given by regularized iterated integration of forms with logarithmic poles along $D$ in the sense of definition \ref{regularized}.
    \end{itemize}
\end{thm}

\begin{rem}
The isomorphism \eqref{2} is a version with tangential basepoints of K. T. Chen's theorem on iterated integrals \cite[Theorem 2.6.1]{bams/1183539443}. To the best of our knowledge, this general form is new, the case of a curve appeared in \cite{Javier}.
\end{rem}

In the presence of a motivic $t$-structure, (in particular for mixed Tate motives) we deduce the definition of the \textbf{motivic fundamental torsor of $X\setminus D$ at} $\mathbf{x},\mathbf{y}$: $$\pi_1^m(X\setminus D,\mathbf{x},\mathbf{y}):= \mathrm{Spec}\big(\mathrm{H}_\mathrm{mot}^0(\+_\mathbf{x}^{}\mathrm{P}^m_\mathbf{y}(X\setminus D))\big)$$
The functor $\mathrm{H}^0_\mathrm{mot}$ is the $0$-truncation functor associated with the motivic $t$-structure.

\begin{rem}
If we take $(X,D) = (\mathbb{P}^1_\mathbb{Q},\{0,1,\infty\})$ and $\mathbf{x},\mathbf{y}$ any tangential basepoints, this gives a complete definition of $\pi_1^m(\mathbb{P}^1_\mathbb{Q}\setminus \{0,1,\infty\},\mathbf{x},\mathbf{y})$ in the category of mixed Tate motives over $\mathbb{Q}$. When $\mathbf{x},\mathbf{y} \in \{(0,\pm \frac{\partial}{\partial z}\big|_0),(1,\pm\frac{\partial}{\partial z}\big|_1),(\infty,\pm\frac{\partial}{\partial z^{-1}}\big|_\infty)\}$ with $z$ the standard coordinate around $0$, it is natural from the definition to view this object in the category of mixed Tate motives over $\mathbb{Z}$. In general, we hope that our results can be generalized over any base scheme.
\end{rem}

Another definition of the motivic fundamental torsor at tangential basepoints can be found in M. Levine's \cite{levine2005motivictubularneighborhoods} in the case of curves of mixed Tate type via his definition of \textit{motivic tubular neigbourhoods}. We expect this construction to be equivalent to ours. 

At the same level of generality as our theorem \ref{main}, Hain proved in \cite{Haintan} the existence of an ind-Hodge structure on $\mathcal{O}(\pi_1^\mathrm{un}(X(\mathbb{C})\setminus D(\mathbb{C}),\mathbf{x}))$, unconditionally by obtaining it as a limit Hodge structure after degenerating a family of ordinary rational points of $X\setminus D$ into the tangential basepoint $\mathbf{x}$. Although we do not compute the Hodge realization of the motivic pointed path space of $X\setminus D$ at $\mathbf{x}$, we hope that it coincides with Hain's, thus extending in a natural way his construction for ordinary basepoints \cite{hain1}.

\subsection{The case of ordinary basepoints}\label{blablabla}\;\smallskip

Let us recall the classical construction of the motivic fundamental group (we restrict to a single basepoint to simplify). For $U$ a smooth scheme of finite type over a field $k$, and $x \in U(k)$, Z. J. Wojtkowiak \cite{wojtkowiak:hal-01293611} defines the \textit{cosimplicial loop space at $x$}, defined by analogy with classical topology as the pullback of cosimplicial schemes \begin{center}
    \begin{tikzcd}
     \+_x^{}\mathrm{P}^\bullet_x U \arrow[r] \arrow[d] \arrow[dr, phantom, "\usebox\pullback" , very near start, color=black] & U^{\Delta[1]} \arrow[d]\\
        * \arrow[r,"{x,x}"] & U \times U.
    \end{tikzcd}
\end{center}

Deligne and Goncharov define the motivic pointed path space of $U$ at $x$, denoted $\+^{}_x\mathrm{P}^m_xU$, as the motivic complex associated with the simplicial motive $h(\+^{}_x\mathrm{P}^\bullet_xU)$, with $h$ the cohomological motive functor. It is isomorphic to the bar construction of the algebra object $h(U)$ augmented by the point $h(x) : h(U) \xrightarrow{}~h(*)$. When $U$ is of mixed Tate type over a field satisfying Beilinson--Soulé's conjecture, the motivic fundamental group is defined by $$\mathcal{O}(\pi_1^m(U,x)):= \mathrm{H}_\mathrm{mot}^0(\+^{}_x\mathrm{P}^m_xU).$$
Its Betti realization is identified via Beilinson's theorem \cite[Theorem 4.1]{goncharov2001multiplepolylogarithmsmixedtate}\cite[Proposition 3.4]{ASENS_2005_4_38_1_1_0} which shows that when $k \subset \mathbb{C}$ there is a natural isomorphism of Hopf algebras over $\mathbb{Q}$ $$ \mathrm{H}_\mathrm{B}^0(\+_x^{}\mathrm{P}^\bullet_x U)\simeq \mathrm{colim}(\mathbb{Q}[\pi_1(U(\mathbb{C}),x)]/I^{n+1})^\vee = \mathcal{O}(\pi_1^\mathrm{un}(U,x)),$$
where the left cohomology group is computed by taking the cohomology of the double complex induced by the cosimplicial structure.

Its de Rham realization is described by the bar construction \cite[Section 1.2]{esnault2007tatemotivesfundamentalgroup} on the augmented (by evaluation at $x$) commutative differential graded algebra of regular differential forms on $U$ denoted $\Omega_{U/k}$ when $k$ is of characteristic $0$. If we suppose $U$ to be affine to simplify, then there is a canonical isomorphism of Hopf algebras over $k$ $$\mathrm{H}_{\mathrm{dR}}^0(\+_x^{}\mathrm{P}^\bullet_x U) \simeq \mathrm{H}^0\big(\+_x \Baar_x(\Omega_{U/k}(U))\big)$$ where $\+_x \Baar_x(\Omega_{U/k}(U))$ is the differential graded Hopf algebra whose elements are denoted $$[\omega_1|\dots|\omega_k],\; \omega_i \in \Omega_{U/k}(U).$$

The Betti--de Rham comparison theorem for the motivic fundamental group is then obtained via K. T. Chen's theorem \cite[Theorem 2.6.1]{bams/1183539443} which states that there is a natural isomorphism of Hopf algebras over $\mathbb{C}$: $$\mathrm{H}^0\big(\+_x \Baar_x(\Omega_{U/k}(U))\big)\otimes_k \mathbb{C} \simeq \mathrm{colim}(\mathbb{C}[\pi_1(U(\mathbb{C}),x)]/I^{n+1})^\vee.$$ Every class in $\mathrm{H}^0\big(\+_x \Baar_x(\Omega_{U/k}(U))\big)\otimes_k \mathbb{C}$ is represented by elements $[\omega_1|\dots|\omega_k]$ where $\omega_i$ is a $1$-form \cite[Lemma 3.263]{Javier}, and the pairing underlying this theorem is induced by the iterated integration of $1$-forms $$\gamma \otimes [\omega_1|\dots|\omega_k] \mapsto \int_\gamma \omega_1\dots \omega_k,$$ where $\gamma$ is a smooth loop on $U$ based at $x$.

\subsection{The case of tangential basepoints}\;\smallskip

In the case of a tangential basepoint $\mathbf{x}$ on $X$ a smooth scheme of finite type over $k$ endowed with a simple normal crossings divisor $D$, this method fails overall. Indeed, since $\mathbf{x}$ cannot be defined by a geometric morphism $* \rightarrow X\setminus D$, it is impossible a priori to define the cosimplicial path space based at $\mathbf{x}$, or the augmentation on $\Omega_{X(\mathbb{C})\setminus D(\mathbb{C})}$ associated with $\mathbf{x}$ in such a straightforward way.

In the topological setting, Deligne explains in \cite{10.1007/978-1-4613-9649-9_3} how tangential basepoints can be viewed as ordinary points. For instance, for $X$ a smooth curve over $k \subset \mathbb{C}$ endowed with $D := \{x\}$ and $\mathbf{x} := (x,v)$ a tangential basepoint at $x$, let $X'$ be the gluing of $\mathrm{Bl}_x^\mathbb{R}X(\mathbb{C})$ the real-oriented blow-up of $X(\mathbb{C})$ at $x$ to $\mathrm{Bl}_0^\mathbb{R}\mathrm{T}_xX(\mathbb{C})$ the real-oriented blow-up of the tangent space at $0$ along the infinitesimal unit circles. The picture representing this construction is:
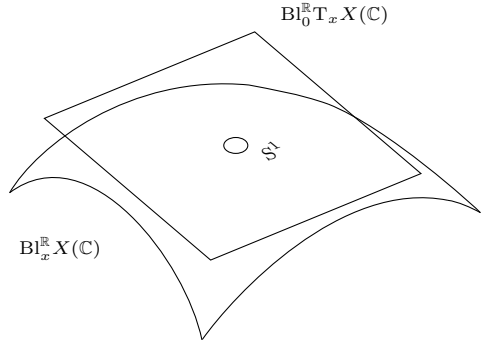
\begin{figure}[!h]
    \centering

    \begin{tikzpicture}[x=0.75pt, y=0.75pt, yscale=-1]
  height of 300

  % The manifold
  \draw (100.04,125.62) node[below right, font=\scriptsize]{$\mathrm{Bl}_x^\mathbb{R}X(\mathbb{C})$};
  \draw (100,108) .. controls (140,78) and (188,143) 
                  .. (196,182) ;
  \draw (196.2,181.8) .. controls (211.8,159.6) and (265.71,101.03) 
                      .. (318.8,111.8) ;
  \draw (318.8,111.8) .. controls (323.67,112.42) and (330.33,115.42) 
                      .. (335,118) ;
  \draw (219.48,53.8) .. controls (230.71,56.06) and (244.01,58.07) 
                              .. (259,63.39) ;

  \draw (259,63.39) .. controls (288.81,75.27) and (332.14,114.93) 
                    .. (335,118) ;
  \draw (100,108) .. controls (113.56,84.56) and (160.04,48.06) 
                  .. (219.48,53.8) ;
  % A point p in the manifold
  \draw (212.9,84.2) circle[x radius = 6, y radius = 4]
                         node[below right, font=\scriptsize, rotate=-337.86,
                           xslant=-0.54]{$S^1$};

  % The tangent space at p
  \draw (222.35, 27.07) -- (305.34, 98.35)
                      -- (200.34, 141.85)
                      -- (117.35, 70.57)
                      -- cycle ;
  \draw (230.51,8) node[below right, 
                           font=\scriptsize]{$\mathrm{Bl}_0^\mathbb{R}\mathrm{T}_xX(\mathbb{C})$};
\end{tikzpicture}
\caption{The topological space $\,X'$} \label{fig:M1}
\end{figure}

The topological space $X'$ has the homotopy type of $X(\mathbb{C})\setminus D(\mathbb{C})$, so Deligne defines the fundamental group of $X(\mathbb{C})\setminus D(\mathbb{C})$ at $\mathbf{x}$ as $$\pi_1(X',v)$$ where the tangential basepoint $\mathbf{x}$ is viewed as the point $$v \in \mathrm{Bl}_0^\mathbb{R}\mathrm{T}_xX(\mathbb{C})\setminus S^1 \simeq \mathrm{T}_xX(\mathbb{C})\setminus \{0\}.$$
This point of view, as well as similar approaches in the étale and crystalline contexts, allows Deligne to construct the expected realizations of the motivic fundamental group at a tangential basepoint in \cite{10.1007/978-1-4613-9649-9_3}.

\subsection{Content of the paper}

\noindent \smallskip

Our construction is inspired by Deligne’s Betti picture (Figure \ref{fig:M1}) for tangential basepoints, and it allows us to apply the cosimplicial loop space/bar construction approach in the setting of tangential basepoints. The main tools we will use are logarithmic geometry and the theory of motives for logarithmic schemes, as developed in the recent work of F. Binda, P. A. Østvær, and D. Park \cite{binda2021triangulatedcategorieslogarithmicmotives,park2023mathbba1homotopytheorylogschemes,binda2024logarithmicmotivichomotopytheory}. Moreover, we adopt the guiding principle that the motivic fundamental torsor should be functorial with respect to the datum $(X,\mathbf{x},\mathbf{y})$. The details of this functoriality are subtle and will be described later (see definition \ref{functo}).

Logarithmic geometry, developed under the influence of J.-M. Fontaine and L. Illusie, and further formulated by K. Kato \cite{kato}, is a variant of classical algebraic geometry designed to formalize the notion of ``algebraic varieties with boundaries". A log scheme is the datum $(\underline{X},\mathcal{M}_X)$ of a scheme $\underline{X}$ and of a sheaf of monoids (for the Zariski topology in our case) $\mathcal{M}_X$ endowed with a morphism of sheaves in monoids $\alpha_X: \mathcal{M}_X \rightarrow \mathcal{O}_X$ inducing an isomorphism $\alpha_X^{-1}(\mathcal{O}_X^\times) \simeq \mathcal{O}_X^\times.$ For instance any scheme $\underline{X}$ is endowed with a trivial log scheme structure, i.e. $\mathcal{M}_X:= \mathcal{O}_X^\times$. A classical family of examples of non-trivial log structures are given by the datum of pairs $(\underline{X},\underline{D})$ where $\underline{X}$ is a regular scheme endowed with $\underline{D}$ a simple normal crossings divisor, so that the log structure is given on $\underline{X}$ by $\mathcal{M}_X:= j_*\mathcal{O}_{X\setminus D}^\times$ where $j: \underline{X}\setminus \underline{D} \hookrightarrow \underline{X}$. We call these log schemes \textit{divisorial log schemes} \cite[\text{III.1.6.1}]{Ogus_2018}.

By the work of Kato and C. Nakayama in \cite{10.2996/kmj/1138044041}, we know that classical cohomology theories are extended to log schemes (log Betti, log de Rham, log étale cohomologies). These extensions compute the same groups as their classical counterparts when applied to a scheme with trivial log structure, and they have the expected property that for $(\underline{X},\underline{D})$ a divisorial log scheme, the log cohomology of $(\underline{X},\underline{D})$ computes the classical cohomology of $\underline{X}\setminus \underline{D}$ (this is true for a more general class of logarithmic schemes: log smooth schemes over a field).

Specifically, for any log scheme $X$ the log de Rham cohomology of $X$ is computed via the log de Rham complex $\Omega_X$ which is nothing but the complex of forms on $\underline{X}$ with logarithmic poles at $\underline{D}$ in the case of a divisorial log scheme $(\underline{X},\underline{D})$. The log Betti cohomology of $X$ is given by the singular cohomology of the Kato--Nakayama space associated with $X$, which is the real-oriented blow-up of $X(\mathbb{C})$ along $D(\mathbb{C})$ in the case of $(\underline{X},\underline{D})$.

These cohomology theories are encompassed within a logarithmic motivic homotopy theory. We recall this theory in detail in section \ref{motivic cat}. In particular, there is a stable $\infty$-category of $\mathbb{A}^1$-invariant log motives over a field $k$ with coefficients $\Lambda$ denoted $\DAlog(k,\Lambda)$ constructed by Park \cite{park2023mathbba1homotopytheorylogschemes}, equivalent to the category $\mathcal{DA}(k,\Lambda)$. In this category, we have the cohomological motive $h(X)$ of $X$ a (fine saturated) log scheme of finite type over $k$. 

In the category $\DAlog(k,\Lambda)$, there is an equivalence \begin{equation}\label{cdhcdhcdh}
    h(0^\mathrm{log}) \simeq h(\mathbb{A}^1,\{0\})
\end{equation} (corollary \ref{tool}) where $0^\mathrm{log}$ is the log scheme obtained by restricting the log structure of $(\mathbb{A}^1,\{0\})$ on~$0$. This phenomenon, when viewed in Betti realization is just a reflection of the topological retractibility of the real-oriented blow-up of $\mathbb{C}$ at $0$ on its infinitesimal unit circle, since the Kato--Nakayama space of $0^\mathrm{log}$ is $S^1$. This can be related to Deligne's trick to view tangential basepoints as ordinary points: for $\underline{X}$ a smooth curve over $k$, $\underline{D}\subset \underline{X}$ a finite set of closed points, and for $x\in \underline{D}(k)$, let us endow $\mathrm{T}_x\underline{X}$ with the log scheme structure $(\mathrm{T}_x\underline{X},\{0\})$. The log scheme $0^\mathrm{log}$ obtained by restriction of the log structure of $(\mathrm{T}_x\underline{X},\{0\})$ to $0$ is canonically isomorphic to the log scheme $x^\mathrm{log}$ obtained by restriction of the log structure of $(\underline{X},\underline{D})$ to $x$. Hence there is a diagram of logarithmic schemes  \begin{center}
    \begin{tikzcd}
    x^\mathrm{log} \arrow[d,hook] \arrow[r,hook] & (\mathrm{T}_x \underline{X},\{0\})\\
    (\underline{X},\underline{D}) & 
    \end{tikzcd}
\end{center}
This diagram is a logarithmic geometric version of the figure \ref{fig:M1}. Take a tangential basepoint $\mathbf{x} := (x,v)$ at $x$. By (\ref{cdhcdhcdh}), there is a sequence of morphisms $$h(\underline{X},\underline{D}) \rightarrow h(x^\mathrm{log}) \simeq h(\mathrm{T}_x \underline{X},\{0\}) \simeq h(\mathrm{T}_x \underline{X} \setminus \{0\}) \xrightarrow{h(v)} h(*) \simeq \Lambda.$$ This defines an augmentation on the algebra object $h(\underline{X},\underline{D})$ associated with $\mathbf{x}$, denoted $h(\mathbf{x})$. Note in particular that the log scheme $(\mathrm{T}_x \underline{X},\{0\})$ has the property that the set of morphisms of log schemes $$* \rightarrow (\mathrm{T}_x \underline{X},\{0\})$$ is in bijection with the set of rational points of $\mathrm{T}_x \underline{X}\setminus \{0\}$ i.e. with the set of tangential basepoint at $x$, since a morphism from the point with trivial log structure to any log scheme lands on the locus of points where the log structure is trivial.

Motivated by this example, in section \ref{virtual} we define, for any divisorial log scheme $(\underline{X},\underline{D})$ and $x \in \underline{X}(k)$, in a functorial manner, a log scheme $\mathrm{N}_x^\mathrm{log}(\underline{X},\underline{D})$ (definition \ref{normalspace}) whose underlying scheme is isomorphic to the normal space of $X$ to $D$ at $x$ and such that the set of its logarithmic rational points (i.e. morphism of log schemes $\mathrm{Spec}(k) \rightarrow \mathrm{N}_x^\mathrm{log}(\underline{X},\underline{D})$) is in bijection with the set of tangential basepoints at $x$.

We deduce from this that the category of tangentially pointed divisorial log schemes is equivalent to the category of \textit{pointed diagrams} that is the category of diagrams of the shape 
\begin{center}
    \begin{tikzcd}
    & * \arrow[d]\\
    x^\mathrm{log} \arrow[d,hook] \arrow[r,hook] & \mathrm{N}_x^\mathrm{log}(\underline{X},\underline{D})\\
    (\underline{X},\underline{D}) & 
    \end{tikzcd}
\end{center}

Once again by a consequence of (\ref{cdhcdhcdh}) (see corollary \ref{invariance}), each pointed diagram defines an augmentation in a functorial way; in section \ref{motivic cat} we show that the cohomological motive functor defines a contravariant functor from the category whose objects are divisorial log schemes over $k$ with the datum of a tangential basepoint $(\underline{X},\underline{D},\mathbf{x})$ to the $\infty$-category $\CAlg(\DAlog(k,\Lambda))^{/\Lambda}$ (definition \ref{mottan}).

\medskip

This allows us to define in section \ref{section The motivic fundamental group}, after complements on the bar construction in an $\infty$-categorical context, via theorem \ref{bar hopf} and definition \ref{pathspace} the motivic pointed path space in a similar way as the case of ordinary basepoints: 
\begin{defi}
    The motivic pointed path space is defined as
    $$\+^{}_\mathbf{x}\mathrm{P}^m_\mathbf{y}(\underline{X}\setminus \underline{D}):= \!_{h(\mathbf{y})}\Baar_{h(\mathbf{x})}(h(\underline{X},\underline{D})):= \Lambda\otimes_{h(\underline{X},\underline{D})}\Lambda \in \CAlg(\mathcal{DA}^\mathrm{log}(k,\Lambda)).$$
\end{defi}

We show in theorem \ref{bar hopf} that the family $\{\+^{}_\mathbf{x}\mathrm{P}^m_\mathbf{y}(\underline{X}\setminus \underline{D})\}$ is endowed with the structure of a homotopy Hopf algebroid. In particular, $\+^{}_\mathbf{x}\mathrm{P}^m_\mathbf{x}(\underline{X}\setminus \underline{D})$ is a Hopf algebra for all augmentation $\mathbf{x}$.

\medskip

In section \ref{realizations}, we check that the Betti and de Rham realizations of the motivic pointed path space yields the expected theorems \ref{bettilogintro} and \ref{bettirhamintro}.  Using the formalism of B. Drew in \cite{drew2018motivichodgemodules}, we build the Betti and de Rham realization functors $\mathrm{R_B},\mathrm{R_{dR}}$ from $\DAlog(k,\mathbb{Q})$ and $\DAlog(k,k)$, and we check the expected property of the Betti fundamental group.
\begin{thm}
[Theorem \ref{thm Betti real pi}] \label{bettilogintro}
Let $(\underline{X},\underline{D})$ be a divisorial log scheme and $\mathbf{x}, \mathbf{y}$ be tangential basepoints of $(\underline{X},\underline{D})$. There is a natural equivalence of algebras
    $$\mathrm{H}^0(\mathrm{R}_\mathrm{B}(\+^{}_\mathbf{x}\mathrm{P}^m_\mathbf{y}(\underline{X}\setminus \underline{D}))) \simeq \mathcal{O}(\pi^\mathrm{un}_1(\underline{X}(\mathbb{C})\setminus \underline{D}(\mathbb{C}), \mathbf{x},\mathbf{y})).$$
\end{thm}

On the de Rham side, we show that the de Rham realization of the motivic pointed path space is $$\+_\mathbf{x}\Baar_\mathbf{y}\textbf{R}\Gamma(\underline{X},\Omega_{(\underline{X},\underline{D})/k}).$$ The augmentations needed to define the bar construction can be viewed via the interpretation of tangential basepoints as \textit{virtual morphisms}, an extension of the available morphisms between log schemes (a notion defined by N. L. Howell in \cite{howell}). This interpretation developed by C. Dupont, E. Panzer and B. Pym in \cite{dupont2024regularizedintegralsmanifoldslog} is compared to the preceding in section \ref{virtual}. Using this comparison, and results from \cite{dupont2024regularizedintegralsmanifoldslog} we compute the Betti--de Rham comparison theorem applied to the motivic pointed path space, and compute the resulting periods:
\begin{thm}[Theorem \ref{end}] \label{bettirhamintro}
    The Betti--de Rham comparison isomorphism applied to $\+^{}_\mathbf{x}\mathrm{P}^m_\mathbf{y}(\underline{X}\setminus \underline{D})$ is induced by the pairing 
        $$
            \mathbb{Q}[\pi_1(\underline{X}(\mathbb{C})\setminus \underline{D}(\mathbb{C}),\mathbf{x},\mathbf{y})] \otimes_\mathbb{Q} \mathrm{H}^0(\+_\mathbf{x}\Baar_\mathbf{y}(\Omega_{(\underline{X},\underline{D})/k}(\underline{X}))) \rightarrow \mathbb{C}$$
given by
           $$ \gamma \otimes [\omega_1|\cdots|\omega_k]  \mapsto  \int_\gamma \omega_1\cdots\omega_k$$
    where the integral sign is the regularized iterated integral of $\omega_1,\dots,\omega_k$ along $\gamma$.
\end{thm}
This generalizes Chen's theorem to forms with log poles. When $X = (\mathbb{P}_\mathbb{Q}^1,\{0,1,\infty\})$, this recovers the regularization of multiple zeta values described in \cite{Javier}.

\subsection{Remarks and relations with other works}\;\smallskip

Virtual logarithmic geometry was defined by N. L. Howell in \cite{howell} and later reformulated by G. Shuklin in \cite{shuklin2024voevodskymotiveassociatedlog}, with already the wish to show that it is integrated in a motivic formalism. In particular, Shuklin is able to define in a functorial way the motive of a \textit{virtual logarithmic scheme} in Voevodsky's category of étale motives. It is possible to define the motivic pointed path space in this context using the correspondence between tangential basepoints and virtual points mentioned earlier. The Betti realization yields the correct object, but we do not know how to compute its de Rham realization. It is unclear if Binda, Park and Østvær's definition of the motive of a log scheme coincide with Shuklin's so we can not prove that such a construction would be equivalent to ours. 

Nevertheless, if a category of log motives with virtual morphisms were to be constructed, it would simplify greatly the constructions in the present article, since the construction for rational points would be transposed directly.

\newpage
\stoptoc
\subsection*{Notations and conventions}

We follow Ogus' reference book \cite{Ogus_2018} regarding most logarithmic geometry notations except for the Kato--Nakayama space of a log-scheme $X$ that we denote by $X^{\mathrm{KN}}$. We only consider log schemes with Zariski log structures. We denote by $\mathbb{A}_{\mathbb{N}^r}:= (\mathbb{A}^r = \mathrm{Spec}(k[t_1,\cdots,t_r], \{t_1\cdots t_r = 0\})$. We also let $*_{\mathbb{N}^r}$ denote the log structure of $\mathbb{A}_{\mathbb{N}^r}$ restricted to $0 \in \mathbb{A}^r$.

To simplify the notations, we will often refer to a divisorial log scheme $(\underline{X},\underline{D})$ simply by ``$X$", so that although we described the motivic pointed path space as ``$\!^{}_\mathbf{x}\mathrm{P}^m_\mathbf{y}(\underline{X}\setminus \underline{D})$" in the introduction, we will denote it by ``$\!^{}_\mathbf{x}\mathrm{P}^m_\mathbf{y}X$". Note that since the motives of $X$ and $\underline{X}\setminus \underline{D}$ coincide, this is consistent.

For all $\infty$-categorical notations, unless specified otherwise, we refer to Lurie's foundational articles \cite{lurie2008higher} and \cite{ha}. Note that for $\mathcal{C}$ an $\infty$-category, and $X$ an object of $\mathcal{C}$, we denote by $\mathcal{C}^{/X}$ (resp. $\mathcal{C}^{//X}$) the category of objects under $X$ (resp. twice under $X$), that is endowed with a morphism (resp. two morphisms) from $X$. We will also denote by curved letters $\mathcal{C}$ the $\infty$-categories and bold letters $\mathbf{C}$ the $1$-categories.

Finally, we will use the following notations, where $S$ is a noetherian scheme (or fs log-scheme when it makes sense) of finite Krull dimension:\\

$\begin{array}{c|l}
    \mathbf{Cat}_1 & \text{the 1-category of small categories}\\
    \Catinf & \text{the $\infty$-category of  $\infty$-categories}\\
    \Lambda & \text{a commutative ring}\\
    \mathbf{Sch}_S & \text{the category of separated noetherian schemes over $S$ of}\\
    & \text{finite Krull dimension}\\
    \mathbf{lSch}_S & \text{the category of separated Noetherian fs log schemes over $S$}\\ 
     & \text{of finite Krull dimension}\\
    \mathbf{Sm}_S & \text{the category of smooth schemes of finite type over $S$}\\
    \mathbf{lSm}_S & \text{the category of log smooth log schemes of finite type over $S$}\\
    \mathcal{M}\mathrm{od}_\Lambda & \text{the $\infty$-category of chain complexes of $\Lambda$-modules}\\
    \Spc & \text{the $\infty$-category of spaces}\\
    \Spt & \text{the $\infty$-category of spectra}\\
    \operatorname{\mathcal{S}h_t(\mathcal{C},\mathcal{V}}) & \text{the $\infty$-category of $t$-sheaves with values in the $\infty$-category $\mathcal{V}$}
\end{array}$\\

\subsection*{Acknowledgements}\;\smallskip

This article was written during my PhD at the IMAG in Montpellier with the support of the ANR Cyclades (Projet ANR-23-CE40-0011). I would like to thank my advisor, Clément Dupont, for his constant guidance and support throughout the writing process, as well as for his careful and precise comments, which allowed me to bring this article to its current form. My warmest thanks go to Swann Tubach for countless discussions on motives and six-functor formalisms, to Nikola Tomic for answering numerous questions on $\infty$-categories, and to Ulysse Mounoud for our wide-ranging discussions on algebraic geometry. I also benefited from conversations with Federico Binda, Damien Calaque, Joost Nuiten and Doosung Park.

\resumetoc
\newpage
\section{Tangential basepoints in log geometry}\label{virtual}
In the first subsection of this section, we describe a logarithmic geometric point of view on tangential basepoints. We define the logarithmic normal space at a rational point of a smooth scheme over a field $k$ endowed with a simple normal crossings divisor, such that its logarithmic rational points are tangential basepoints, and we use it to translate Deligne's approach \ref{fig:M1} in this context via the notion of \textit{pointed diagrams} defined in definition \ref{pointeddia1}. In particular, this point of view helps us describe the functoriality of tangential basepoints. Pointed diagrams will allow us to define motivic tangential basepoints in the next section. In the second subsection we connect this point of view with the approach of Dupont, Panzer and Pym in \cite{dupont2024logarithmicmorphismstangentialbasepoints,dupont2024regularizedintegralsmanifoldslog}. Namely we prove that the category of pointed diagrams is equivalent to the category of \textit{virtually pointed divisorial log schemes}. This will help us compute the realizations of the motivic pointed path space in the last section of this article.

\subsection{The log normal space and pointed diagrams}\label{normal}
\subsubsection{The log normal space at a point}\;\smallskip

We recall the classical definition of tangential basepoints. For a pair $(\underline{X},\underline{D})$ given by a smooth scheme $\underline{X}$ over $k$ endowed with a simple normal crossings divisor $\underline{D} = \bigcup_i \underline{D}_i$ with $\underline{D}_i$ its irreducible components, $\underline{X}$ decomposes into $$\underline{X} = \underline{X}_0 \sqcup \underline{X}_1 \sqcup \cdots \sqcup \underline{X}_j \sqcup \cdots$$ where $\underline{X}_j$ is the locus of intersection of exactly $j$ irreducible components of $\underline{D}$. For $x \in \underline{X}_j$, the normal space to $\underline{D}$ at $x$ is defined as the $k$-vector space \begin{equation}\label{schematic}
    \mathrm{N}_x(\underline{X},\underline{D}):= \mathrm{T}_x\underline{X}/\mathrm{T}_x\underline{X}_j
\end{equation}
where $\mathrm{T}_x\underline{X}, \mathrm{T}_x \underline{X}_j$ are the tangent ($k$-vector) spaces of $\underline{X},\underline{X}_j$ at $x$.

\smallskip

For each $\underline{D}_i$ containing $x$, we let $$\mathrm{N}_x(\underline{X},\underline{D}_i):= \mathrm{T}_x\underline{X}/\mathrm{T}_x\underline{D}_i$$ so that there is a decomposition $$\mathrm{N}_x(\underline{X},\underline{D}) = \bigoplus_i \mathrm{N}_x(\underline{X},\underline{D}_i).$$ 

\smallskip
The definition of tangential basepoints we use is the following from \cite{10.1007/978-1-4613-9649-9_3}:

\begin{defi} \label{tangent}
    A tangential basepoint of $(\underline{X},\underline{D})$ is the datum $(x,v)$ of a rational point $x\in \underline{X}(k)$ and a vector $v \in \mathrm{N}_x(\underline{X},\underline{D})$ such that its projection on $\mathrm{N}_x(\underline{X},\underline{D}_i)$ for each $\underline{D}_i$ containing $x$ is nonvanishing. We denote this tangential basepoint by $\mathbf{x}:= (x,v)$.
\end{defi}

Note that there is a unique tangential basepoint at a rational point $x$ of $\underline{X}\setminus \underline{D}$; it is $(x,0)$. A tangential basepoint at a point lying on a single irreducible component of $\underline{D}$ is the same thing as the datum of a nonzero normal vector to this component at this point. 

\bigskip

We describe a logarithmic translation of this definition. Let $\mathbf{Div}_k$ be the category of divisorial log schemes over $k$. This is nothing but the category of log schemes $(\underline{X},\underline{D})$ with $\underline{X}$ a smooth scheme over $k$ with a simple normal crossings divisor $\underline{D}$, whose morphisms $f: (\underline{X},\underline{D}) \rightarrow (\underline{Y},\underline{E})$ are given by morphisms of $k$-schemes $\underline{f}: \underline{X} \rightarrow \underline{Y}$ that satisfy $\underline{f}^{-1}(\underline{E}) \subset \underline{D}$. To simplify, we will often denote by $X$ the whole datum of $(\underline{X},\underline{D})\in \mathbf{Div}_k.$ Recall that by \cite[A.5.10]{binda2021triangulatedcategorieslogarithmicmotives}, the category $\mathbf{Div}_k$ is equivalent to the category of log smooth log schemes whose underlying scheme is smooth over $k$.

\medskip

Let $X \in \mathbf{Div}_k$, and $x: \mathrm{Spec}(k) \rightarrow X$ a rational point. We denote by $$x^\mathrm{log}:= (\mathrm{Spec}(k), \mathcal{M}_{X}|_x)$$ the log scheme whose underlying scheme is $\mathrm{Spec}(k)$ with the log structure pulled back from $X$ via $x$.

\begin{rem}
Note that the structure monoid $\mathcal{M}_{X}|_x \neq \mathcal{M}_{X,x}$ in general. Indeed, if we let $\beta$ denote the composition $\mathcal{M}_{X,x} \rightarrow \mathcal{O}_{X,x} \rightarrow k$, then by definition $\mathcal{M}_{X}|_x$ is obtained as the pushout of the diagram \begin{center}
\begin{tikzcd}
\beta^{-1}(k^\times) \simeq \mathcal{O}_{X,x}^\times \arrow[r] \arrow[d] & k^\times\\
\mathcal{M}_{X,x} &
\end{tikzcd}
\end{center}
\end{rem}

It is proved in \cite[Theorem 3.5]{dupont2024logarithmicmorphismstangentialbasepoints} that for $X:= (\underline{X},\underline{D}) \in \mathbf{Div}_k$ and $x\in \underline{X}(k)$, there is a bijection between the set of tangential basepoints at $x$ and the set of group morphisms $$(\mathcal{M}_{X}|_x)^\mathrm{gp} \rightarrow k^\times$$ which restrict to the identity on $k^\times$. IIn coordinates, if we let $x_1,\dots, x_j$ be local coordinates of $\underline{D}$ at $x$, then there is a (non-canonical) isomorphism $$\mathcal{M}_{X}|_x \simeq k^\times x_1^\mathbb{N}\cdots x_j^\mathbb{N}$$ and the datum of a morphism $(\mathcal{M}_{X}|_x)^\mathrm{gp} \rightarrow k^\times$ which restricts to the identity on $k^\times$ is equivalent to the data of $\lambda_1,\dots, \lambda_j  \in k^\times,$ where $\lambda_i$ is the image of $x_i$. This datum is associated to the tangential basepoint $$\bigg(x,\lambda_1 \frac{\partial}{\partial x_1}\bigg|_x + \dots + \lambda_j \frac{\partial}{\partial x_j}\bigg|_x\bigg).$$ 

This monoid is related to the schematic normal space in the following way:

\begin{lem}\label{plus}
    There is a canonical isomorphism of schemes over $k$
    $$\mathrm{N}_x(\underline{X},\underline{D}) \xrightarrow{\sim} \mathrm{Spec}(k[\mathcal{M}_{X}|_x]\otimes_{k[k^\times]} k)$$ where $\mathrm{N}_x(\underline{X},\underline{D})$ is seen as the affine space associated with (\ref{schematic}).
\end{lem}
\begin{proof}
The normal space is defined by $$\mathrm{N}_x(\underline{X},\underline{D}):= \mathrm{Spec(Sym}((\mathrm{T}_xX/\mathrm{T}_xX_j)^\vee)),$$ and the $k$-algebra $k[\mathcal{M}_{X}|_x]\otimes_{k[k^\times]} k$ is identified with $$k[x_1,\dots, x_n]$$ where $x_1,\dots, x_j$ is any choice of local coordinates of $D$ at $x$.
There is a morphism of schemes over $k$ $$\mathrm{N}_x(\underline{X},\underline{D}) \rightarrow \underline{\mathrm{N}}^\mathrm{log}_x(\underline{X},\underline{D})$$ defined by the morphism of $k$-algebras $$k[\mathcal{M}_{X}|_x]\otimes_{k[k^\times]} k \rightarrow \mathrm{Sym}((\mathrm{T}_xX/\mathrm{T}_xX_j)^\vee$$ which associates to $x_i$ the element $dx_i|_x \in (\mathrm{T}_xX/\mathrm{T}_xX_j)^\vee.$ Note that this construction does not depend on the choice of coordinates we made since another choice $\{x'_i\}$ would satisfy $x'_i = \lambda_i x_{i}$ in $\mathcal{M}_{X}|_x$ with $\lambda_i \in k^\times$.
The normal space satisfies $$\mathrm{N}_x(\underline{X},\underline{D}) \simeq \mathrm{Spec}(k[dx_1|_x,\dots ,dx_j|_x])\simeq \mathbb{A}^j_k,$$ so that we indeed defined an isomorphism.
\end{proof}

This lemma motivates the following definition:

\begin{defi}\label{normalspace}
For $(\underline{X},\underline{D}) \in \mathbf{Div}_k$ and $x \in \underline{X}(k)$, we call \textbf{log normal space to} $D$ \textbf{at} $x$ the log scheme $\mathrm{N}^\mathrm{log}_x(\underline{X},\underline{D}) \in \mathbf{lSch}_k$ whose underlying scheme is defined by $$\underline{\mathrm{N}}^\mathrm{log}_x(\underline{X},\underline{D}):= \mathrm{Spec}(k[\mathcal{M}_{X}|_x]\otimes_{k[k^\times]} k),$$ and whose log structure is induced by the morphism of monoids $$\mathcal{M}_{X}|_x \rightarrow k[\mathcal{M}_{X}|_x] \rightarrow k[\mathcal{M}_{X}|_x]\otimes_{k[k^\times]} k.$$
\end{defi}

\begin{rem}
From the description given in the proof of lemma \ref{plus}, we see that the log structure induced on $\mathrm{N}_x(\underline{X},\underline{D})$ is a divisorial log structure given by $$dx_1|_x\dots dx_j|_x = 0.$$ The log scheme obtained is simply $\mathbb{A}_{\mathbb{N}^j}.$ Hence, note that $\mathrm{N}^\mathrm{log}_x(\underline{X},\underline{D}) \in \mathbf{Div}_k$.
\end{rem}

\begin{exmp}
\label{example Normal log}
    \begin{itemize}
        \item For $(\underline{X},\underline{D}) = \mathbb{A}_\mathbb{N}$, $x = 0 \in \mathbb{A}^1$, then $$\mathrm{N}^\mathrm{log}_x(\underline{X},\underline{D}) = \mathbb{A}_\mathbb{N}$$
        \item For $(\underline{X},\underline{D}) = \mathbb{A}_{\mathbb{N}^2}$, $x = (1,0)$, then $$\mathrm{N}^\mathrm{log}_x(\underline{X},\underline{D}) = \mathbb{A}_\mathbb{N}.$$
        \item For any $(\underline{X},\underline{D})$, if $x\in X\setminus D$, $$\mathrm{N}^\mathrm{log}_x(\underline{X},\underline{D}) = *.$$
    \end{itemize}
\end{exmp}

The constructions $x^\mathrm{log}$ and $\mathrm{N}^\mathrm{log}_x(\underline{X},\underline{D})$ are related by a functorial (strict) closed immersion of log schemes over $k$ $$x^\mathrm{log} \hookrightarrow \mathrm{N}^\mathrm{log}_x(\underline{X},\underline{D})$$ defined as follows: its underlying morphism of schemes is given by the morphism of $k$-algebras $$k[\mathcal{M}_{X}|_x] \otimes_{k[k^\times]} k \rightarrow k$$ coming from the morphism of monoids defining the log structure on $x^\mathrm{log}$ $$\mathcal{M}_{X}|_x \rightarrow k,$$ and the morphism of log structures is described by \begin{center}
    \begin{tikzcd}
    \mathcal{M}_{X}|_x \arrow[r,equal] \arrow[d] & \mathcal{M}_{X}|_x \arrow[d]\\
    k[\mathcal{M}_{X}|_x]\otimes_{k[k^\times]} k \arrow[r] & k
    \end{tikzcd}
\end{center}
which shows that the morphism of log schemes thus defined is strict.

\bigskip

The log scheme $\mathrm{N}^\mathrm{log}_x(\underline{X},\underline{D})$ is the correct logarithmic version of the schematic normal space to $D$ at $x$ associated with the normal space (\ref{schematic}). In this direction, we show that the set of morphisms from the log scheme $\mathrm{Spec}(k)$ with trivial log structure to the log normal space to $D$ at $x$ are in canonical bijection with the set of tangential basepoints of $(\underline{X},\underline{D})$ at $x$.

\begin{lem}\label{ola}
    For $(\underline{X},\underline{D}) \in \mathbf{Div}_k$ and $x \in \underline{X}(k)$, there is a bijection between the set of tangential basepoints at $x$ and the set of logarithmic rational points of $\mathrm{N}^\mathrm{log}_x(\underline{X},\underline{D})$, that is the set of morphisms of logarithmic schemes $$\mathrm{Spec}(k) \rightarrow \mathrm{N}^\mathrm{log}_x(\underline{X},\underline{D}).$$
\end{lem}

\begin{proof}
For $X$ a fs log scheme, we denote by $\partial X$ the boundary of $X$, i.e. the closed subscheme of $X$ such that for $x\in X,$ $x\in \partial X$ if and only if $\mathcal{M}_{X,x} \neq \mathcal{O}_{X,x}^\times$. It is the locus of points at which the log structure is non-trivial. For $Y$ a log scheme with trivial log structure, any morphism $Y \rightarrow X$ to a fs log scheme lands in the open locus of points with trivial log structure $X\setminus \partial X$. Thus there is a bijection between the set of logarithmic rational point of $\mathrm{N}^\mathrm{log}_x(\underline{X},\underline{D})$ and the set of rational points of $\underline{\mathrm{N}}^\mathrm{log}_x(\underline{X},\underline{D})\setminus \partial\underline{\mathrm{N}}^\mathrm{log}_x(\underline{X},\underline{D}).$ 

To identify the open locus with trivial log strucure of the log normal space to $D$ at $x$, note that it is the open subset on which any function $m \in \mathcal{M}_{X}|_x$ is inverted. Hence it is simply described as the spectrum of the localization of $k[\mathcal{M}_{X}|_x]\otimes_{k[k^\times]} k$ at these elements (or at a set of generators of $\mathcal{M}_{X}|_x$). Therefore the open locus is $$\mathrm{Spec}(k[(\mathcal{M}_{X}|_x)^\mathrm{gp}]\otimes_{k[k^\times]} k).$$ A rational point of this scheme is given in a unique way by a group morphism $$(\mathcal{M}_{X}|_x)^\mathrm{gp} \rightarrow k^\times$$ which restricts to the identity on $k^\times$. 
By the discussion above definition \ref{normalspace}, this is what we wanted to prove.
\end{proof}

\begin{rem}

If we let $\mathrm{N}_x(\underline{X},\underline{D}_i):= \mathrm{Spec}(\mathrm{Sym}(\mathrm{T}_x\underline{X}/\mathrm{T}_x\underline{D}_i)^\vee)$ there is a decomposition $$\mathrm{N}_x(\underline{X},\underline{D}) = \prod_i \mathrm{N}_x(\underline{X},\underline{D}_i),$$ so that the log structure is obtained via the product $$\prod_i (\mathrm{N}_x(\underline{X},\underline{D}_i),0).$$ Since the locus with non-trivial log structure of $(\mathrm{N}_x(\underline{X},\underline{D}_i),0)$ is reduced to the point $\{0\}$, this point of view explains in another way that a logarithmic rational point of $\mathrm{N}^\mathrm{log}_x(\underline{X},\underline{D})$ is the same thing as a tangential basepoint at $x$: it is the datum of a vector $v \in \mathrm{N}_x(\underline{X},\underline{D})$ whose projection on each $\mathrm{N}_x(\underline{X},\underline{D}_i)$ does not vanish.
\end{rem}

\subsubsection{The functoriality of tangential basepoints}\;\smallskip

Since the construction $\mathrm{N}^\mathrm{log}_x(\underline{X},\underline{D})$ is functorial, it allows us to describe the functoriality for tangential basepoints. 

\begin{defi}\label{functo}
A morphism $f : (\underline{X},\underline{D}) \rightarrow (\underline{Y},\underline{E})$ of divisorial log schemes over $k$ satisfies $f(\mathbf{x}) = \mathbf{y}$, with $\mathbf{x}$ and $\mathbf{y}$ tangential basepoints on $ (\underline{X},\underline{D})$ and $(\underline{Y},\underline{E})$ respectively if $f(x) = y$, and if the induced morphism of log schemes $\mathrm{N}_x^\mathrm{log}(\underline{X},\underline{D}) \rightarrow \mathrm{N}_y^\mathrm{log}(\underline{Y},\underline{E})$ satisfies that the following diagram commutes: \begin{center}
    \begin{tikzcd}
    \mathrm{Spec}(k) \arrow[d,"\mathbf{x}"] \arrow[rd, "\mathbf{y}"] & \\
    \mathrm{N}^\mathrm{log}_x(\underline{X},\underline{D}) \arrow[r] & \mathrm{N}^\mathrm{log}_y(\underline{Y},\underline{E})
    \end{tikzcd}
\end{center}
where the morphisms from $\mathrm{Spec}(k)$ are the log rational points associated with $\mathbf{x}$ and $\mathbf{y}$ by the lemma \ref{ola}.
\end{defi}

\smallskip

In a simple case, suppose $(\underline{X},\underline{D}) = (\mathbb{A}^n,x_1\dots x_j= 0),$ $(\underline{Y},\underline{E}) = (\mathbb{A}^m,y_1\dots y_k = 0)$ and let $f : (\underline{X},\underline{D}) \rightarrow (\underline{Y},\underline{E})$ be a morphism of divisorial log schemes such that $f(0) = 0$. There are monomials in the coordinates $x_1,\dots, x_j$ denoted by $M_1, \dots, M_k \in K[x_1,\dots, x_j]$ such that $$f^*y_i = M_i(x_1,\dots, x_j).$$ If we let $$\mathbf{0}:= \bigg(0,\lambda_1 \frac{\partial}{\partial x_1}\bigg|_0 + \dots + \lambda_j \frac{\partial}{\partial x_j}\bigg|_0\bigg)$$ be a tangential basepoint of $(\underline{X},\underline{D})$ at $0$, then it is easy to see that the tangential basepoint defined by the composition \begin{center}
    \begin{tikzcd}
    \mathrm{Spec}(k) \arrow[d,"\mathbf{x}"] & \\
    \mathrm{N}^\mathrm{log}_x(\underline{X},\underline{D}) \arrow[r] & \mathrm{N}^\mathrm{log}_y(\underline{Y},\underline{E})
    \end{tikzcd}
\end{center}
is $$f(\mathbf{0}):=\bigg(0,M_1(\lambda_1,\dots, \lambda_j) \frac{\partial}{\partial y_1}\bigg|_{0} + \dots + M_k(\lambda_1,\dots,\lambda_j) \frac{\partial}{\partial y_k}\bigg|_{0}\bigg).$$

\begin{rem}
For a normal vector $t$ of $D$ at $x$, the normal vector associated with $t$ by this functoriality by $f: (\underline{X},\underline{D}) \rightarrow (\underline{Y},\underline{E})$ is not in general $df_x(t)$. For instance if $f: (\mathbb{A}^1,0) \rightarrow (\mathbb{A}^1,0)$ is defined by $f(x) = x^n$ with $ n \geq 2$, then $df_0(t) = 0$. Instead, if $t = \lambda \partial_x |_0$, then $f(\mathbf{0}) = (0, \lambda^n \partial_x |_0)$.
\end{rem}

\begin{rem}
This functoriality seems to be aligned with, and was partly inspired by the point of view which consists in viewing tangential basepoints via morphisms $\mathrm{Spec}(k((x))) \rightarrow X$, see for instance \cite[Lemma 10.1]{petrov2024universalitygaloisactionfundamental}. We will see in the next subsection another motivation for this functoriality.
\end{rem}

\subsubsection{The category of pointed diagrams}\;\smallskip

Let us now define the category of pointed diagrams which will be our framework when we define motivic tangential basepoints in subsection \ref{motivic tangential basepoints}. Note that the diagram of log schemes over $k$ \begin{equation}\label{diag}
        \begin{tikzcd}
            x^\mathrm{log} \arrow[d,hook]  \arrow[r,hook] & \mathrm{N}^\mathrm{log}_x(\underline{X},\underline{D})\\
            X & 
        \end{tikzcd}
    \end{equation}
    is functorial with respect to morphisms of divisorial log schemes over $k$.

\begin{defi}\label{pointeddia1}
We call the \textbf{category of pointed diagrams} the category whose objects are diagrams of log schemes over $k$ of the form
\begin{center}
        \begin{tikzcd}
        & \mathrm{Spec}(k) \arrow[d,"\mathbf{x}"]\\
            x^\mathrm{log} \arrow[d,hook]  \arrow[r,hook] & \mathrm{N}^\mathrm{log}_x(\underline{X},\underline{D})\\
            X & 
        \end{tikzcd}
    \end{center}
where $(\underline{X},\underline{D}) \in \mathbf{Div}_k$, and $\mathbf{x}: \mathrm{Spec}(k) \rightarrow \mathrm{N}^\mathrm{log}_x(\underline{X},\underline{D})$ is a tangential basepoint at a rational point $x$. To simplify, we denote such diagrams by $(\underline{X},\underline{D},\mathbf{x})$.

A morphism of pointed diagrams $(\underline{X},\underline{D},\mathbf{x}) \rightarrow (\underline{Y},\underline{E},\mathbf{y})$ is a morphism of diagrams of log schemes over $k$ given by a morphism of divisorial log schemes $f: (\underline{X},\underline{D}) \rightarrow (\underline{Y},\underline{E})$ sending $x$ to $y$, the induced morphism of log schemes $x^\mathrm{log} \rightarrow y^\mathrm{log}$, the induced morphism $\mathrm{N}^\mathrm{log}_x(\underline{X},\underline{D}) \rightarrow \mathrm{N}^\mathrm{log}_y(\underline{Y},\underline{E})$ and the identity on $\mathrm{Spec}(k)$. In particular the diagram \begin{center}
    \begin{tikzcd}
    \mathrm{Spec}(k) \arrow[d,"\mathbf{x}"] \arrow[rd, "\mathbf{y}"] & \\
    \mathrm{N}^\mathrm{log}_x(\underline{X},\underline{D}) \arrow[r] & \mathrm{N}^\mathrm{log}_y(\underline{Y},\underline{E})
    \end{tikzcd}
\end{center}
must commute, which just means that $f(\mathbf{x}) = \mathbf{y}$.
\end{defi}

Hence, it is clear that the category of pointed diagrams is equivalent to the category of tangentially pointed smooth schemes over $k$ with simple normal crossings divisor, that is the category whose objects are the datum of a divisorial log scheme $(\underline{X},\underline{D})$ over $k$ and a tangential basepoint for $(\underline{X},\underline{D})$, and whose morphisms are morphisms of divisorial log schemes over $k$ which are compatible with the fixed tangential basepoints in the sense described above. The reason why we focus on the category of pointed diagrams lies in the fact that pointed diagrams are a higher dimensional and functorial analogue of Deligne's gluing procedure to view tangential basepoints as points. Indeed if we apply the Kato--Nakayama space functor to the diagram \ref{diag}
with $X:= (\underline{X},\underline{D})$ a smooth curve over $k$, we get
\begin{center}
    \begin{tikzcd}
    S^1 \arrow[r,hook] \arrow[d,hook] & \mathrm{Bl}^\mathbb{R}_0\mathbb{C}\\
    \mathrm{Bl}_D^\mathbb{R}\underline{X}(\mathbb{C}) &
    \end{tikzcd}
\end{center}
where, as in the introduction of this article, $\mathrm{Bl}^\mathbb{R}$ stands for the real-oriented blow-up. Deligne's picture is simply the topological pushout of this diagram.

\subsection{Virtual points}\label{pointed}\;\smallskip

Now we connect our approach of tangential basepoints via logarithmic geometry to the point of view developed in \cite{dupont2024logarithmicmorphismstangentialbasepoints}, based on the notion of \textit{virtual morphisms of log schemes} first defined in \cite{howell} and \cite{shuklin2024voevodskymotiveassociatedlog}. This class of maps between log schemes enlarge the class of classical morphisms of logarithmic schemes and introduces for instance, maps from a log scheme with trivial log structure which do not land in the locus of points with trivial log structure (see the proof of lemma \ref{ola}). This, as we will see is the case for virtual points.

\begin{defi}\label{virtualmo}{\cite[Definition 2.6]{dupont2024logarithmicmorphismstangentialbasepoints}}
    Let $X,Y \in \mathbf{lSch}_k$, a \textbf{virtual morphism} $X \rightarrow Y$ is a pair $(\underline{\phi},\phi^\flat)$ consisting of a morphism of schemes $\underline{\phi}: \underline{X} \rightarrow \underline{Y}$, and a morphism $\phi^\flat: \underline{\phi}^{-1}\mathcal{M}^\mathrm{gp}_Y \rightarrow \mathcal{M}^\mathrm{gp}_X$ of sheaves of groups, making the following diagram commute: \begin{center}
        \begin{tikzcd}
        \underline{\phi}^{-1}\mathcal{O}^{\times}_Y \arrow[d,hook] \arrow[r,"\underline{\phi}^\sharp"] & \mathcal{O}^{\times}_X \arrow[d,hook]\\
            \underline{\phi}^{-1}\mathcal{M}^\mathrm{gp}_Y \arrow[r,"\phi^\flat"] & \mathcal{M}^\mathrm{gp}_X
            \end{tikzcd}
    \end{center}
    We denote by $\+^\mathbf{v}\mathbf{lSch}_k$ the category of fs log schemes with virtual morphisms between them.
\end{defi}

A morphism of log schemes induces in a unique way a virtual morphism of log schemes so that $\mathbf{lSch}_k$ is a subcategory of $\+^\mathbf{v}\mathbf{lSch}_k$. In order to distinguish morphisms in $\mathbf{lSch}_k$ from those in $\+^\mathbf{v}\mathbf{lSch}_k$ we call the former \emph{ordinary morphisms of log schemes}.

\begin{defi}
We let $\mathbf{Div}_{k}^{\mathbf{v}*/}$ be the category of virtually pointed divisorial log schemes over $k$. It is the category of divisorial log schemes $(\underline{X},\underline{D})$ with a virtual morphism $\mathrm{Spec}(k) \rightarrow (\underline{X},\underline{D})$ and in which the morphisms are ordinary morphisms of log schemes commuting with the fixed virtual points. 
\end{defi}

We relate the category of pointed diagrams and the category of virtually pointed divisorial log schemes over $k$ using \cite[Theorem 3.5]{dupont2024logarithmicmorphismstangentialbasepoints}:

\begin{thm}\label{clement}\cite[Theorem 3.5]{dupont2024logarithmicmorphismstangentialbasepoints}
    For $X \in \mathbf{Div}_k$, there is a bijection between tangential basepoints of $X$ and virtual points of $X$.
\end{thm}

Hence, tangential basepoints can be described by both virtual morphisms on divisorial log schemes over $k$ and pointed diagrams. Nevertheless, pointed diagrams offer the advantage of being constructed using ordinary morphisms only, which is what we will need to work in a log motivic category. We have the functorial version of the theorem \ref{clement}:

\begin{prop}\label{pointeddia}
   The category $\mathbf{Div}^{\mathbf{v}*/}_k$ is equivalent to the category of pointed diagrams.
\end{prop}
\begin{proof}
The proof of \cite[Theorem 3.5]{dupont2024logarithmicmorphismstangentialbasepoints} is based on the fact that the datum of a pointed diagram $(\underline{X},\underline{D},\mathbf{x})$, and of a virtual morphisms $\mathrm{Spec}(k) \rightarrow (\underline{X},\underline{D})$ landing on $x$ are both equivalent in a natural way to morphisms of groups $\mathcal{M}_{X}|_x^\mathrm{gp} \rightarrow k^\times$ that restrict to the identity on $k^\times.$
\end{proof}

Based on the previous proposition, we will simply denote by the thick letter $\mathbf{x}$ the equivalent datum of:\begin{itemize}
    \item a tangential basepoint at $x\in \underline{X}(k),$
    \item a pointed diagram $(\underline{X},\underline{D},\mathbf{x}),$
    \item a virtual morphism $\mathrm{Spec}(k) \rightarrow X.$
    \end{itemize}

There is a last (partial) representation of tangential basepoints when $k \subset \mathbb{C}$. If $\mathbf{x}: * \rightarrow X$ is a virtual point, then, since the Kato--Nakayama space construction is functorial with respect to virtual morphisms \cite[Proposition 4.2]{dupont2024logarithmicmorphismstangentialbasepoints}, this datum determines a point on $X^\mathrm{KN}$.
More precisely since $X^\mathrm{KN}$ is the real-oriented blow-up of $\underline{X}$ along $\underline{D}$, this determines a point of $S^j$ where $j$ is the number of irreducible components of $\underline{D}$ at $x$. Each coordinate of $S^j$ determines a direction normal to one of the irreducible component of $\underline{D}$.

\begin{exmp}\label{yes}
If $X = \mathbb{A}_\mathbb{N}$ with coordinate $t$, if the image of a virtual point $\mathbf{x}$ is the origin, then it is determined by a constant $\lambda \in k^\times.$ The associated tangential basepoint is $(x,\lambda\frac{\partial}{\partial t}\big|_x)$, and its associated point on $\mathbb{A}_\mathbb{N}^\mathrm{KN}$ is given by the coordinate $\frac{\lambda}{|\lambda|}$ on the infinitesimal unit circle at $0$
\end{exmp}

In the previous example, we note that two virtual morphisms determined by two constants $\lambda, a\lambda$ where $a$ is a strictly positive constant have the same associated point on the Kato--Nakayama space. Hence a point on the Kato--Nakayama is not as such equivalent to the datum of a tangential basepoint. This is actually related to a classical fact in the theory of the fundamental group. The Betti side of the Hodge structure on the prounipotent completion of the fundamental group only sees the direction of the tangential basepoints, since a change of parametrisation of any path based at a tangential basepoint can change the tangent vector by multiplication by an arbitrary strictly positive constant. In section \ref{deRhamreal2}, we will recall from \cite[Section 8.4]{dupont2024regularizedintegralsmanifoldslog} how the structure of \textbf{manifold with log corners} of the Kato--Nakayama space allows to complete this approach so that the missing datum is encoded in the manifold with log corners version of virtual morphisms.

\newpage
\section{Log motives and motivic tangential basepoints}\label{motivic cat}
In this section, after a brief reminder of the constructions of different $\infty$-categories of log motives as stated in \cite{binda2021triangulatedcategorieslogarithmicmotives}, and the properties of their $\mathbb{A}^1$-invariant versions, introduced in \cite{park2023mathbba1homotopytheorylogschemes}, we explain how \textit{log-cdh-descent} allows us to define \textit{motivic tangential basepoints}. These are the augmentations of the cohomological motive of a divisorial log scheme associated with a tangential basepoint. This will lie at the heart of our definition of the motivic pointed path space in the next section. In the first two subsections we simply recall what is defined in \cite{park2023mathbba1homotopytheorylogschemes,park2024logmotivicexceptionaldirect,park2024motivicsixfunctorformalismlog}. In the last subsection \ref{motivic tangential basepoints}, we define the motivic tangential basepoints and we spend some time to ensure that our construction is functorial as expected. In this section, $k$ is a fixed field.

\subsection{Log motivic categories}\label{log motives}

\begin{defi}\cite[Definition 2.2.4]{park2023mathbba1homotopytheorylogschemes}
Let $\mathbf{C}$ be a (1-)category with $\mathcal{P}$ a class of morphisms which contains all isomorphisms and which is stable by pullback. A $\mathcal{P}$-premotivic $\infty$-category over $\mathbf{C}$ is a functor. $$\mathcal{T}: \mathrm{N}(\mathbf{C})^\mathrm{op} \rightarrow \CAlg(\PrL)$$ such that 
\begin{itemize}
    \item For all $f\in \mathcal{P}$, $f^*:= \mathcal{T}(f)$ has a left adjoint $f_\sharp$,    \item for every cartesian square \begin{center}
        \begin{tikzcd}
            X' \arrow[r, "g'"] \arrow[d,"f'"] & X \arrow[d,"f"]\\
            S' \arrow[r,"g"] & S
         \end{tikzcd}
        
    \end{center}
    the morphism $$f'_\sharp g'^* \rightarrow g^*f_\sharp$$ is an equivalence,
    \item for every morphism $f \in \mathcal{P}$, the $\mathcal{P}$-projection $$f_\sharp((-)\otimes f^*(-)) \rightarrow f_\sharp(-) \otimes (-)$$ is an equivalence.
\end{itemize}
\end{defi}

\begin{defi}
    For $\mathcal{T}$ any $\mathcal{P}$-premotivic $\infty$-category over $\mathbf{C}$, and for $S\in \mathbf{C}$ we let $$M_S:\mathcal{P}/S \rightarrow \mathcal{T}(S)$$ be the \textbf{homological motive} functor
    where we denote by $\mathcal{P}/S$ the full subcategory of $\mathbf{C}/S$ with objects the morphisms $X \rightarrow S \in \mathcal{P}$. We also denote by
$$ h_S: (\mathbf{C}/S)^\mathrm{op} \rightarrow \mathcal{T}(S)$$ the \textbf{cohomological motive} functor. They are defined for $f: X\rightarrow S \in \mathcal{P}$  by $$M_S(X):= f_\sharp f^* \mathbf{1}_S$$
and for $f \in \mathbf{C}/S$, if we let $f_*$ denote a right adjoint of $f^*$ (which exists since $f^*$ is colimit preserving by hypothesis) by $$ h_S(X):= f_*f^*\mathbf{1}_S,$$ where $\mathbf{1}_S \in \mathcal{T}(S)$ is the unit for the monoidal structure.
\end{defi}

These two constructions are dual to each other in the sense that \begin{center}
    $h_S(X) \simeq \underline{\operatorname{Map}}(M_S(X),\mathbf{1}_S).$
\end{center}

\begin{rem}
 When $S = \mathrm{Spec}(k)$, for $k$ a field, and the context is clear we will drop the $S$ in the notations $h_S, M_S$.
\end{rem}

Let us give the definitions leading to the premotivic $\infty$-categories with which we will work.

\begin{defi} \begin{itemize}
\item Let $\textit{sét}$ be the \textbf{strict étale topology} on $\mathbf{lSm}_S$, that is the Grothendieck topology generated by strict étale morphisms.
    \item Let $\textit{sNis}$ be the \textbf{strict Nisnevich topology} on $\mathbf{lSm}_S$, that is the Grothendieck topology generated by covers of the form $(p: Y\rightarrow X, j: X' \rightarrow X)$ where $p$ and $j$ appear in a cartesian square 
    \begin{center}
        \begin{tikzcd}
            Y' \arrow[r, hook] \arrow[dr, phantom, "\usebox\pullback" , very near start, color=black] \arrow[d] & Y \arrow[d,"p"]\\
            X' \arrow[r,hook,"j"] & X
        \end{tikzcd}
    \end{center}where $j$ is an open immersion and $p$ is a strict étale morphism inducing an isomorphism $$p^{-1}(X\setminus X') \simeq X\setminus X'.$$
    \item Let $\mathrm{dNis}$ (resp. $\text{dét}$) be the \textbf{dividing Nisnevich topology} (resp. the \textbf{dividing étale topology}) on $\mathbf{lSm}_S$, that is the Grothendieck topology generated by $\mathrm{sNis}$ (resp. $\text{sét}$) and dividing covers, that is log étale, proper monomorphisms $p: Y \rightarrow X$.
    \item Let $\mathrm{ver}_S$ be the Grothendieck topology on $\mathbf{lSm}_S$ generated by open immersions $U\rightarrow V$ such that $U\setminus \partial_S U \rightarrow V\setminus \partial_S V$ is an isomorphism, where $\partial_S$ defines the vertical boundary over $S$, i.e. the points with an open neighborhood on which the structure morphism to $S$ admits a Zariski local vertical chart \cite[Definition I.4.3.1]{Ogus_2018}.
\end{itemize}
\end{defi}
The strict étale and Nisnevich topology are natural log versions of the classical étale and Nisnevich topology on $\mathbf{Sm}_S$ and the dividing topology encompasses log blow-ups. When $S$ has trivial log structure, the $\mathrm{ver}_S$ topology is generated by open immersions $U\setminus \partial U \rightarrow U$ so that sheaves in this topology are canonically equivalent to presheaves on $\mathbf{Sm}_S$.

\begin{defi}\cite[Remark 4.12]{robalo2013noncommutativemotivesiuniversal}
    Let $\mathcal{C}$ be a presentable symmetric monoidal $\infty$-category, and $X$ an object of $\mathcal{C}$. The stabilisation $\mathrm{Stab}_X(\mathcal{C})$ of $\mathcal{C}$ by $X$ is defined as $$\mathrm{lim}(\cdots \mathcal{C} \xrightarrow{\Omega_X} \mathcal{C} \xrightarrow{\Omega_X} \mathcal{C}) \simeq \colim(\mathcal{C} \xrightarrow{\Sigma_X} \mathcal{C} \xrightarrow{\Sigma_X} \mathcal{C} \cdots),$$ where $\Omega_X:= \underline{\operatorname{Map}}(X,-)$ is the internal mapping space, where $\Sigma_X = - \otimes X$ and where the limit is taken in $\Catinf$ and the colimit in $\PrL$. 
\end{defi}

\medskip

The recent work of Park culminated in \cite{park2024motivicsixfunctorformalismlog}, where he managed to extend the classical six-functor formalism derived from the $\mathbf{Sm}$-premotivic $\infty$-category $$\mathbf{Sch}^\mathrm{op} \rightarrow \mathcal{DA}(-,\Lambda):= \mathrm{Stab}_{\mathbb{G}_m/1}(\mathrm{L}_{\mathbb{A}^1}\mathcal{S}\mathrm{h}_\mathrm{Nis}(\mathbf{Sm}_-,\mathcal{M}\mathrm{od}_\Lambda)),$$ where $\mathbb{G}_m/1 := \mathrm{cofib}(M_S(*) \xrightarrow{1} M_S(\mathbb{G}_m))$ and $\mathrm{L}_{\mathbb{A}^1}$ denotes the $\mathbb{A}^1$ localization, into a full six-functor formalism derived from a $\mathbf{lSm}$-premotivic $\infty$-category over $\mathbf{lSch}$. This structure will be the main framework for our constructions.

\begin{defi}\cite[Definition 2.5.5]{park2023mathbba1homotopytheorylogschemes}
    The $\mathbf{lSm}$-premotivic $\infty$-category of motivic sheaves on $\mathbf{lSch}$ with coefficients in $\Lambda$ is defined by (the $\otimes$ highlights the symmetric monoidal structure) $$\mathcal{DA}^\mathrm{log}(-,\Lambda)^\otimes:= \mathrm{Stab}_{\mathbb{G}_m^\mathrm{log}/1}(\mathrm{L}_{\mathbb{A}^1}\mathrm{L}_\mathrm{ver}\mathcal{S}\mathrm{h}_\mathrm{dNis}(\mathbf{lSm}_-,\mathcal{M}\mathrm{od}_\Lambda))^\otimes \in \CAlg(\PrL)$$ and the $\mathbf{lSm}$-premotivic stable homotopy $\infty$-category of motivic sheaves by $$\mathcal{SH}^\mathrm{log}(-)^\otimes:= \mathrm{Stab}_{\mathbb{G}_m^\mathrm{log}/1}(\mathrm{L}_{\mathbb{A}^1}\mathrm{L}_{\mathrm{ver}}\Sh_{\mathrm{dNis}}(\mathbf{lSm}_{-},\Spt))^\wedge \in \CAlg(\PrL)$$ Where $\mathbb{G}_m^\mathrm{log}/1:= \cofib(M_S(*) \xrightarrow{1} M_S(\mathbb{G}_m^\mathrm{log})) \in  \mathrm{L}_{\mathbb{A}^1}\mathrm{L}_{\mathrm{ver}}\Sh_{\mathrm{dNis}}(\mathbf{lSm}_-,\Mod_\Lambda))$, with $\mathbb{G}^\mathrm{log}_m:= (\mathbb{P}^1,\{0,\infty\})$.
\end{defi}

\medskip

From \cite[Construction 2.5.4]{park2023mathbba1homotopytheorylogschemes}, when $S \in \mathbf{Sch}$, there is a natural adjunction \begin{center}
\begin{tikzcd}
\lambda: \mathbf{Sm}_S \arrow[r,shift left=.5ex]
&
\mathbf{lSm}_S \arrow[l,shift left=.5ex]: \omega
\end{tikzcd}
\end{center} defined by $\lambda(X) = (X,\mathcal{O}^\times_X)$ for $X\in \mathbf{Sm}_S$ and $\omega(X) = X \setminus \partial X$ for $X\in \mathbf{lSm}_S$. It is easy to see that since the induced functor $$\lambda^*: \mathcal{P}\mathrm{sh}(\mathbf{lSm_S},\Mod_\Lambda) \rightarrow \mathcal{P}\mathrm{sh}(\mathbf{Sm}_S,Mod_\lambda)$$
satisfies $\lambda^*y_X \simeq y_{\omega(X)}$ where $y$ denotes the Yoneda functor, so that it induces to a functor 
   $$ \begin{array}{ccccc}
       \lambda^* & :  & \DAlog(S,\Lambda) & \rightarrow & \mathcal{DA}(S,\Lambda) \\
       &  & M_S(X) & \mapsto & M_S(X\setminus \partial X)
    \end{array}$$

\begin{prop}\label{log motives motives}\cite[Construction 2.5.4]{park2023mathbba1homotopytheorylogschemes}
The functor $\lambda^*$ is an equivalence $$\lambda^*: \DAlog(S,\Lambda) \xrightarrow{\sim} \mathcal{DA}(S,\Lambda).$$
    
\end{prop}

\begin{rem}
    This is straightforward once we not that $\omega$ trivializes dividing and ver covers, and sends $\mathbb{G}_m^\mathrm{log}$ on $\mathbb{G}_m$.
\end{rem}

We keep separate notations $\DAlog$ and $\mathcal{DA}$ even on log schemes with trivial log structures to distinguish the realization functors on each of these categories, see section 6. We will also refer to the étale version of the previous construction in the case of rational coefficients: 

\begin{defi}
    For $S\in \mathbf{lSch}$, the $\infty$-category of étale motivic sheaves on $\mathbf{lSm}_S$ with $\mathbb{Q}$ coefficients is defined by $$\mathcal{DA}_\text{ét}^\mathrm{log}(S,\mathbb{Q})^\otimes:= \mathrm{Stab}_{\mathbb{G}_m^\mathrm{log}/1}(\mathrm{L}_{\mathbb{A}^1}\mathrm{L}_{\mathrm{ver}}\Sh_\text{dét}(\mathbf{lSm}_-,\Mod_\mathbb{Q}))^\otimes \in \CAlg(\PrL).$$
\end{defi}
    When $S\in \mathbf{Sch}$, the same proof as the proposition \ref{log motives motives} gives the same result in this context, so that there is a commutative diagram in $\CAlg(\PrL)$ \begin{center}
    \begin{tikzcd}
        \DAlog(S,\mathbb{Q}) \arrow[d,"\sim"] \arrow[r,"\mathrm{L}_\text{ét}"] &\mathcal{DA}_\text{ét}^\mathrm{log}(S,\mathbb{Q})\arrow[d,"\sim"]\\
        \mathcal{DA}(S,\mathbb{Q}) \arrow[r,"\mathrm{L}_\text{ét}"] &  \mathcal{DA}_\text{ét}(S,\mathbb{Q})
    \end{tikzcd}
        
    \end{center}
    with the functor induced by the étale sheafification on the horizontal arrows.

\begin{rem}
Since the category of étale log motives is not defined in Park's work and for the sake of generality, we will work mainly in $\DAlog(S,\Lambda)$, and go to $\DAlog_\text{ét}(S,\mathbb{Q})$ only to define the motivic fundamental groupoid in the mixed Tate case (see subsection \ref{fundamentalgroup}). One should keep in mind that the definitions we make should work exactly the same in $\DAlog_\text{ét}(S,\mathbb{Q})$, so that the rest of the ideas developed in this article work when $\DAlog$ is replaced with $\DAlog_\text{ét}$. Note that we took $\mathbb{Q}$ coefficients to avoid introducing hypersheaves, hence to simplify the exposition. With more general coefficients, the right definition would involve étale hypersheaves rather than étale sheaves, and these definitions do not coincide in general. They do however with $\mathbb{Q}$ coefficients.
\end{rem}

\subsection{log-cdh-descent for log motives}\label{logcdh}\;\smallskip

For $S \in \mathbf{lSch}$, in \cite{park2024logmotivicexceptionaldirect}, Park shows that $\DAlog(S,\Lambda)$ satisfies a log version of cdh-descent:
\begin{defi}
    A commutative square of log-schemes over $S$ \begin{center}
        \begin{tikzcd}
    Z' \arrow[r,"i'"] \arrow[d,"p'"] & X' \arrow[d,"p"]\\
    Z \arrow[r,"i"] & X
        \end{tikzcd}
    \end{center}
    is a log-cdh distinguished square if it is a cartesian square, $i$ is a strict closed immersion, $p$ is proper, and the induced morphism $p^{-1}(X\setminus i(Z)) \rightarrow X \setminus i(Z)$ is an isomorphism.
\end{defi}

\begin{prop}\label{cdh}\cite[Corollary 3.6.8]{park2024logmotivicexceptionaldirect} 
For any log-cdh distinguished square as in definition 7.1, the commutative square of functors \begin{center}
    \begin{tikzcd}
        id \arrow[r] \arrow[d] & p_*p^* \arrow[d]\\
        i_*i^* \arrow[r] & q_*q^*
    \end{tikzcd}
\end{center}
with $q:= i \circ p'$, is cartesian.
\end{prop}
From this proposition, if we denote by $f: X \rightarrow S$, $g: X' \rightarrow S$, $h: Z' \rightarrow S$ the structure morphisms, knowing that the pushforward and the pullback along a log smooth morphism are right adjoints, and hence preserves cartesian squares, the following square is cartesian: \begin{center}
    \begin{tikzcd}
        f_*f^* \arrow[r] \arrow[d] &(f\circ i)_*(f\circ i)^* \arrow[d]\\
         g_*g^*  \arrow[r] & h_*h^*
    \end{tikzcd}
\end{center}
Evaluating at $\mathbf{1}_S$, we thus have that the square \begin{center}
    \begin{tikzcd}
        h_S(X) \arrow[r] \arrow[d] &   h_S(Z) \arrow[d]\\
        h_S(X')\arrow[r] & h_S(Z')
    \end{tikzcd}
\end{center}
is cartesian. This implies the following in the case where $S = \mathrm{Spec}(k)$:

\begin{cor}\label{tool}
     The inclusion of the boundary of $\mathbb{A}_\mathbb{N}$ induces an equivalence in $\mathcal{DA}^\mathrm{log}(k,\Lambda)$ $$h(\mathbb{A}_\mathbb{N})\xrightarrow{\sim} h(*_\mathbb{N}).$$
\end{cor}
\begin{proof}
We apply the previous discussion to the following log-cdh square: \begin{center}
        \begin{tikzcd}
    *_\mathbb{N} \arrow[r,"i'"] \arrow[d,"p'"] & \mathbb{A}_\mathbb{N} \arrow[d,"p"]\\
    * \arrow[r,"i"] & \mathbb{A}^1
        \end{tikzcd}
    \end{center}
    with $i$ the inclusion of the zero point, and $p$ the morphism forgetting the log structure.
Therefore we see that the square  \begin{center}
    \begin{tikzcd}
        h(\mathbb{A}^1) \arrow[r] \arrow[d] &   h(*) \arrow[d]\\
        h(\mathbb{A}_\mathbb{N})\arrow[r] & h(*_\mathbb{N})
    \end{tikzcd}
\end{center}
is cartesian (and cocartesian by stability of $\mathcal{DA}^\mathrm{log}(k,\Lambda)$). Since the top arrow is an equivalence, the lemma is proved.
\end{proof}

It is not clear in general that the functor $h$ satisfies Künneth's isomorphism. It obviously does on log smooth log schemes over $k$ since $h$ is the dual of $M$, but we will need it for the log point $*_{\mathbb{N}^{n+m}} \simeq *_{\mathbb{N}^{n}} \times_k *_{\mathbb{N}^{m}}$ which is non-log-smooth on $k$. This is the content of the following lemma.

\begin{lem}
    Let $n,m$ be two nonnegative integers. The natural morphism $$h(*_{\mathbb{N}^n}) \otimes h(*_{\mathbb{N}^m}) \rightarrow h(*_{\mathbb{N}^{n+m}})$$ is an isomorphism.
\end{lem}
\begin{proof}
Let us consider the following cartesian square:
  \begin{center}
    \begin{tikzcd}
       *_{\mathbb{N}^{n+m}} \arrow[r,"p_2"] \arrow[d,"p_1"] &  *_{\mathbb{N}^{m}} \arrow[d,"f_2"]\\
        *_{\mathbb{N}^{n}} \arrow[r,"f_1"] & *.
    \end{tikzcd}
\end{center}
The classical proof of Künneth's isomorphism in the context of a six functors formalism (see \cite[Theorem 1.2]{Scholze6}) relies on the projection formulae for the pushforwards $$(p_1)_*((-) \otimes p_1^* (-)) \simeq (p_1)_*(-) \otimes (-) $$
and $$(f_1)_*((-) \otimes f_1^* (-)) \simeq (f_1)_*(-) \otimes (-) $$
and on the base change isomorphism $$f_1^*(f_2)_* \simeq (p_1)_*p_2^*.$$
Since all of the maps in the cartesian square are proper (i.e. their underlying morphism of schemes is proper), the three previous isomorphisms holds by \cite[Corollary 3.7.6, Theorem 3.7.8, Theorem 1.2.2]{park2024logmotivicexceptionaldirect}.
\end{proof}

\begin{cor}\label{invariance}
    Let $(\underline{X},\underline{D}) \in \mathbf{Div}_k$, and $x\in \underline{X}(k)$, the natural closed immersion $x^\mathrm{log} \rightarrow \mathrm{N}_x^\mathrm{log}(\underline{X},\underline{D})$ induces a natural equivalence $$h(\mathrm{N}_x^\mathrm{log}(\underline{X},\underline{D})) \xrightarrow{\sim} h(x^\mathrm{log})$$
\end{cor}
\begin{proof}
By the description of the log normal space given after the definition \ref{normalspace}, we know that $\mathrm{N}_x^\mathrm{log}(\underline{X},\underline{D}) \simeq \mathbb{A}_{\mathbb{N}^j}$ and $x^\mathrm{log} \simeq *_{\mathbb{N}^j}$ so that this corollary reduces to the corollary \ref{tool} thanks to the previous lemma.
\end{proof}

\subsection{Motivic tangential basepoints} \label{motivic tangential basepoints}\;\smallskip

We can now build a functor $$\mathbf{Div}_{k}^{\mathbf{v}*/,op} \rightarrow \CAlg(\DAlog(k,\Lambda))^{/\Lambda}$$ allowing 
an extension of the classical construction of the augmentation $h(X) \rightarrow \Lambda$ associated to a point $x\in \underline{X}(k)$ to the case of tangential basepoints and thus enhancing the functor induced by the cohomological motive functor at the level of arrows $$\operatorname{Arr}(\mathbf{Div}_{k}^\mathrm{op}) \rightarrow \operatorname{Arr}(\CAlg(\DAlog(k,\Lambda))) := \CAlg(\DAlog(k,\Lambda))^{\Delta[1]}.$$

\begin{defi}
    We let \begin{itemize}
        \item $\Delta[1]$ be the category \begin{center}
    \begin{tikzcd}
        * \arrow[r] & *
    \end{tikzcd}
\end{center}
\item $\widetilde{\Delta[2]}$ be the category \begin{center}
    \begin{tikzcd}
        * \arrow[r] & * \arrow[r] & * \arrow[l, "\sim"] \arrow[r] & * 
    \end{tikzcd}\end{center}
     \item $\Delta[2]$ be the category \begin{center}
    \begin{tikzcd}
        * \arrow[r] & * \arrow[r] & *
    \end{tikzcd}\end{center}
     \item $\overset{\leftarrow}{\Delta[3]}$ be the category \begin{center}
    \begin{tikzcd}
        * \arrow[r] & *  & * \arrow[r] \arrow[l] & *.
    \end{tikzcd}\end{center}
    \end{itemize}
\end{defi}

\begin{rem}\label{sequence}
    There is a sequence of functors $$\Delta[1] \rightarrow \Delta[2] \xleftarrow[]{\sim} \widetilde{\Delta[2]}$$ where the first morphism is the $1$-face map and the second contracts the middle morphism.
\end{rem} 
The category $\mathbf{Div}_{k}^{\mathbf{v}*/}$, viewed as the category of pointed diagrams (proposition \ref{pointeddia}), is a subcategory of $\mathbf{Div}_k^{\overset{\leftarrow}{\Delta[3]}}$, so that there is a natural functor induced by the cohomological motive functor $$\mathbf{Div}^{\mathbf{v}*/,op}_{k} \rightarrow \mathbf{Div}_k^{\overset{\leftarrow}{\Delta[3]},op} \rightarrow \CAlg(\DAlog(k,\Lambda))^{\overset{\leftarrow}{\Delta[3]}}.$$ Since the image of this composition lands in $\DAlog(k,\Lambda)^{\widetilde{\Delta[2]}}$, there is an induced morphism $$\mathbf{Div}_{k}^{\mathbf{v}*/,op} \rightarrow \CAlg(\DAlog(k,\Lambda))^{\widetilde{\Delta[2]}}$$ by the following lemma:
\begin{lem}
    The functor $$\CAlg(\DAlog(k,\Lambda))^{\widetilde{\Delta[2]}} \rightarrow \CAlg(\DAlog(k,\Lambda))^{\overset{\leftarrow}{\Delta[3]}}$$ induced by the inclusion $\overset{\leftarrow}{\Delta[3]} \rightarrow \widetilde{\Delta[2]}$ is fully faithful.
\end{lem}
\begin{proof}
    The functor $\overset{\leftarrow}{\Delta[3]} \rightarrow \widetilde{\Delta[2]}$ is the $1$-categorical localization of the category $\overset{\leftarrow}{\Delta[3]}$ by the middle morphism. Note that this is given by the pushout of the following diagram \begin{center}
        \begin{tikzcd}
            \Delta[1] \arrow[d] \arrow[r,hook] & \overset{\leftarrow}{\Delta[3]}\\
            \Delta[1]^{\simeq}
        \end{tikzcd}
    \end{center}
    where the horizontal morphism selects the middle morphism and $\Delta[1]^\simeq$ is the groupoid associated to $\Delta[1]$.

     $\widetilde{\Delta[2]}$ is still the pushout of this diagram in $\Catinf$ after application of the nerve functor. This is not straightfoward since the nerve functor does not preserve pushouts in general, see \cite[Figure 1.1]{Hackney} for a counterexample, but it can be checked by computing the pushout explicitely.
     
     More precisely, if we denote by $\mathcal{C}$ the pushout of the diagram above, there is a canonical arrow $\mathcal{C} \rightarrow \widetilde{\Delta[2]}$. Note that the objects of $\mathcal{C}$ are given by the pushouts of objects of the categories involved in the diagram, as can be checked by applying the homotopy category functor $ho(-)$ to it (note that as a left adjoint to the nerve functor it preserves colimits). Now, since the functor $\displaystyle\Delta[1] \rightarrow \overset{\leftarrow}{\Delta[3]}$ is fully faithful, we can describe all the mapping spaces of $\mathcal{C}$ using \cite[Theorem 0.1]{haine2025fullyfaithfulfunctorspushouts} which are all straightforwardly shown to be weakly equivalent to the corresponding mapping space in $\widetilde{\Delta[2]}$ (they are actually homeomorphic since they are reduced to a point, or empty). Hence the functor $\mathcal{C}\rightarrow \widetilde{\Delta[2]}$ is an equivalence of $\infty$-categories.

    This shows that, as $\infty$-categories, $\widetilde{\Delta[2]}$ is the Dwyer--Kan localization of $\overset{\leftarrow}{\Delta[3]}$ along the middle morphism, see \cite[Proposition 2.4.8]{Land}. The lemma now comes from the universal property of localization \cite[Definition 2.4.2]{Land}.
\end{proof}

Hence, the cohomological motive functor defines a functor after composing with the sequence from remark \ref{sequence}:
\begin{eqnarray*}
\mathbf{Div}_{k}^{\mathbf{v}*/,op} & \rightarrow & \CAlg(\DAlog(k,\Lambda))^{\widetilde{\Delta[2]}}\\
& \simeq & \CAlg(\DAlog(k,\Lambda))^{\Delta[2]} \\
& \rightarrow & \CAlg(\DAlog(k,\Lambda))^{\Delta[1]} \simeq \mathrm{Arr}(\CAlg(\DAlog(k,\Lambda))).
\end{eqnarray*}
Since $\DAlog(k,\Lambda)^{/\Lambda}$ is the faithful subcategory of $\mathrm{Arr}(\CAlg(\DAlog(k,\Lambda)))$ consisting of arrows whose codomain is $\Lambda$, and since this category is obtained as a fiber product of $\CAlg(\DAlog(k,\Lambda))^{\Delta[1]}$ and the point category on $\CAlg(\DAlog(k,\Lambda))$, this defines a functor $$\mathbf{Div}^{\mathbf{v}*/,op}_{k} \rightarrow \CAlg(\DAlog(k,\Lambda))^{/\Lambda}.$$

\begin{rem}
Note that this construction is in fact defined only up to homotopy. More precisely, it depends on a choice of a quasi-inverse $\Delta[2] \rightarrow \Tilde{\Delta[2]}$ to the contraction of the middle morphism. Two different choices leads to homotopy equivalent constructions.
\end{rem}

\begin{defi}\label{mottan}
   We call \textbf{motivic tangential augmentation} (or \textbf{motivic tangential basepoint}) functor the functor $$\begin{array}{ccc}
        \mathbf{Div}_{k}^{\mathbf{v}*/,op} & \rightarrow & \CAlg(\DAlog(k,\Lambda))^{/\Lambda}\\
        (\underline{X},\underline{D},\mathbf{x})  & \mapsto & \begin{tikzcd}
    h(x^\mathrm{log}) \arrow[r] & h(\mathrm{N}_x^\mathrm{log}(\underline{X},\underline{D})) \arrow[l,"\sim"] \arrow[r,] & h(*) = \Lambda\\
    h(X) \arrow[u] & &
\end{tikzcd}
    \end{array}$$
     given by the composition of the sequence of arrows on the right. We denote by $h(\mathbf{x}): h(X) \rightarrow \Lambda$ the resulting arrow.
\end{defi}
\begin{rem}
    In a more direct fashion, from a diagram \begin{center}\begin{tikzcd}
    h(x^\mathrm{log}) & h(\mathrm{N}_x^\mathrm{log}(\underline{X},\underline{D})) \arrow[l,"\sim"] \arrow[r,] & h(*) = \Lambda\\
    h(X) \arrow[u] & &
\end{tikzcd}\end{center} by a choice of homotopy inverse of the central equivalence, we get a diagram \begin{center}\begin{tikzcd}
    h(x^\mathrm{log}) \arrow[r,"\sim"] & h(\mathrm{N}_x^\mathrm{log}(\underline{X},\underline{D})) \arrow[r] & h(*) = \Lambda\\
    h(X) \arrow[u] & &
\end{tikzcd}\end{center} that we compose to get an augmentation $h(X) \rightarrow \Lambda$.
\end{rem}

This construction actually works in the same way in any "category of log motives" in which the cohomological motives of the standard log point, and of $\mathbb{A}_\mathbb{N}$ are equivalent, and in which Künneth's isomorphism holds for the cohomological motives of products of the log point.

\begin{rem}\label{usual}
Note that if $\mathbf{x}$ is just given by a rational point $x \in \underline{X}(k)\setminus \underline{D}(k)$, then the motivic augmentation defined by $\mathbf{x}$ is the classical augmentation $h(x)$. Indeed the pointed diagram associated with $\mathbf{x}$ is \begin{center}
        \begin{tikzcd}
        & * \arrow[d,equal]\\
            * \arrow[d,"x"]  \arrow[r,equal] & *\\
            X & 
        \end{tikzcd}
    \end{center}
    so that the morphism \begin{center}\begin{tikzcd}
    & \Lambda\\
    \Lambda \arrow[r,equal] & \Lambda \arrow[u,equal] \\
    h(X) \arrow[u,"h(x)"] &
\end{tikzcd}\end{center} is the expected one.
\end{rem}

\newpage
\section{The motivic fundamental group}\label{section The motivic fundamental group}
In this section, we describe in detail properties of the bar construction in an $\infty$-categorical context, an we use it to construct the motivic pointed path space based at tangential basepoints. In particular, we highlight the resulting homotopy Hopf algebroid structure of the datum of the motivic pointed path spaces based at tangential basepoints. We compare this construction with Wojtkowiak's approach \cite{wojtkowiak:hal-01293611} recalled in the introduction, which will be useful to compute the realizations of the motivic pointed path space in the next section. We end this section with the mixed Tate case which leads to the definition of the motivic fundamental torsor at tangential basepoints. The key points of this section are the theorem \ref{bar hopf}, definition \ref{pathspace} and proposition \ref{path space facile}. In this section, $k$ is still a fixed field.

\subsection{The bar construction} \label{The bar construction} \;\smallskip

Let us first summarize facts about the bar construction. In this section, nothing is new except the description of the homotopy Hopf algebroid structure in an $\infty$-categorical context. The main reference here is \cite[Sections 4.4.2, 5.2.2]{ha}.\smallskip

Let $\mathcal{C}$ be a cocomplete symmetric monoidal $\infty$-category such that the tensor product $\otimes: \mathcal{C} \times \mathcal{C} \rightarrow \mathcal{C}$ commutes with colimits in each variable.

\begin{prop}\cite[Construction 4.4.2.7, Example 4.4.2.11]{ha}  For $A \in \CAlg(\mathcal{C})$, $M,N\in \mathrm{Mod}_A(\mathcal{C})$, there is a canonical simplicial diagram $\Baar_A(M,N)_\bullet$ in $\mathrm{Mod}_A(\mathcal{C})$ informally given by: \begin{center}

    \begin{tikzcd}
        \cdots \;\; M\otimes A \otimes A \otimes N \ar[r,yshift = 8] \ar[r,yshift = 4, <-] \ar[r,yshift = 0] \arrow[r,yshift = -4, <-] \arrow[r,yshift = -8] & M\otimes A \otimes N \ar[r,yshift = 4] \ar[r,yshift = 0, <-] \arrow[r,yshift = -4] & M\otimes N
    \end{tikzcd}

\end{center}
such that its colimit is the relative tensor product of $M$ and $N$ over $A$: $$M\otimes_A N:= \Baar_A(M,N):= \colim_\Delta \Baar_A(M,N)_\bullet.$$
It is the bar construction of $A$ on $M$ and $N$.
    
\end{prop}

Degreewise $\Baar_A(M,N)_n \simeq M \otimes A^n \otimes N$ with face maps induced by the product \begin{center}
    \begin{tikzcd}
        A\otimes A \arrow[r] & A
    \end{tikzcd}
\end{center}
and the actions on $M$ and $N$ \begin{center}
    \begin{tikzcd}
        M \otimes A \arrow[r] & M, & A\otimes N \arrow[r] & A.
    \end{tikzcd}
\end{center} The degeneracies are given by the unit map \begin{center}
    \begin{tikzcd}
        \mathbf{1} \arrow[r] &  A.
    \end{tikzcd}
\end{center}

For $A \in \CAlg(\mathcal{C})$, the relative tensor product over $A$ endows the category $\mathrm{Mod}_A(\mathcal{C})$ with a symmetric monoidal $\infty$-category structure \cite[Theorem 4.5.2.1]{lurie2008higher}.

\bigskip

We will apply the relative tensor product in a more restrictive case, that is when $A$ is a biaugmented commutative algebra object in $\mathcal{C}$ with augmentations (i.e. algebra morphisms) $\begin{tikzcd} A \ar[r,yshift = 2, "x"] \ar[r,yshift = -2, "y"'] & \mathbf{1}
\end{tikzcd}$. We denote by $\mathbf{1}_x, \mathbf{1}_y$ the unit element of $\mathcal{C}^\otimes$ viewed as commutative $A$-algebras through the maps $x$ and $y$. 

\begin{defi}
    We call bar construction of $A$ based at $x$ and $y$ the object \begin{center}
    $\;_x\Baar_y(A):= \Baar_A(\mathbf{1}_y,\mathbf{1}_x) \in \mathrm{Mod}_A(\mathcal{C}).$
\end{center}
\end{defi}

\begin{rem}
The choice for the order of the entries is a convention explained in remark \ref{ordre}.
\end{rem}

We recall that this object comes with a natural structure of commutative $A$-algebra object in $\mathcal{C}$. Indeed, by applying \cite[Proposition 3.2.4.7, Corollary 3.2.4.8]{ha} to $\mathrm{Mod}_A(\mathcal{C})$, we see that $\CAlg_A(\mathcal{C}) = \CAlg(\mathrm{Mod}_A(\mathcal{C}))$ has coproducts and that these coproducts coincide with iterated relative tensor products. 

In particular, \begin{center}
    $\;_x\Baar_y(A) = \mathbf{1}_y\otimes_A \mathbf{1}_x = \mathbf{1}_y\sqcup_A \mathbf{1}_x \in \CAlg_A(\mathcal{C}).$
\end{center}

This way, it is easy to see that the bar construction defines a functor \begin{center}
    $\Baar:\CAlg(\mathcal{C})^{//\mathbf{1}} \rightarrow \CAlg(\mathcal{C})$
\end{center}

Note that the $A$-algebra structure is defined by the map $\+_x\nabla_y$ described by \begin{eqnarray*}
    \;_x\Baar_y(A) \otimes_A \,_x\Baar_y(A) & = & (\mathbf{1}_y \otimes_A \mathbf{1}_x) \otimes_A (\mathbf{1}_y \otimes_A \mathbf{1}_x)\\
    & \simeq & (\mathbf{1}_y \otimes_A \mathbf{1}_y) \otimes_A (\mathbf{1}_x \otimes_A \mathbf{1}_x)\\
    & \rightarrow & \mathbf{1}_y \otimes_A \mathbf{1}_x\\
    & = & \!_x\Baar_y(A)
\end{eqnarray*}
and by the natural map \begin{center}
    $\!_x\eta_y:\mathbf{1} \rightarrow A \rightarrow \mathbf{1}_y \sqcup_A \mathbf{1}_x$.
\end{center}

\medskip

Let us recall the definition of a homotopy Hopf algebra in an $\infty$-category in the sense of \cite[Definition 4.0.1, Corollary 4.1.4]{item_d5d0d09024574dc1a6af56fd0ed1dc4e}, which is a bialgebra with a homotopy antipode. We remind the reader that the category of bialgebra objects in $\mathcal{C}$ is defined to be $\mathrm{Alg}(\CAlg(\mathcal{C})^\mathrm{op}).$

\begin{defi}
A homotopy Hopf algebra object $H \in \mathcal{C}$ is a bialgebra object endowed with a map $$s: H \rightarrow H$$ such that $H$ is an Hopf algebra in the homotopy category of $\mathcal{C}$ with antipode $s$.
\end{defi}

We now prove a general structure result for the bar construction:
\begin{thm}\label{groupoid}
    The datum of the objects and morphisms $\{\+_x\Baar_y(A)\}_{x,y}$ for every augmentations $x,y$ of $A$ is enhanced into a homotopy Hopf algebroid, i.e. \begin{itemize}
        \item For all augmentation $x$, $\+_x\Baar_x(A)$ has the structure of a homotopy Hopf algebra in $\mathcal{C}.$
        \item For three augmentations $x,y,z$, there is a deconcatenation map $$ \+^{}_x\Delta_z^y \;: \;_x\Baar_z(A) \rightarrow  \;_y\Baar_z(A) \otimes \,_x\Baar_y(A)$$ which is the coproduct on $\;_x\Baar_x(A)$ when $x = y = z$.
        \item For all augmentations $x,y$, there is an involution $$\!_xs_y: \;_x\Baar_y(A) \rightarrow \;_y\Baar_x(A)$$ which is the homotopy antipode of $\;_x\Baar_x(A)$ when $x = y$.
        \item If we let $\epsilon_x: \+_x\Baar_x(A) \rightarrow \mathbf{1}$ denote the counit of $\+_x\Baar_x(A)$ for all augmentation $x$, then the morphisms $$\{\+_x\nabla_y, \+_x\eta_y, \+^{}_x\Delta_z^y, \+_xs_y, \epsilon_x\}_{x,y,z}$$ satisfy the homotopy counit, coinverse and coassociativity axioms which respectively states that in the homotopy category $\mathrm{ho}(\mathcal{C})$ \begin{itemize}
            \item $(\epsilon_y \otimes \mathrm{Id})\circ \+^{}_x\Delta_y^y = \mathrm{Id}$ and $(\mathrm{Id} \otimes \epsilon_x)\circ \+^{}_x\Delta_y^x = \mathrm{Id},$
            \item $\+_x\nabla_y\circ(\+_ys_x \otimes \mathrm{Id})\circ\+^{}_x\Delta_y^x = \+_y\eta_x\circ\epsilon_x$ and $\+_y\nabla_x\circ(\mathrm{Id} \otimes \+_xs_y )\circ\+^{}_x\Delta_y^x = \+_x\eta_y \circ \epsilon_x$
            \item for $x,y,z,t$ four augmentations of $A$, $(\+^{}_y\Delta_t^z \otimes \mathrm{Id})\circ\+^{}_x\Delta_t^y = (\mathrm{Id} \otimes \+^{}_x\Delta_z^y)\circ\+^{}_x\Delta_t^z$.
        \end{itemize}
    \end{itemize} 
\end{thm}

This theorem shows that the data of the pointed bar constructions of $A$, taken over all augmentations of $A$, carries a structure dual to that of a groupoid in the homotopy category of $\mathcal{C}$. Here, by a groupoid structure we mean the one carried for instance by the system of pointed path spaces of a topological space $X$: namely, the collection of all pointed path spaces of $X$, together with the associative composition of paths, the existence of inverse paths, and the unit path at each point of $X$. It is in this sense that we will interpret and work with the notion of a homotopy Hopf algebroid in the following sections. In particular, in subsection \ref{fundamentalgroup} we will define the motivic analogue of the fundamental groupoid in the presence of a motivic $t$-structure.

Note that we have defined the structure of a homotopy Hopf algebroid only in a weak homotopical sense. In an $\infty$-categorical framework, one would usually formalize such a structure by verifying all possible higher coherences, so our notion is not as rigorous as it could be in that setting. However, lifting this definition to an $\infty$-categorical level lies beyond the scope of this article, as we only require the structure to be well defined in the homotopy category of~$\mathcal{C}$.

The rest of this subsection is dedicated to the proof of theorem \ref{groupoid}. This proof consists in the lemmas \ref{premlem} \ref{deco}, \ref{medlem}, \ref{bar hopf} and \ref{lastlem}. 

Let us first describe the $\infty$-categorical structure of bialgebra that completes the structure of commutative algebra object of the pointed bar construction:

\begin{lem}\label{premlem}
     Let $A \in \CAlg(\mathcal{C})$. For each augmentation $x: A \rightarrow \mathbf{1}$, $\+_x\Baar_x(A)$ has a natural associative bialgebra structure.
\end{lem}
The complete and careful proof of this fact is \cite[Theorem 5.2.2.17]{ha}, where the bar construction is enhanced into a functor \begin{center}
    $\Baar: \CAlg^{/\mathbf{1}}(\mathcal{C})^\mathrm{op} \rightarrow \mathrm{Alg}(\mathcal{C}^\mathrm{op})$
\end{center}
which shows that the bar construction over a single augmentation has a natural coalgebra structure. The algebra and coalgebra structures are compatible since the algebra structure can be directly derived from the previous functor:
Since $$\CAlg^{/\mathbf{1}}(\CAlg(\mathcal{C})) \simeq \CAlg^{/\mathbf{1}}(\mathcal{C})$$ \cite[Corollary 2.4.3.10]{ha}, we can apply the bar construction directly in $\CAlg(\mathcal{C})$ which gives a functor $$\Baar': \CAlg^{/\mathbf{1}}(\mathcal{C}) \simeq \CAlg^{/\mathbf{1}}(\CAlg(\mathcal{C}))^\mathrm{op}  \rightarrow \mathrm{Alg}(\CAlg(\mathcal{C})^\mathrm{op}).$$ Since the forgetful functor $\CAlg(\mathcal{C}) \rightarrow \mathcal{C}$ commutes with colimits of simplicial diagrams \cite[Corollary 3.2.3.2]{ha}, we have that $\Baar = \Baar'$.

\smallskip

The coproduct inducing the bialgebra structure of the bar construction can be derived from the more general deconcatenation map:

\begin{lem}\label{deco}
    Suppose $A$ has three augmentations $x,y,z$. Then there is a natural deconcatenation map in $\mathrm{CAlg}(\mathrm{Mod}_A(\mathcal{C}))$ \begin{center}
    \begin{tikzcd}
       \+^{}_x\Delta_z^y \;: \;_x\Baar_z(A) \arrow[r] & \;_y\Baar_z(A) \otimes \,_x\Baar_y(A)
    \end{tikzcd}
\end{center}
\end{lem}
\begin{proof}
    The map is built as follows:
    \begin{eqnarray*}
        \!_x\Baar_z(A) & = & \mathbf{1}_z\otimes_A \mathbf{1}_x\\
             & \simeq & \mathbf{1}_z \otimes_A A \otimes_A \mathbf{1}_x\\
             & \rightarrow & \mathbf{1}_z \otimes_A \mathbf{1}_y \otimes_A  \mathbf{1}_x\\
             & \simeq & (\mathbf{1}_z \otimes_A \mathbf{1}_y) \otimes_\mathbf{1} (\mathbf{1}_y \otimes_A \mathbf{1}_x)\\
             & = & \!_y\Baar_z(A) \otimes \!_x\Baar_y(A).
    \end{eqnarray*}
    Note that the line 4 uses both the associativity and commutativity of the relative tensor product and a change of base algebra \cite[Section 4.5.3]{ha}.
\end{proof}

Over a single augmentation $x$, the deconcatenation map is the coproduct map on the bar construction of $A$ over $x$. Note that there is as well a canonical augmentation given by the natural morphism $$\epsilon_x: \mathbf{1}_x \sqcup_A \mathbf{1}_x \rightarrow \mathbf{1}_x.$$

We can readily prove the counit axiom:

\begin{lem}\label{medlem}
    For $x,y$ augmentations of $A$, the following equality $$(\epsilon_y \otimes \mathrm{Id})\circ \+^{}_x\Delta_y^y = \mathrm{Id}$$ and $$(\mathrm{Id} \otimes \epsilon_x)\circ \+^{}_x\Delta_y^x = \mathrm{Id},$$ hold in the homotopy category of $\mathcal{C}$.
\end{lem}
\begin{proof}
Let us treat the first case, as the second one is similar. The composition is \begin{eqnarray*}
        \mathbf{1}_y\otimes_A \mathbf{1}_x & \simeq & \mathbf{1}_y \otimes_A A \otimes_A \mathbf{1}_x\\
        & \rightarrow & \mathbf{1}_y \otimes_A \mathbf{1}_y \otimes_A \mathbf{1}_x\\
        & \simeq & (\mathbf{1}_y \otimes_A \mathbf{1}_y) \otimes (\mathbf{1}_y \otimes_A \mathbf{1}_x)\\
        & \rightarrow & \mathbf{1}_y \otimes (\mathbf{1}_y \otimes_A \mathbf{1}_x)\\
        & \simeq & \mathbf{1}_y \otimes_A \mathbf{1}_x
        \end{eqnarray*}
        this morphism is given by the canonical morphism $$\mathbf{1}_x \rightarrow \mathbf{1}_y \otimes_A \mathbf{1}_x$$ and by the morphism on the left component $$\mathbf{1}_y \rightarrow \mathbf{1}_y \otimes_A \mathbf{1}_y \rightarrow \mathbf{1}_y$$ which is the identity by the algebra structure of $\mathbf{1}_y \otimes_A \mathbf{1}_y$.
\end{proof}

Let us now consider the involution $$\!_xs_y: \;_x\Baar_y(A) = \mathbf{1}_y \otimes_A \mathbf{1}_x \rightarrow \mathbf{1}_x \otimes_A \mathbf{1}_y = \;_y\Baar_x(A)$$ exchanging the two augmentations. Then, we can check the coinverse axiom. This proves in particular that $\+_x\Baar_x(A)$ is a homotopy Hopf algebra for all augmentations $x$ of $A$.

\begin{lem}\label{bar hopf}
    The algebras $\+_x\Baar_x(A)$, $\+_x\Baar_y(A)$ and $\+_y\Baar_x(A)$ equipped with the deconcatenation $\+_x^{}\Delta_x^y$, product $\+_x\nabla_y$, unit $\+_x\eta_y: \mathbf{1} \rightarrow \+_x\Baar_y(A)$, counit $\epsilon_x: \+_x\Baar_x(A) \rightarrow \mathbf{1}$ and involution $\+_xs_y$  fit in the homotopy commutative diagrams:\begin{center}\begin{tikzcd}[row sep=large, column sep=tiny]
        & \+_y\Baar_x(A) \otimes \+_x\Baar_y(A) \arrow[rr] && \+_x\Baar_y(A) \otimes \+_x\Baar_y(A) \arrow[rd] & \\
        \+_x\Baar_x(A) \arrow[ru] \arrow[rr] & & \mathbf{1} \arrow[rr] && \+_x\Baar_y(A)\\
        
    \end{tikzcd}\end{center}
    and 
    \begin{center}
        \begin{tikzcd}[row sep=large, column sep= tiny]
         \+_x\Baar_x(A) \arrow[rr] \arrow[rd] & & \mathbf{1} \arrow[rr] && \+_y\Baar_x(A)\\
             & \+_y\Baar_x(A) \otimes \+_x\Baar_y(A) \arrow[rr] && \+_y\Baar_x(A) \otimes \+_y\Baar_x(A) \arrow[ru] & 
        \end{tikzcd}
    \end{center}
\end{lem}
\begin{proof}
    We only prove the proposition in the upper diagram since the other case works in the same way.\\
    The upper composition can be described by the sequence \begin{eqnarray*}
        \mathbf{1}_x\otimes_A \mathbf{1}_x & \simeq & \mathbf{1}_x \otimes_A A \otimes_A \mathbf{1}_x\\
        & \rightarrow & \mathbf{1}_x \otimes_A \mathbf{1}_y \otimes_A \mathbf{1}_x\\
        & \simeq & (\mathbf{1}_x \otimes_A \mathbf{1}_y) \otimes (\mathbf{1}_y \otimes_A \mathbf{1}_x)\\
        & \simeq & (\mathbf{1}_y \otimes_A \mathbf{1}_x) \otimes (\mathbf{1}_y \otimes_A \mathbf{1}_x)\\
        & \rightarrow & \mathbf{1}_y \otimes_A \mathbf{1}_x
        \end{eqnarray*}
    where the last arrow is identified with the componentwise product. Note that this last arrow can be described by the change of base ring equivalence:
    $$(\mathbf{1}_y \otimes_A \mathbf{1}_x) \otimes (\mathbf{1}_y \otimes_A \mathbf{1}_x) \simeq \mathbf{1}_y \otimes_A (\mathbf{1}_x \otimes_A \mathbf{1}_x) \rightarrow \mathbf{1}_y \otimes_A \mathbf{1}_x$$ where the last arrow is induced by the product on $\mathbf{1}_x$.
    Then the full composition is homotopic to the composition \begin{eqnarray*}
        \mathbf{1}_x \otimes_A \mathbf{1}_x & \simeq & A \otimes_A \mathbf{1}_x \otimes_A \mathbf{1}_x\\
        & \rightarrow & \mathbf{1}_y \otimes_A \mathbf{1}_x \otimes_A \mathbf{1}_x\\
        & \rightarrow & \mathbf{1}_y \otimes_A \mathbf{1}_x
    \end{eqnarray*}
    where the last arrow is again induced by the product on $\mathbf{1}_x$.
    This is itself identified with the map \begin{eqnarray*}
        \mathbf{1}_x \otimes_A \mathbf{1}_x & \rightarrow & \mathbf{1}_x\\
    & \simeq & A \otimes_A \mathbf{1}_x\\
    & \rightarrow & \mathbf{1}_y \otimes_A \mathbf{1}_x
        \end{eqnarray*}
        and this composition is homotopic to $\+_y\eta_x\circ\epsilon_x$.
        \end{proof}
        
Finally it remains to check the coassociativity axiom:

\begin{lem}\label{lastlem}
    Let $x,y,z,t$ four augmentations of $A$. Then the equality $$(\+^{}_y\Delta_t^z \otimes \mathrm{Id})\circ\+^{}_x\Delta_t^y = (\mathrm{Id} \otimes \+^{}_x\Delta_z^y)\circ\+^{}_x\Delta_t^z$$ holds in the homotopy category of $\mathcal{C}$.
\end{lem}
\begin{proof}
The lemma is proved by checking that the two compositions are described (up to associativity of the tensor product) by the composition $$\mathbf{1}_t \otimes_A \mathbf{1}_x \simeq \mathbf{1}_t \otimes_A A \otimes_A A \otimes_A \mathbf{1}_x \rightarrow \mathbf{1}_t \otimes_A \mathbf{1}_z \otimes_A \mathbf{1}_y \otimes_A \mathbf{1}_x$$
and then two successive change of base algebras.
\end{proof}
        
\begin{rem}
In their recent work, I. Dan-Cohen and A. Horev proved that \cite[Corollary 8.4.3]{dancohen2024relativetensorproductskoszul} the deconcatenation maps from lemma \ref{deco} endow $\+_{x}\Baar_{y}(A)$ with left (resp. right) coalgebra structure under the coaction of $\+_{x}\Baar_{x}(A)$ (resp. $\+_{y}\Baar_{y}(A)$). This lifts part of the structure described above $\infty$-categorically.
\end{rem}
    
\begin{rem}
Since it is defined as the colimit of a simplicial object in $\mathcal{C}$, the bar construction is a filtered object in $\mathcal{C}$ \cite[Remark 1.2.4.2]{ha}. The filtration is given by the geometric realization of skeletons that we denote for $n \in \mathbb{N}$ by $$\Baar_{\leq n}:= \colim_\Delta\;\mathrm{sk}_n\Baar_\bullet.$$ In a $1$-categorical context, $\!_x\Baar_x(A)$ has a structure of filtered Hopf algebra since the operations we introduced are compatible with the filtration, see \cite{wojtkowiak:hal-01293611}, \cite[Section 3.5]{Javier} for a full description. Furthermore, $\+_x\Baar_y(A)$ has the structure of a filtered comodule under the coactions of $\+_x\Baar_x(A)$ and $\+_y\Baar_y(A)$. We do not know if such a result can be extended to an $\infty$-categorical context, but we will not need it in the rest of this paper, since we will only describe the filtered structure at the level of realizations.
\end{rem}

\begin{exmp}\label{classical}
Let us detail the construction in the classical case of a commutative differential graded algebra $A$, viewed as an object of the $\infty$-category of chain complexes over a field $K$, endowed with two augmentations $x,y: A \rightarrow K$, since this will come in handy in the computation of the realizations of the motivic fundamental groupoid. By \cite[Corollary 3.13]{arakawa2023homotopylimitshomotopycolimits}, the geometric realization of a simplicial object $X_\bullet$ in $\mathcal{M}\mathrm{od}_K$ is given by the total complex associated with the bicomplex $\mathrm{M}(X_\bullet)$ where $\mathrm{M}$ is the Moore complex functor. Hence,
$$\!_x\Baar_y(A) \simeq \mathrm{Tot}\Big(\mathrm{M}(\Baar_A(\mathbf{1}_x,\mathbf{1}_y)_\bullet)\Big)$$
If we use cohomological conventions, then $$\!_x\Baar^n_y(A) \simeq \bigoplus_{n = p - q} \bigoplus_{i_1 + \dots + i_q = p} A^{i_1} \otimes \dots \otimes A^{i_q}.$$ Hence, a degree $n$ element of the bar complex is denoted by $$[\omega_1|\cdots|\omega_q], \text{ with } \omega_i \in A, \text{ and } \sum_i \mathrm{deg}(\omega_i) = n + q.$$
In this construction, the filtration on the bar complex is interpreted as the classical length filtration, that is $\;_x\Baar_y(A)_{\leq n}$ is the subcomplex spanned by elements of the form given above, with $q \leq n$.

\end{exmp}
\subsection{The motivic pointed path space based at tangential basepoints}\label{pointedpathspace}\;\smallskip

We now have all the tools in hand to build the motivic pointed path space based at tangential basepoints. Let $X\in \mathbf{Div}_k$, $\mathbf{x}, \mathbf{y}$ two tangential basepoints of $X$ and $h(\mathbf{x}), h(\mathbf{y}): h(X) \rightarrow \Lambda$ the corresponding motivic tangential augmentations. Then $h(X)$ has the structure of a biaugmented commutative algebra object in $\mathcal{DA}(k,\Lambda)$. We can then apply the bar construction to this datum.  

\begin{defi}\label{pathspace}
 We call motivic pointed path space of $X$ based at $\mathbf{x}$ and $\mathbf{y}$ the algebra object $$\+^{}_\mathbf{x}\mathrm{P}^\mathrm{m}_\mathbf{y}X:= h(\mathbf{1_y}) \otimes_{h(X)} h(\mathbf{1_x}):= \;_{h(\mathbf{x})}\Baar_{h(\mathbf{y})}(h(X))$$  
 This determines a functor
$$\mathbf{Div}^{\mathbf{v}*//}_{k} \rightarrow  \CAlg(\mathcal{DA}^\mathrm{log}(k,\Lambda))$$
\end{defi}

\begin{rem}\label{ordre}
Note that we chose this order for points to be aligned with the convention chosen in \cite{ASENS_2005_4_38_1_1_0,Javier}. However, we noted by $``\+^{}_\mathbf{x}\mathrm{P}^\mathrm{m}_\mathbf{y}X"$ what would be written $``\+^{}_\mathbf{y}\mathrm{P}^\mathrm{m}_\mathbf{x}X"$ in these papers because we wanted this notation to represent the space of paths going from $\mathbf{x}$ to $\mathbf{y}$.
\end{rem}

By the results of the previous subsection, the motivic based path spaces is a highly structured algebra in $\mathcal{DA}^\mathrm{log}(k,\Lambda)$. Precisely, it is a straightforward consequence of theorem \ref{bar hopf} that the following proposition holds:

\begin{prop}\label{truetan}
The family $\{\+^{}_\mathbf{x}\mathrm{P}^m_\mathbf{y}X\}_{\mathbf{x},\mathbf{y}}$ for all tangential basepoints $\mathbf{x},\mathbf{y}$ has the structure of a homotopy Hopf algebroid in $\mathcal{DA}^\mathrm{log}(k,\Lambda)$.
\end{prop}

\begin{rem}
We draw inspiration from \cite{dancohen2021rationalmotivicpathspaces}, \cite{spitzweck2010derivedfundamentalgroupstate}, and \cite{iwanari2017motivicrationalhomotopytype}, where respectively I. Dan-Cohen and T. Schlank, M. Spitzweck, and I. Iwanari defined the motivic fundamental group in an $\infty$-categorical/model theoretic context.
\end{rem}

\subsubsection{The cosimplicial pointed path space}\;\smallskip \label{The cosimplicial pointed path space}

When $\mathbf{x}, \mathbf{y}$ are ordinary basepoints, as mentioned in the introduction of this article, the definition of the motivic pointed path space in the case of ordinary basepoints as explained by Deligne and Goncharov in \cite{ASENS_2005_4_38_1_1_0} or by Levine and H. Esnault in \cite{esnault2007tatemotivesfundamentalgroup} relies on a translation to the language of algebraic geometry, due to Wojtkowiak in \cite{wojtkowiak:hal-01293611}, of the topological definition of the pointed path space. We generalize this definition in our context, since we will need it to compute the realization of the motivic pointed path space in the next section:

    For $X \in \mathbf{lSch}_k$ let us consider the cosimplicial log scheme given by \begin{center}
        $X^{[0,1]}:= X^{\Delta[1]}.$ 
    \end{center}
    For $n \in \mathbb{N}$, $$(X^{\Delta[1]})^\mathbf{n} = \prod_{s\in \Delta[1]_n}X = X \times X^n \times X,$$
    the i-th coface map for $0 \leq i \leq n+1$ is given by $$(x_0,\dots, x_i, \dots,x_{n+1}) \longmapsto (x_0, \dots, x_i,x_i,x_{i+1}, \dots, x_{n+1})$$
    and the i-th codegeneracy map for $0 \leq i \leq n-1$ by the projection $$X^{i+1} \times X \times X^{n-i} \rightarrow X^{i+1} \times X^{n-i}.$$

\begin{rem}
    $X^{[0,1]}$ comes equipped with a map $$X^{[0,1]} \rightarrow X^{\partial\Delta[1]} = X \times X$$ induced by the inclusion $\partial\Delta[1] \subset \Delta[1]$, where the right term is the constant cosimplicial log scheme. This map is degreewise the projection on the first and last factor.
\end{rem}

\begin{defi}\label{cosimplicial path space general}
For $\mathbf{x},\mathbf{y}$ two tangential basepoints on $X\in \mathbf{Div}_k$, we call \textbf{cosimplicial path space of} $X$ based at $\mathbf{x},\mathbf{y}$ the cosimplicial object of $\+^\mathbf{v}\mathbf{lSch}_k$ defined by the pullback \begin{center}\begin{tikzcd}
\+^{}_{\mathbf{x}}\mathrm{P}^\bullet_{\mathbf{y}}X \arrow[dr, phantom, "\usebox\pullback" , very near start, color=black] \arrow[r] \arrow[d] & X^{[0,1]} \arrow[d]\\
\mathrm{Spec}(k) \arrow[r,"{\mathbf{y},\mathbf{x}}"] & X\times X
\end{tikzcd}\end{center}
\end{defi}

Let us explain how this fiber product is defined in the category of cosimplicial diagrams of log schemes with virtual morphisms between them. Its $n$-cosimplices are $X^n$, the degeneracies are all the projections $X^{n} \rightarrow X^{n-1}$, and the cofaces are the same as the ones the upper description, except for the $0$-th and the $(n+1)$-th, defined by the product of the virtual point $\mathbf{y}$ with the identity of $X^{n+1}$ and by the product of the identity of $X^{n+1}$ with the virtual point $\mathbf{x}$.

\medskip

Since we cannot view this object in a log motivic formalism in a canonical way (there are no category of virtual log motives) we cannot compare the bar construction of $h(X)$ based at $\mathbf{x},\mathbf{y}$ with a \textit{simplicial motive associated with the cosimplicial path space based at} $\mathbf{x},\mathbf{y}$, as in the case with ordinary basepoints. Nevertheless we can check that our definition coincide with the latter in the case with ordinary basepoints.

\begin{prop}\label{path space facile}
    When $\mathbf{x},\mathbf{y}$ are rational points of $\underline{X}\setminus \underline{D}$, there is a natural equivalence in $\CAlg (\mathcal{DA}^\mathrm{log}(k,\Lambda))$: $$\displaystyle\+^{}_\mathbf{x}\mathrm{P}^m_\mathbf{y}X \simeq \colim_{\Delta} h(\+_\mathbf{x}^{}\mathrm{P}^\bullet_\mathbf{y} X).$$
\end{prop}
\begin{proof}
Once we check that the motivic augmentations associated with $\mathbf{x},\mathbf{y}$ are the usual augmentation associated with the corresponding rational points $x,y$ which is done in remark \ref{usual}, this is a classical fact, proved in \cite[Theorem A.7]{iwanari2017motivicrationalhomotopytype} for instance, or see the same proof applied in the Betti case in the discussion following theorem \ref{thm Betti real pi}.
\end{proof}

\subsection{The motivic fundamental groupoid at tangential basepoints}\label{fundamentalgroup}\;\smallskip

We can now define the motivic fundamental groupoid based at tangential basepoints. Although we work in the context of mixed Tate motives, what follows works in a similar fashion in the general case of the existence of a motivic $t$-structure.

Recall that the triangulated category of mixed Tate motives over a field $k$ is classically defined to be the smallest triangulated full subcategory of the triangulated category of mixed motives \textit{à la Voevodsky}  $\mathbf{DM}(k,\mathbb{Q})$ \cite{Voevodsky+2000+188+238} generated by the Tate objects $\mathbb{Q}(n)$ for $n\in \mathbb{Z}$. In the language of $\infty$-categories, the right definition is (see \cite[Section 6.2]{iwanari2017motivicrationalhomotopytype} for details):

\begin{defi}
    Let $\mathcal{DM}(k,\mathbb{Q})$ be the stable $\infty$-category of mixed motives over $k$ with $\mathbb{Q}$ coefficients (see for instance \cite{iwanari2017motivicrationalhomotopytype,park2023mathbba1homotopytheorylogschemes} for a definition). Then we let $\mathcal{DMT}_\mathrm{gm}(k,\mathbb{Q})$ be the smallest stable subcategory of $\mathcal{DM}(k,\mathbb{Q})$ generated by the Tate objects $\mathbb{Q}(n)$. The stable $\infty$-category of mixed Tate motives over $k$ is the full subcategory of $\mathcal{DM}(k,\mathbb{Q})$ defined by $$\mathcal{DMT}(k,\mathbb{Q}):= \mathrm{Ind}\mathcal{DMT}_\mathrm{gm}(k,\mathbb{Q}).$$ 
\end{defi}

Suppose now that $k$ satisfies the Beilinson--Soulé vanishing conjecture, $\mathcal{DMT}(k,\mathbb{Q})$ and $\mathcal{DMT}_\mathrm{gm}(k,\mathbb{Q})$ come with a t-structure \cite{Levine1993}, \cite[Section 7]{iwanari2013barconstructiontannakization} (defined on the homotopy categories see \cite[Definition 1.2.1.4]{ha}). The heart of $\mathcal{DMT}_\mathrm{gm}(k,\mathbb{Q})$ is the neutral tannakian (1-)category of mixed Tate motives over $k$, denoted by $\mathbf{MT}(k,\mathbb{Q})$. There is a $0$-cohomology functor $$\mathrm{H}_\mathrm{mot}^0 : \mathcal{DMT}_\mathrm{gm}(k,\mathbb{Q}) \rightarrow \mathbf{MT}(k,\mathbb{Q}).$$ The heart of the $t$-structure on $\mathcal{DMT}(k,\mathbb{Q})$ is $\mathrm{Ind}\mathbf{MT}(k,\mathbb{Q})$ and the corresponding $0$-cohomology functor $\mathcal{DMT}(k,\mathbb{Q})$ is the indization of $\mathrm{H}^0_\mathrm{mot}$.

\medskip

Recall as well from \cite[Theorem 16.2.18, Theorem 16.1.4]{Cisinski_2019} that there is an equivalence of symmetric monoidal $\infty$-categories $$\mathcal{DA}_{\text{ét}}(k,\mathbb{Q})^\otimes \simeq\mathcal{DM}(k,\mathbb{Q})^\otimes$$
so that $\mathcal{DMT}(k,\mathbb{Q})$ can be viewed as a full subcategory of $\mathcal{DA}_\text{ét}(k,\mathbb{Q})$.

\begin{defi}
    We say that a log smooth log scheme $X$ over $k$ is of mixed Tate type if $M^{\text{ét}}_k(X) \in \mathcal{DMT}(k,\mathbb{Q})$.
\end{defi}

If $X \in \mathbf{Div}_k$ then the structure of homotopy Hopf algebroid of the family of motivic path spaces at tangential basepoint is preserved after étale localization, i.e. after applying the functor induced by étale sheafification. Furthermore, for $\mathbf{x},\mathbf{y}$ tangential basepoints on $X$, the motivic path space based at $\mathbf{x},\mathbf{y}$ is an element of $\mathcal{DMT}(k,\mathbb{Q})$ by definition. We now show that taking its $0$-th cohomology group carries the whole structure to $\mathrm{Ind}\mathbf{MT}(k,\mathbb{Q})$:

\begin{thm}
    The family of $\mathrm{Ind}\mathbf{MT}(k,\mathbb{Q})$ objects $\{\mathrm{H}^0_\mathrm{mot}((\+^{}_\mathbf{x}\mathrm{P}_\mathbf{y}^mX)^\text{ét})\}_{\mathbf{x},\mathbf{y}}$ has the structure of a Hopf algebroid deduced from the structure of theorem \ref{groupoid}, i.e. the corresponding family in the opposite category $\mathrm{Ind}\mathbf{MT}(k,\mathbb{Q})^\mathrm{op}$ has a structure of a groupoid.
\end{thm}
\begin{proof}
The only potential obstruction to this lies in the fact that the functor $\mathrm{H}^0_\mathrm{mot}$ is not symmetric monoidal. Let us show that for $\mathbf{x,x',y,y'}$ tangential basepoints on $X$ it satisfies $$\mathrm{H}^0_\mathrm{mot}((\+^{}_\mathbf{x}\mathrm{P}_\mathbf{x'}^mX)^{\text{ét}}\otimes (\+^{}_\mathbf{y}\mathrm{P}_\mathbf{y'}^mX)^{\text{ét}}) \simeq \mathrm{H}^0_\mathrm{mot}((\+^{}_\mathbf{x}\mathrm{P}_\mathbf{x'}^mX)^{\text{ét}})\otimes \mathrm{H}^0_\mathrm{mot}((\+^{}_\mathbf{y}\mathrm{P}_\mathbf{y'}^mX)^{\text{ét}})).$$ Which naturally implies the theorem. We know from \cite[Example 3.3.2]{BIGLARI2007645} that the $t$-structure on $\mathcal{DMT}_\mathrm{gm}(k,\mathbb{Q})$ is compatible with the tensor structure, so that the associated family of cohomology functors $\mathrm{H}^i_\mathrm{mot}$ satisfies the Künneth formula. This implies that the family $\mathrm{H}^i_\mathrm{mot}$ satisfies the Künneth formula as well, since the functors $\mathrm{H}^i_\mathrm{mot}$ commute with filtered colimits.

It is now clear that it suffices to show that for $\mathbf{x},\mathbf{y}$ tangential baspoints on $X$, the cohomology in negative degrees of $\+^{}_\mathbf{x}\mathrm{P}_\mathbf{y}^mX$ is zero. It suffices to check this for the filtered pieces $\+^{}_\mathbf{x}\mathrm{P}_\mathbf{y}^mX_{\leq n}$. We only need to check that the Betti realization of this motive is concentrated in nonpositive degree because the Betti realization is conservative. This is trivial from the description at the end of section \ref{Bettikato}.
\end{proof}

From this, we are able to give the following definition of the motivic fundamental groupoid at tangential basepoints:

\begin{defi}\label{mixed}
    Let $X \in \mathbf{Div}_k$, with tangential basepoints $\mathbf{x},\mathbf{y}$, and suppose $X$ is of mixed Tate type. Then the motivic fundamental torsor at $\mathbf{x},\mathbf{y}$ is defined by $$\pi_1^m(X,\mathbf{x},\mathbf{y}):= \operatorname{Spec}\mathrm{H}^0_\mathrm{mot}((\+^{}_\mathbf{x}\mathrm{P}_\mathbf{y}^mX)^\text{ét}) \in \mathrm{Pro}\mathbf{MT}(k,\mathbb{Q})-\mathbf{Sch}$$ Where $\mathrm{Pro}\mathbf{MT}(k,\mathbb{Q})-\mathbf{Sch}$ is the opposite of the category of commutative algebra objects in $\mathrm{Ind}\mathbf{MT}(k,\mathbb{Q}).$
\end{defi}

The following proposition is trivial from the proposition \ref{truetan}

\begin{prop}
The family of the motivic fundamental torsors $\pi_1^m(X,\mathbf{x},\mathbf{y})$ for all tangential basepoints $\mathbf{x},\mathbf{y}$ of $X$, equipped with the morphisms duals of the morphisms constructed in theorem \ref{bar hopf}, has the structure of a groupoid.
\end{prop} 

Precisely, for $\mathbf{x},\mathbf{y},\mathbf{z}$ three tangential basepoints of $X$, there is an operation of composition of paths $$\pi^m_1(X,\mathbf{y},\mathbf{z}) \times \pi^m_1(X,\mathbf{x},\mathbf{y}) \rightarrow \pi^m_1(X,\mathbf{x},\mathbf{z})$$ which is associative in the obvious sense,
an isomorphism of inversion of paths $$\pi_1^m(X,\mathbf{x},\mathbf{y}) \simeq \pi_1^m(X,\mathbf{y},\mathbf{x}),$$ which yields an inverse for the composition of paths, and each $$\pi_1^m(X,\mathbf{x}):= \pi_1^m(X,\mathbf{x},\mathbf{x})$$ as the structure of a pro-$\mathbf{MT}(k,\mathbb{Q})$ group scheme. All these data satisfy the axioms dual to the axioms given in theorem \ref{bar hopf}.

We call the family $\{\pi_1^m(X,\mathbf{x},\mathbf{y})\}_{\mathbf{x},\mathbf{y}}$ equipped with the operations described above the \textbf{motivic fundamental groupoid at tangential basepoints.}

\begin{rem}
In particular, for all $\mathbf{x,y}$ tangential basepoints of $X$, $\pi_1^m(X,\mathbf{x},\mathbf{y})$ is endowed with a left (resp. right) action of $\pi_1^m(X,\mathbf{x})$ (resp. $\pi_1^m(X,\mathbf{y})$).
\end{rem}

\newpage
\section{Realizations}\label{realizations}
In this last section, we show that the degree zero cohomology of the Betti and de Rham realizations of the motivic pointed path space recover the expected Betti--de Rham theory of the prounipotent completion of the fundamental torsor at tangential basepoints. To do so, in the first subsection we explain how to build realization functors for $\mathbb{A}^1$-invariants log motives. An approach to do this in an $\infty$-categorical context is to employ the language of mixed Weil cohomology theories. We follow and adapt Drew's work in \cite{drew2018motivichodgemodules}. We compute the Betti (resp. the de Rham) realization functor, and we compute the realizations of the motivic pointed path space in the second subsection (resp. the third). In both cases we find explicit complexes representing this realization. We navigate between the different approaches of tangential basepoints explained in subsection \ref{pointed}, and between $\infty$ and $1$-categories according to our needs. In the last subsection, we compute the periods of the motivic fundamental groupoid via the notion of \textit{regularized iterated integrals}, adapted from regularized integration as defined in \cite{dupont2024regularizedintegralsmanifoldslog}.
\subsection{Realization functors}\label{build}\;\smallskip

In this subsection, $k$ is an arbitrary field. We can naturally generalize the definition of mixed Weil theories given in \cite{drew2018motivichodgemodules} to a log geometric context:

\begin{defi}\label{log Weil}
     A $\Lambda$-linear log mixed Weil theory over $k$ is a symmetric monoidal functor $\mathbf{\Gamma}^\otimes: \mathbf{lSm}_k^{op,\times} \rightarrow \mathcal{M}\mathrm{od}_\Lambda^\otimes$ such that if we let $\mathbf{\Gamma}$ be the underlying functor of $\infty$-categories, the following conditions are satisfied: \begin{itemize}
        \item $\mathbf{\Gamma}$ has its values in $\operatorname{Perf}(\Lambda)$, the sub-$\infty$-category of bounded complexes of finite projective modules.
        \item $\mathbf{\Gamma}$ is a $\mathcal{M}\mathrm{od}_\Lambda$-valued $\mathrm{dNis}$-sheaf.
        \item $\mathbf{\Gamma}^\mathrm{op}: \mathbf{lSm}_k \rightarrow \mathcal{M}\mathrm{od}_\Lambda^\mathrm{op}$ is $\mathbb{A}^1$-invariant and $\mathrm{ver}$-invariant.
        \item $\mathbf{\Gamma}$ sends $\cofib(1\rightarrow \mathbb{G}^\mathrm{log}_m)$ to a $\otimes$-invertible object.
    \end{itemize}
    Furthermore, we say that a log mixed Weil theory is an étale log mixed Weil theory if it is as well a strict étale sheaf.
\end{defi}

Log mixed Weil theories, modeling log cohomology theories, are the starting point to build realization functors. We prove a logarithmic version of the theorem 3.7 in \cite{drew2018motivichodgemodules}. Note that it is a log and $\Lambda$-linear version of \cite[Theorem 2.9]{ROBALO2015399}:

\begin{prop}\label{log linear yoneda}
    The homological motive functor induces a homotopy equivalence of Kan complexes $$\operatorname{Fun}_{\CAlg(\PrLotimes)^{\mathcal{M}\mathrm{od}_\Lambda^\otimes/}}(\DAlog(k,\Lambda)^\otimes,\mathcal{M}\mathrm{od}_\Lambda^\otimes) \xrightarrow{\sim} \operatorname{Fun}_{\CAlg(\Catinf^\times)}(\mathrm{N}(\mathbf{lSm}_k)^\times,\mathcal{M}\mathrm{od}_\Lambda^\otimes)_{\mathrm{P}^\mathrm{log}}.$$
\end{prop}

The subscript $\mathrm{P}^\mathrm{log}$ means the subcomplex generated by the functors $\mathbf{\Gamma}$ such that $\mathbf{\Gamma}^\mathrm{op}$ satisfies the three last points of definition \ref{log Weil}.

\begin{rem}
    The proposition means that the category $\DAlog(k,\Lambda)$ has the universal property that any functor defined from the category of log smooth schemes over $k$ to the category of perfect complexes of $\Lambda$-modules which satisfies $dNis$-descent, which trivializes $\mathbb{A}^1$, $\mathrm{ver}$-covers and $\mathbb{G}_m^\mathrm{log}/1$ factorizes by it. We will use this fact as follows: from a log mixed Weil theory $\mathbf{\Gamma}$, since $\mathrm{Perf}(\Lambda)$ is the sub-$\infty$-category of dualizable objects in $\mathcal{M}\mathrm{od}_\Lambda$, we can consider the functor $$(\mathbf{\Gamma}^\mathrm{op})^\vee: \mathbf{lSm}_k \xrightarrow{\mathbf{\Gamma}^\mathrm{op}} \mathrm{Perf}(\Lambda)^\mathrm{op} \xrightarrow{(-)^\vee} \mathrm{Perf}(\Lambda)\hookrightarrow \mathcal{M}\mathrm{od}_\Lambda.$$ which has the property $\mathrm{P}^\mathrm{log}$ so that it induces a realization functor $$\mathrm{R}^\otimes: \DAlog(k,\Lambda)^\otimes \rightarrow \mathcal{M}\mathrm{od}^\otimes_\Lambda$$ satisfying for $X \in \mathbf{lSm}_k$, $$\mathbf{R}(M(X)) \simeq \mathbf{\Gamma}(X)^\vee.$$
\end{rem}

In order to prove the proposition \ref{log linear yoneda}, we first prove the following, based on scattered remarks in \cite{drew2018motivichodgemodules} (see for instance the proof of proposition 3.14):

\begin{lem}\label{da sh}
    For $S\in \mathbf{lSch}$, there is a canonical equivalence $$\DAlog(S,\Lambda) \simeq \SHlog(S) \otimes \mathcal{M}\mathrm{od}_\Lambda,$$ where the tensor product is taken in $\PrL$.
\end{lem}
\begin{proof}
    By \cite[Proposition 2.4 (1)]{drew2018motivichodgemodules}, we have the natural equivalence (note that hypersheaves and sheaves coincide here by \cite[Remark 2.4.7]{binda2024logarithmicmotivichomotopytheory}) $$\Sh_{\mathrm{dNis}}(\mathbf{lSm}_S,\mathcal{M}\mathrm{od}_\Lambda) \simeq \Sh_{\mathrm{dNis}}(\mathbf{lSm}_S,\mathcal{S}\mathrm{pc})\otimes \mathcal{M}\mathrm{od}_\Lambda.$$
    Now, the tensor product of presentable categories commutes with small colimits separately in each variables by \cite[4.8.1.24]{ha}, and the inclusion $\PrL_{st} \subset \PrL$ preserves small colimits. Therefore, since localization of stable presentable $\infty$-categories are cofibers in $\PrL_{st}$ (Verdier quotients see \cite[Proposition 5.6]{Blumberg_2013}), the localization commutes with tensor product with $\Mod_\Lambda$: \begin{center}$\begin{array}{ccc}
        \mathrm{L}_{\mathbb{A}^1}\mathrm{L}_{\mathrm{ver}}\Sh_{\mathrm{dNis}}(\mathbf{lSm}_S,\Mod_\Lambda) & \simeq & (\mathrm{L}_{\mathbb{A}^1}\mathrm{L}_{\mathrm{ver}}\Sh_{\mathrm{dNis}}(\mathbf{lSm}_S,\Spc))\otimes \Mod_\Lambda\\
        &\simeq & (\mathrm{L}_{\mathbb{A}^1}\mathrm{L}_{\mathrm{ver}}\Sh_{\mathrm{dNis}}(\mathbf{lSm}_S,\Spc_*))\otimes \Mod_\Lambda 
    \end{array}$\end{center}
    where the last equivalence follows from the fact that $\Mod_\Lambda$ is pointed, see \cite[Definition 3.1 (4)]{drew2018motivichodgemodules}.

Using again the fact that the tensor product in $\PrL$ commutes with small colimits, we have the following natural equivalences: \begin{center}
    $\begin{array}{ccl}
      \DAlog(S,\Lambda)   &  \simeq &  \colim_{\Sigma_{\mathbb{G}_m^\mathrm{log}/1}} \mathrm{L}_{\mathbb{A}^1}\mathrm{L}_{\mathrm{ver}}\Sh_{\mathrm{dNis}}(\mathbf{lSm}_S,\Mod_\Lambda)\\
         & \simeq & (\colim_{\Sigma_{\mathbb{G}_m^\mathrm{log}/1}}\mathrm{L}_{\mathbb{A}^1}\mathrm{L}_{\mathrm{ver}}\Sh_{\mathrm{dNis}}(\mathbf{lSm}_S,\Spc_*)) \otimes \Mod_\Lambda\\
         & \simeq & \SHlog(S) \otimes \Mod_\Lambda
    \end{array}$
\end{center}
which concludes the proof.
\end{proof}

Let us now check the log version of \cite[Theorem 2.39]{ROBALO2015399}, already proven for the category $\mathrm{log}\mathcal{SH}$ in \cite{binda2024logarithmicmotivichomotopytheory}.

\begin{lem}\label{linear yoneda}
    The homological motive functor induces an equivalence of $\infty$-categories $$\operatorname{Map}_{\CAlg(\PrLotimes)}(\SHlog(k)^\otimes,\Mod_\Lambda^\otimes) \xrightarrow{\sim} \operatorname{Map}_{\CAlg(\Catinf^\times)}(\mathrm{N}(\mathbf{lSm}_k)^\times,\Mod_\Lambda^\otimes)_{\mathrm{P}^\mathrm{log}}$$
\end{lem}
\begin{proof}
    Following \cite[Theorem 2.39]{ROBALO2015399} the proof is exactly the same; the Yoneda lemma combined with the universal property of localization in $\PrL$ \cite[5.5.4.20]{lurie2008higher} and the application of \cite[Proposition 2.9]{ROBALO2015399} for $X = \cofib(1 \rightarrow \mathbb{G}^\mathrm{log}_m)$. We refer to \cite{ROBALO2015399} for the discussion regarding the monoidal structures, since the construction is similar.
\end{proof}

We can now prove proposition \ref{log linear yoneda}. The proof in \cite[Theorem 3.7]{drew2018motivichodgemodules} is directly adaptable:

\begin{proof}[proof of proposition \ref{log linear yoneda}]
    The adjunction $\CAlg(\PrLotimes) \rightleftarrows \CAlg(\PrLotimes)^{\Mod_\Lambda^\otimes/}$ given by $-\otimes \Mod_\Lambda$ and the forgetful functor, together with the equivalence of lemma \ref{da sh} show that there is a homotopy commutative diagram \begin{center}
        \begin{tikzcd}
            \operatorname{Map}_{\CAlg(\PrLotimes)^{\Mod_\Lambda^\otimes/}}(\mathcal{DA}(k,\Lambda)^\otimes,\mathcal{D}^\otimes) \arrow[dd,"\sim"] \arrow[rd,bend left] & \\
             & \operatorname{Map}_{\CAlg(\Catinf^\times)}(\mathrm{N}(\mathbf{lSm}_k)^\times,\mathcal{D}^\otimes)_{\mathrm{P}^\mathrm{log}}\\
            \operatorname{Map}_{\CAlg(\PrLotimes)}(\mathcal{SH}(k)^\otimes, \mathcal{D}^\otimes) \arrow[ur, bend right]&
        \end{tikzcd}
    \end{center} 
   Since the bottom arrow is a homotopy equivalence by lemma \ref{linear yoneda}, this diagram readily gives the result.
\end{proof}

\begin{rem}
    In the previous proposition, the dividing Nisnevich topology can be straightfowardly replaced by the dividing étale topology, inducing an equivalence $$\operatorname{Map}_{\CAlg(\PrLotimes)^{\Mod_\Lambda^\otimes/}}(\mathcal{DA}_\text{ét}^\mathrm{log}(k,\Lambda)^\otimes,\Mod_\Lambda^\otimes) \xrightarrow{\sim} \operatorname{Map}_{\CAlg(\Catinf^\times)}(\mathrm{N}(\mathbf{lSm}_k)^\times,\Mod_\Lambda^\otimes)_{\mathrm{P}^\mathrm{log}_\text{ét}}$$
    Where we added strict étale descent in the subscript. Note that this equivalence is compatible with the previous construction through the étale sheafification functor, that is the following diagram commutes:
    \begin{center}
        \begin{tikzcd}
            \operatorname{Map}_{\CAlg(\PrLotimes)_{\Mod_\Lambda^\otimes/}}(\mathcal{DA}_\text{ét}^\mathrm{log}(k,\Lambda)^\otimes,\Mod_\Lambda^\otimes) \arrow[d] \arrow[r,"\sim"] & \operatorname{Map}_{\CAlg(\Catinf^\times)}(\mathrm{N}(\mathbf{lSm}_k)^\times,\Mod_\Lambda^\otimes)_{\mathrm{P}^\mathrm{log}_\text{ét}} \arrow[d,hook]\\
            \operatorname{Map}_{\CAlg(\PrLotimes)_{\Mod_\Lambda^\otimes/}}(\DAlog(k,\Lambda)^\otimes,\Mod_\Lambda^\otimes)  \arrow[r,"\sim"] & \operatorname{Map}_{\CAlg(\Catinf^\times)}(\mathrm{N}(\mathbf{lSm}_k)^\times,\Mod_\Lambda^\otimes)_{\mathrm{P}^\mathrm{log}}
        \end{tikzcd}
    \end{center}
    This means that any étale log mixed Weil theory $\mathbf{\Gamma}$ gives rise to compatible realizations that we denote by the same name $\mathrm{R}: \mathcal{DA}_{(\text{ét})}^\mathrm{log}(k,\Lambda) \rightarrow \Mod_\Lambda$.
\end{rem}

There is also a natural compatibility between log mixed Weil theories and mixed Weil theories: It is easy to see that the functor $\lambda$ from the adjunction \begin{center}
\begin{tikzcd}
\lambda: \mathbf{Sm}_k \arrow[r,shift left=.5ex]
&
\mathbf{lSm}_k \arrow[l,shift left=.5ex]: \omega
\end{tikzcd}
\end{center}
defined in subsection \ref{log motives} exchanges log mixed Weil theories and mixed Weil theories.

The induced equivalence $$\lambda^*: \DAlog(k,\Lambda) \xrightarrow{\sim} \mathcal{DA}(k,\Lambda)$$ defined by $\lambda^*(M(X)) \simeq M(X\setminus \partial X)$ fits into the following commutative diagram \begin{center}
    \begin{tikzcd}
        \operatorname{Map}_{\CAlg(\PrLotimes)^{\Mod_\Lambda^\otimes/}}(\DAlog(k,\Lambda)^\otimes,\Mod_\Lambda^\otimes) \arrow[d] \arrow[r,"\sim"] & \operatorname{Map}_{\CAlg(\Catinf^\times)}(\mathrm{N}(\mathbf{lSm}_k)^\times,\Mod_\Lambda^\otimes)_{\mathrm{P}^\mathrm{log}} \arrow[d,hook]\\
            \operatorname{Map}_{\CAlg(\PrLotimes)^{\Mod_\Lambda^\otimes/}}(\mathcal{DA}(k,\Lambda)^\otimes,\Mod_\Lambda^\otimes)  \arrow[r,"\sim"] & \operatorname{Map}_{\CAlg(\Catinf^\times)}(\mathrm{N}(\mathbf{Sm}_k)^\times,\Mod_\Lambda^\otimes)_{\mathrm{P}}
    \end{tikzcd}
\end{center}
where the bottom equivalence (and the subscript) is described in \cite[Theorem 3.7]{drew2018motivichodgemodules}.\\

    We add a useful fact, that will allow us to to compute the realization of the cohomological motives of the log point $*_\mathbb{N}$ which is a non log smooth log scheme over $k$: Let $\mathbf{\Gamma}$ be a log mixed Weil theory, and $\mathrm{R}$ a realization functor associated with $\mathbf{\Gamma}$. Since $k$ is of characteristic $0$, compact objects in $\DAlog(k,\Lambda)$ are dualizable so that since $R$ is symmetric monoidal, for all $X\in \mathbf{lSm}_k$, \begin{center}$\begin{array}{ccl}
        \mathrm{R}(h(X)) & \simeq &  \mathrm{R}(\underline{\mathrm{Hom}}(M(X),\mathbf{1}_k))\\
         & \simeq & \mathrm{R}(M(X)^\vee)\\
         & \simeq & \mathrm{R}(M(X))^\vee\\
         & \simeq &  \mathbf{\Gamma}(X).
    \end{array}$\end{center}
This property can be generalized a little bit: Suppose $\mathbf{\Gamma}$ is defined as a functor $\mathbf{lSch}^\mathrm{op}_k \rightarrow \Mod_\Lambda$. And suppose it is such that the inclusion $*_\mathbf{N} \hookrightarrow \mathbb{A}_\mathbb{N}$ induces an equivalence $$\mathbf{\Gamma}(\mathbb{A}_\mathbb{N}) \xrightarrow[]{\sim} \mathbf{\Gamma}(*_\mathbf{N})$$ In this setting,
\begin{lem}\label{log pt real}
   $$\mathrm{R}(h(*_\mathbb{N})) \simeq \mathbf{\Gamma}(*_\mathbb{N})$$
\end{lem}
\begin{proof}
    $h(*_\mathbb{N})$ is identified as the pushout of the diagram \begin{center}
        \begin{tikzcd}
            h(\mathbb{A}^1) \arrow[d] \arrow[r] & h(\mathbb{A}_\mathbb{N})\\
            h(*) & 
        \end{tikzcd}
    \end{center}
    by the discussion following proposition \ref{cdh}. This stays true after applying $\mathrm{R}$, since it is a colimit preserving functor. Since the following diagram is a pushout diagram by hypothesis,  
    \begin{center}
        \begin{tikzcd}
            \mathbf{\Gamma}(\mathbb{A}^1) \arrow[d] \arrow[r] & \mathbf{\Gamma}(\mathbb{A}_\mathbb{N}) \arrow[d]\\
            \mathbf{\Gamma}(*) \arrow[r] & \mathbf{\Gamma}(*_\mathbb{N}) 
        \end{tikzcd}
    \end{center}
    the claim is proved.
\end{proof}

\subsection{The Betti/Kato--Nakayama realization}\label{Bettikato}\;\smallskip

We can now define the Betti fundamental groupoid at tangential basepoint, and after having defined the Betti realization for log motives, compare the Betti realization of the motivic fundamental groupoid based at tangential basepoints with the classical prounipotent completion of the fundamental groupoid at tangential basepoints. We achieve this in the theorem \ref{thm Betti real pi}. In this part, we suppose $k\subset \mathbb{C}$.

\subsubsection{The topological fundamental groupoid at tangential basepoints}\;\smallskip

We first define the torsor of homotopy classes of paths based at tangential basepoints on a general smooth scheme over $k$ with a simple normal crossings divisor. The definition is a higher dimensional version of what is done in \cite[Section 3.8]{Javier} in the case of smooth projective curves minus a finite number of points.

 \begin{defi}
     Let $(\underline{X},\underline{D}) \in \mathbf{Div}_k$ with $\mathbf{x}:= (x,v), \mathbf{y}:= (y,w)$ be tangential basepoints. The path space of $\underline{X}(\mathbb{C})\setminus\underline{D}(\mathbb{C})$ based at $\mathbf{x},\mathbf{y}$ is the space of smooth maps $\gamma: [0,1] \rightarrow \underline{X}(\mathbb{C})$ such that: \begin{itemize}
    \item $\gamma$ sends the open interval $(0,1)$ to $\underline{X}(\mathbb{C})\setminus \underline{D}(\mathbb{C})$,
    \item $\gamma(0) = x$ and $\gamma(1) = y$,
    \item $\gamma'(0) = v$ and $\gamma'(1) = -w$.
\end{itemize}
We will call such paths \textbf{smooth paths of} $\underline{X}(\mathbb{C})\setminus \underline{D}(\mathbb{C})$ \textbf{based at} $\mathbf{x},\mathbf{y}$.
\end{defi}
 
The definition of homotopies of paths is adapted to the tangential datum, by generalizing the definition given in \cite[3.8.3]{Javier}:

\begin{defi}
Let $\gamma_1,\gamma_2$ be two smooth paths of $\underline{X}(\mathbb{C}) \setminus \underline{D}(\mathbb{C})$ based at tangential basepoints $\mathbf{x},\mathbf{y}$. A homotopy between $\gamma_1$ and $\gamma_2$ is a smooth map $$H : [0,1]\times [0,1] \rightarrow \underline{X}(\mathbb{C})$$ such that: \begin{itemize}
    \item $\forall t \in [0,1],$ $H(t,0) = \gamma_1(t)$ and $H(t,1) = \gamma_2(t)$,
    \item $\forall s \in [0,1]$, $H(0,s) = x$ and $H(1,s) = y$,
    \item $\forall s \in [0,1]$, $\frac{\partial H}{\partial t}(0,s) = v$ and $\frac{\partial H}{\partial t}(1,s) = -w$,
    \item $\forall s \in [0,1], \forall t \in (0,1)$, $H(s,t) \in \underline{X}(\mathbb{C}) \setminus \underline{D}(\mathbb{C}).$
\end{itemize}
\end{defi}

We can define the torsor of homotopy classes of paths in this situation:

\begin{defi}\label{enfin}
We call \textbf{torsor of homotopy classes of paths of} $\underline{X}(\mathbb{C}) \setminus \underline{D}(\mathbb{C})$ \textbf{based at} $\mathbf{x},\mathbf{y}$ the quotient of the path space of $\underline{X}(\mathbb{C}) \setminus \underline{D}(\mathbb{C})$ based at $\mathbf{x},\mathbf{y}$ by the equivalence relation generated by homotopies between paths.
\end{defi}
 
As in the classical case, the data of all the torsors of homotopy classes of paths based at tangential basepoints has a structure of a groupoid. We do not describe precisely the operations involved as they are similar to the case of ordinary basepoints. Nevertheless, the product of homotopy classes of paths is carefully defined in \cite[Definition 3.357]{Javier} in the case of $ \mathbb{P}_k^1\setminus \{0,1,\infty\}$. In higher dimension, the definition of composition can be deduced from the former by viewing $\underline{X}(\mathbb{C})$ locally as a product of $\mathbb{A}^1_k$. Recall that for $\gamma_1$ (resp. $\gamma_2$) a smooth path of $\mathbb{P}_\mathbb{C}^1\setminus \{0,1,\infty\}$ based at tangential basepoints $(x,u), (y,v)$ (resp. $(y,v), (z,w)$), the composition of classes $[\gamma_1][\gamma_2]$ is defined locally around $y$ by the following picture:

\medskip

\begin{figure}[!h]
    \centering
\begin{tikzpicture}
    \draw (-2,-2) rectangle (2,2);
    \draw (4,-2) rectangle (8,2);
    \draw[postaction=decorate,decoration={markings,mark=at position 1.4cm with {\arrow[line width=0.3mm]{>}}}] (-2,1.6) arc [start angle=90,delta angle=-90,x radius=2,y radius=1.6];
    \draw[postaction=decorate,decoration={markings,mark=at position 1.65cm with {\arrow[line width=0.3mm]{>}}}] (0,0) arc [start angle=180,delta angle=-90,x radius=2,y radius=1.6];
    \draw[postaction=decorate,decoration={markings,mark=at position 1.4cm with {\arrow[line width=0.3mm]{>}}}] (4,1.6) arc [start angle=90,delta angle=-70.5,x radius=2,y radius=1.6];
    \draw[postaction=decorate,decoration={markings,mark=at position 1.1cm with {\arrow[line width=0.3mm]{>}}}] (6,0)++(80:0.55) arc [start angle=160,delta angle=-70.5,x radius=2,y radius=1.6];
    \draw (6,0)++(80:0.55) arc [start angle=80,delta angle=-340,radius=0.55];
    \draw[postaction=decorate,decoration={markings,mark=at position 0.13cm with {\arrow[line width=0.15mm]{>}}}] (6,0)++(100:0.55) arc [start angle=100,delta angle=-20,radius=0.55];
    \filldraw[black] (0,0) circle (2pt);
    \filldraw[black] (6,0) circle (2pt);
    \draw (0,-3) node{Topological composition of paths};
    \draw (6,-3) node{Continuous path representing $[\gamma_1\gamma_2]$};
\end{tikzpicture}
\label{pt}
\end{figure}

If we take $\gamma$ to be any smoothing of the second composed path, then $$[\gamma_1][\gamma_2] := [\gamma].$$

\subsubsection{The Betti fundamental groupoid at tangential basepoints}\;\smallskip

We now define the Betti fundamental groupoid, and we compute the Betti realization of the motivic fundamental torsor at tangential basepoints. Let $X\in \mathbf{Div}_k$, endowed with tangential basepoints $\mathbf{x},\mathbf{y}$. Recall from section \ref{virtual} that $\mathbf{x}$ and $\mathbf{y}$ induce points on the Kato--Nakayama topological space $X^{\operatorname{KN}}$. Although as we explained, there is not an equivalence between tangential basepoints and points on $X^\mathrm{KN}$, we will still denote by the same letter $\mathbf{x},\mathbf{y}$ the associated points.

The log Betti cohomology of $X$ is by definition $\mathbf{R}\Gamma(X^{\mathrm{KN}},\underline{\mathbb{Q}}) \in \CAlg(\mathcal{M}\mathrm{od}_\mathbb{Q})$ where $\underline{\mathbb{Q}}$ is the constant sheaf with value $\mathbb{Q}$, and $\mathbf{R}\Gamma$ denotes the derived functor of the global section functor. It is endowed with augmentations defined by the points $\mathbf{x},\mathbf{y} \in X^{\mathrm{KN}}.$ Note that what follows works in the same way if $\mathbb{Q}$ is replaced by $\mathbb{Z}$.
\begin{defi}
    The Betti path space of $X$ based at $\mathbf{x},\mathbf{y}$ is defined by $$\+^{}_\mathbf{x}\mathrm{P}^\mathrm{B}_\mathbf{y}X:= \+_\mathbf{x}\Baar_\mathbf{y}\textbf{R}\Gamma(X^{\operatorname{KN}},\underline{\mathbb{Q}}) \in \CAlg(\mathcal{M}\mathrm{od}_\mathbb{Q}).$$
    The Betti fundamental torsor at $\mathbf{x},\mathbf{y}$ of $X$ is defined by $$\pi^\text{B}_1(X,\mathbf{x},\mathbf{y}):= \operatorname{Spec}\operatorname{H}^0(\+^ {}_\mathbf{x}\mathrm{P}^\text{B}_\mathbf{y}X).$$
    \end{defi}
    
\begin{rem}
As before, it follows from theorem \ref{bar hopf}, that the family $\{\+^{}_\mathbf{x}\mathrm{P}^\mathrm{B}_\mathbf{y}X\}_{\mathbf{x},\mathbf{y}}$ is endowed with the structure of a homotopy Hopf algebroid, so that the family $\{\pi^\text{B}_1(X,\mathbf{x},\mathbf{y})\}_{\mathbf{x},\mathbf{y}}$ has the structure of a groupoid. We call it the \textbf{Betti fundamental groupoid}.
\end{rem}

Similarly to the classical setting, the Betti fundamental torsor is compared to the prounipotent completion of the fundamental torsor of $X^{\operatorname{KN}}$, the main tool being Beilinson's theorem. The comparison that we prove is the following:
 \begin{prop}\label{pi_1 Betti}
    There is a natural isomorphism of groupoids given by isomorphisms
    $$\pi_1^{\operatorname{B}}(X,\mathbf{x},\mathbf{y}) \simeq \pi_1^{\operatorname{un}}(\underline{X}(\mathbb{C})\setminus \underline{D}(\mathbb{C}),\mathbf{x},\mathbf{y})$$ for $\mathbf{x},\mathbf{y}$ all tangential basepoints of $X$.
\end{prop}
\begin{proof}
First, let us apply Beilinson's theorem \cite[Proposition 3.4]{ASENS_2005_4_38_1_1_0} to $X^\mathrm{KN}$. This theorem is proved by Deligne and Goncharov in the context of a locally path connected, locally simply connected topological space $Y$ such that its first homotopy group is of finite type, and such that the functor $\mathbf{R}p_*$ for $p: Y^m\times Y^n \rightarrow Y^n$ commutes with fibers for constant sheaves. Since $X^\mathrm{KN}$ is a manifold with boundaries \cite[Remark 1.5.1]{10.2996/kmj/1138044041} with the homotopy type of a complex algebraic variety, and as such is a paracompact locally contractible space with a finitely generated fundamental group, we can use Beilinson's theorem to see that there is a natural isomorphism of schemes
    $$\pi_1^{\operatorname{B}}(X,\mathbf{x},\mathbf{y}) \simeq \pi_1^{\operatorname{un}}(X^{\operatorname{KN}},\mathbf{x},\mathbf{y}).$$
    
We need to compare the prounipotent completion of the fundamental torsor of $X^\mathrm{KN}$ and of $\underline{X}(\mathbb{C})\setminus \underline{D}(\mathbb{C})$. Note that if $x, y \in \underline{X}(k)\setminus \underline{D}(k)$, then it is straightforward that we have $$\pi_1^{\text{B}}(X,x,y) \simeq \pi_1^{\text{un}}(X^{\text{KN}},x,y) \simeq \pi_1^{\text{un}}(\underline{X}(\mathbb{C})\setminus \underline{D}(\mathbb{C}),x,y)$$ since $X^{\text{KN}}$ and $\underline{X}(\mathbb{C})\setminus \underline{D}(\mathbb{C})$ are homotopy equivalent.
This result is still true for tangential basepoints, as we show in the following lemma concluding the proof of this proposition.
\end{proof}

\begin{lem}\label{pi_1 usuel}
    For all tangential basepoints $\mathbf{x},\mathbf{y}$ on $X$, there is a natural isomorphism $$\pi_1(\underline{X}(\mathbb{C})\setminus \underline{D}(\mathbb{C}),\mathbf{x},\mathbf{y}) \xrightarrow{\sim} \pi_1(X^{\operatorname{KN}},\mathbf{x}, \mathbf{y})$$
\end{lem}
\begin{proof}
    The morphism is given by continuation of paths: for $\gamma: (0,1) \rightarrow \underline{X}(\mathbb{C})$ a smooth path of $\underline{X}(\mathbb{C})\setminus \underline{D}(\mathbb{C})$ from $\mathbf{x}$ to $\mathbf{y}$, we can consider the path $\overline{\gamma}: [0,1] \rightarrow X^\mathrm{KN}$ such that $\eta_{|(0,1)} = \gamma$, $\eta(0) = \mathbf{x}, \eta(1) = \mathbf{y}$.
    
To show that this defines an isomorphism between the homotopy groups, we first see that by continuation of homotopies between paths in the same fashion, this operation trivially defines a morphism, $  \pi_1(\underline{X}(\mathbb{C})\setminus \underline{D}(\mathbb{C}),\mathbf{x},\mathbf{y}) \rightarrow \pi_1(X^\mathrm{KN},\mathbf{x},\mathbf{y})$. We now exhibit this morphism as a composition of isomorphisms:

Let $\eta_\mathbf{x}:\; (0,1) \; \rightarrow \underline{X}(\mathbb{C})$ (resp. $\eta_\mathbf{y}$) be any smooth path from $\mathbf{x}$ (resp. $\mathbf{y}$) to any point denoted $x'$ (resp. $y'$) that we suppose to be in $\underline{X}(\mathbb{C})\setminus \underline{D}(\mathbb{C})$.\\

    The morphism we define is given by the following diagram \begin{center}
        \begin{tikzcd}
            \pi_1(\underline{X}(\mathbb{C})\setminus \underline{D}(\mathbb{C}),\mathbf{x},\mathbf{y})  \arrow[d,dashed] \arrow[r,"\eta_\mathbf{x}^{-1}(-)\eta_\mathbf{y}"] & \pi_1(\underline{X}(\mathbb{C})\setminus \underline{D}(\mathbb{C}),x',y') \arrow[d,"\sim"]\\
           \pi_1(X^\mathrm{KN},\mathbf{x},\mathbf{y})  & \pi_1(X^\mathrm{KN},x',y')\arrow[l,"\overline{\eta_\mathbf{x}}(-)\overline{\eta_\mathbf{y}}^{-1}"] 
        \end{tikzcd}
    \end{center}
where the middle isomorphism $$h: \pi_1(\underline{X}(\mathbb{C})\setminus \underline{D}(\mathbb{C}),x',y') \xrightarrow[]{\sim} \pi_1(X^\mathrm{KN},x',y')$$ is induced by the homotopy equivalence $\underline{X}(\mathbb{C})\setminus \underline{D}(\mathbb{C}) \hookrightarrow X^\mathrm{KN}$. The reason that this composition is the continuation of path is simple: take $\gamma$ any smooth path of $\underline{X}(\mathbb{C})\setminus \underline{D}(\mathbb{C})$ from $\mathbf{x}$ to $\mathbf{y}$. The definition of the composition of homotopy classes of smooth paths of $\underline{X}(\mathbb{C})\setminus \underline{D}(\mathbb{C})$ based at tangential basepoints shows that $[\eta_\mathbf{x}^{-1}][\gamma][\eta_\mathbf{y}]$ is represented by a smooth path $\lambda$ which locally around $x$ and $y$ looks like picture (\ref{pt}) and keeps the paths involved unchanged outside of these local pictures. Hence $h([\lambda]) = [\lambda]$, so that the image of $\gamma$ by the composition of morphisms we are studying is the class represented by $\overline{\eta_\mathbf{x}}\lambda\overline{\eta_\mathbf{y}}^{-1}$ which around $x$ and $y$ looks like (\ref{pt}) and similarly keeps everything unchanged outside. It is clear that this path is just homotopic to the continuation of $\gamma$. 
\end{proof}

Hence, the proposition \ref{pi_1 Betti} is proved. This result will be at the center of the computation of the Betti realization of the motivic pointed path space with tangential basepoints. In order to identify this realization, let us first build the Betti realization functor of $\mathbb{Q}$-linear log motives.

\subsubsection{The Betti realization functor}\;\smallskip

The Betti realization we obtain is closed to the construction for the log motivic homotopy category in  \cite{binda2024logarithmicmotivichomotopytheory}. We prove the following:
\begin{prop}\label{betti real}
     There exists a $\mathbb{Q}$-linear Betti realization functor $$\mathrm{R}^\otimes_\mathrm{B}: \mathcal{DA}^{\text{log}}(k,\mathbb{Q})^\otimes \rightarrow \Mod_\mathbb{Q}^\otimes$$ such that for $X \in \mathbf{lSm}_k$, there is a natural equivalence $$\mathrm{R}_\mathrm{B}(h(X)) \simeq \mathbf{R}\Gamma(X^{\mathrm{KN}},\underline{\mathbb{Q}})$$ and such that the following diagram commutes: \begin{center}
\begin{tikzcd}[ column sep = 0 cm]
    \DAlog(k, \mathbb{Q})^{\otimes} \arrow[rr,"\sim"] \arrow[rdd,"\mathrm{R}^\otimes_\mathrm{B}"] && \mathcal{DA}(k, \mathbb{Q})^\otimes \arrow[ddl,"\mathrm{R}^\otimes_\mathrm{B}"]\\
    &&\\
   & \Mod_\mathbb{Q}^\otimes &
\end{tikzcd}        
    \end{center}
\end{prop}

The functor $\mathrm{R}_\mathrm{B}$ on the right is the classical Betti realization functor as defined in \cite{Ayoub_2010}.

In order to show this, we apply the results of subsection \ref{build} to the log mixed Weil theory over $k$ \begin{center}
$\begin{array}{ccccc}
\mathbf{\Gamma}^\otimes_\mathrm{B} &: & \mathbf{lSm}^{op,\times}_S & \to & \Mod_\mathbb{Q}^\otimes.\\
\end{array}$
\end{center}
defined as follows:

\begin{defi}
    The log Betti $\Mod_\mathbb{Q}$-valued log mixed Weil theory over $k$ is the symmetric monoidal functor \begin{center}
$\begin{array}{ccccc}
\mathbf{\Gamma}^\otimes_{\text{B}} &: & \mathbf{lSm}^{op,\times}_S & \to & \Mod_\mathbb{Q}^\otimes.\\
\end{array}$
\end{center} defined as the composition of symmetric monoidal functors $$\mathbf{lSm}_S^{op,\times} \xrightarrow{(-)^{\text{KN}}} \Spc^{op,\times} \xrightarrow{\mathbf{R}\Gamma(-,\underline{\mathbb{Q}})} \Mod_\mathbb{Q}^\otimes.$$
\end{defi}

Note that up to a canonical equivalence, this functor sends $X \in \mathbf{lSm}_k$ to $\operatorname{C}^*(X^{\text{KN}},\mathbb{Q})$ the singular cochain complex of $X^{\text{KN}}$. Hence, it takes values in $\operatorname{Perf}(\mathbb{Q})$ the sub-$\infty$-category of perfect complexes of $\mathbb{Q}$-modules, since the term in each degree of the singular cochain complex is a free $\mathbb{Q}$-module, and $X^{\text{KN}}$ has the homotopy type of a finite dimensional CW-complex. 

\begin{proof}[Proof of proposition \ref{betti real}.]

       The functor $(-)^{\text{KN}}$ is symmetric monoidal, satisfies $\mathrm{dNis}$-descent and sends $\mathbb{G}_m^\mathrm{log}$ to a space homotopically equivalent to $S^1$. It is as well $\operatorname{ver}$-invariant since $X^\mathrm{KN}$ and $\underline{X}(\mathbb{C})\setminus \underline{D}(\mathbb{C})$ are homotopically equivalent \cite[Proposition A.9.2]{binda2021triangulatedcategorieslogarithmicmotives}. 
     
     Since $\mathbf{R}\Gamma(-,\underline{\mathbb{Q}})$ is symmetric monoidal, so is $\mathbf{\Gamma}_{\text{KN}}$. Thus it is a $\mathrm{dNis}$-sheaf, $\text{ver}$-invariant, it sends $\mathbb{G}_m^\mathrm{log}$ on a $\otimes$-invertible object, and it is trivially $\mathbb{A}^1$-invariant.
     Hence, $\mathbf{\Gamma}_{\operatorname{KN}}$ is a $\Mod_\mathbb{Q}$-valued log mixed Weil theory, from which the proposition follows.
\end{proof}

\begin{rem}\label{log pt betti}
    Note that since $*_\mathbb{N}^{\operatorname{KN}}$ and $\mathbb{A}_\mathbb{N}^{\operatorname{KN}}$ are homotopically equivalent, the morphism $$\mathbf{\Gamma}(\mathbb{A}_\mathbb{N}) \rightarrow \mathbf{\Gamma}(*_\mathbb{N})$$ is an equivalence.
\end{rem}

\subsubsection{Computation of the Betti realization of the motivic fundamental groupoid}\;\smallskip

Now that the Betti realization functor is built, we can get back to the motivic fundamental group. We will show the following:
\begin{thm}
\label{thm Betti real pi}
Let $X \in \mathbf{Div}_k$ and $\mathbf{x}, \mathbf{y}$ tangential basepoints on $X$. There is a natural equivalence of algebras
    $$\mathrm{H}^0(\mathrm{R}_\mathrm{B}(\mpspace{\mathbf{x}}{m}{\mathbf{y}})) \simeq \mathcal{O}(\pi^\mathrm{un}_1(\underline{X}(\mathbb{C})\setminus \underline{D}(\mathbb{C}), \mathbf{x},\mathbf{y})).$$
    In particular, when $X$ is such that $M(X)$ is a mixed Tate motive, there is a natural isomorphism of groupoid induced by $$\mathrm{R}_\mathrm{B}\pi_1^m(X,\mathbf{x},\mathbf{y}) \simeq \pi^{\operatorname{un}}_1(\underline{X}(\mathbb{C})\setminus \underline{D}(\mathbb{C}),\mathbf{x},\mathbf{y}).$$
\end{thm}

\begin{proof}
By the proposition \ref{pi_1 Betti}, we only need to show that there is a natural equivalence in $\CAlg(\mathcal{M}\mathrm{od}^\otimes_\mathbb{Q})$ $$\mathrm{R}_{\mathrm{B}}(\+^{}_\mathbf{x}\mathrm{P}_{\mathbf{y}}^mX) \simeq \+^{}_\mathbf{x}\mathrm{P}^\mathrm{B}_{\mathbf{y}}X.$$
Since $\mathrm{R}^\otimes_{\mathrm{B}}$ is colimit preserving, it commutes with the bar construction, hence $$\mathrm{R}_{\mathrm{B}}(\+^{}_\mathbf{x}\mathrm{P}_{\mathbf{y}}^mX) \simeq \!_{\mathrm{R}_{\mathrm{B}}(\mathbf{x})}\Baar_{\mathrm{R}_{\operatorname{B}}(\mathbf{y})}\mathbf{R}\Gamma(X^{\operatorname{KN}},\underline{\mathbb{Q}}).$$
    Recall that up to homotopy equivalence, $h(\mathbf{x}): h(X) \rightarrow \mathbb{Q}$ is defined by the choice of a homotopy inverse of $$h(\mathrm{N}^\mathrm{log}_x(\underline{X},\underline{D})) \rightarrow h(x^\mathrm{log}).$$ The realization of this morphism is $$C^*(\mathrm{N}_x^\text{log}(\underline{X},\underline{D})^{\operatorname{KN}},\mathbb{Q}) \rightarrow C^*((x^\text{log})^{\operatorname{KN}},\mathbb{Q})$$ (see lemma \ref{log pt real} and remark \ref{log pt betti}) induced by the embedding $$(x^\mathrm{log})^{\operatorname{KN}} \hookrightarrow \mathrm{N}_x^\mathrm{log}(\underline{X},\underline{D})^{\operatorname{KN}}.$$ A natural homotopy inverse to this morphism is the map induced by the retraction of the real blow-up of $\mathbb{C}^i$ along its axes on the product of infinitesimal spheres at the origin $(S^1)^i$. Therefore, the augmentation $\mathrm{R}_{\mathrm{B}}(h(\mathbf{x}))$ is up to homotopy the map obtained from the sequence $$* \xrightarrow{v} \mathrm{N}_x^\mathrm{log}(\underline{X},\underline{D})^{\mathrm{KN}} \rightarrow (S^1)^i \rightarrow X^{\operatorname{KN}}$$ by applying $C^*(-,\mathbb{Q})$. The composition of these maps is just the point $$* \rightarrow X^{\operatorname{KN}}$$ defined by $\frac{v}{||v||} \in S^i = (x^\mathrm{log})^{\operatorname{KN}}$. The theorem is proved.
\end{proof}

We identified the Betti realization of the motivic pointed path space. Let us now relate it to the cosimplicial path space construction. Passing to the Betti cohomology allows us to bypass the limitations of proposition \ref{path space facile}. Indeed, we can see that the Betti path space can be classically extracted from the cosimplicial space \begin{center}
    \begin{tikzcd}
        \+^{}_\mathbf{x}\mathrm{P}^\bullet_\mathbf{y}X^{\operatorname{KN}} \arrow[r] \arrow[d] \arrow[dr, phantom, "\usebox\pullback" , very near start, color=black] & (X^{\operatorname{KN}})^{[0,1]} \arrow[d]\\
        * \arrow[r,"\mathbf{y,x}"] & X^{\operatorname{KN}}\times X^{\operatorname{KN}}
    \end{tikzcd}
\end{center}
after applying the functor $\text{R}\Gamma(-,\underline{\mathbb{Q}})$ and taking the colimit on the simplicial diagram thereby obtained. Let us explain this point further. Is is easy to see that $$\+^{}_\mathbf{x}\mathrm{P}^\mathrm{B}_\mathbf{y}X \simeq \mathbb{Q} \otimes_{\mathbf{\Gamma}_\mathrm{B}(X)\otimes \mathbf{\Gamma}_\mathrm{B}(X)} \mathbf{\Gamma}_\mathrm{B}(X) \simeq \mathbb{Q} \otimes_{\mathrm{R}\mathbf{\Gamma}(X^\mathrm{KN},\underline{\mathbb{Q}})\otimes \mathrm{R}\mathbf{\Gamma}(X^\mathrm{KN},\underline{\mathbb{Q}})} \mathrm{R}\mathbf{\Gamma}(X^\mathrm{KN},\underline{\mathbb{Q}}).$$ Since $\mathrm{R}\mathbf{\Gamma}(X^\mathrm{KN},\underline{\mathbb{Q}}) \simeq \mathrm{colim}_\Delta \mathrm{R}\mathbf{\Gamma}((X^{\mathrm{KN}})^{[0,1]},\underline{\mathbb{Q}})$, we see that $$\+^{}_\mathbf{x}\mathrm{P}^\mathrm{B}_\mathbf{y}X \simeq \colim_\Delta \mathbb{Q} \otimes_{\mathrm{R}\mathbf{\Gamma}(X^\mathrm{KN},\underline{\mathbb{Q}})\otimes \mathrm{R}\mathbf{\Gamma}(X^\mathrm{KN},\underline{\mathbb{Q}})} \mathrm{R}\mathbf{\Gamma}((X^{\mathrm{KN}})^{[0,1]},\underline{\mathbb{Q}}).$$ It is now straightforward to see that the simplicial objects $\mathbf{\Gamma}_\mathrm{B}(\displaystyle \+^{}_\mathbf{x}\mathrm{P}^\bullet_\mathbf{y}X^{\mathrm{KN}},\underline{\mathbb{Q}})$ and $\mathbb{Q} \otimes_{\mathrm{R}\mathbf{\Gamma}(X^\mathrm{KN},\underline{\mathbb{Q}})\otimes \mathrm{R}\mathbf{\Gamma}(X^\mathrm{KN},\underline{\mathbb{Q}})} \mathrm{R}\mathbf{\Gamma}((X^{\mathrm{KN}})^{[0,1]},\underline{\mathbb{Q}})$ are weakly equivalent, so that $$\+^{}_\mathbf{x}\mathrm{P}^\mathrm{B}_\mathbf{y}X \simeq \colim_\Delta \mathbf{\Gamma}_\mathrm{B}( \+^{}_\mathbf{x}\mathrm{P}^\bullet_\mathbf{y}X^{\operatorname{KN}},\underline{\mathbb{Q}}).$$ By \cite[Corollary 3.13]{arakawa2023homotopylimitshomotopycolimits}, we see that there is a natural quasi-isomorphism $$\+^{}_\mathbf{x}\mathrm{P}^\text{B}_\mathbf{y}X \simeq \operatorname{Tot}\textbf{R}\Gamma(\+^{}_\mathbf{x}\mathrm{P}^\bullet_\mathbf{y}X^{\operatorname{KN}},\underline{\mathbb{Q}}).$$

The previous quasi-isomorphism has a filtered version: $$\+^{}_\mathbf{x}\mathrm{P}^\text{B}_\mathbf{y}X_{\leq n} \simeq \operatorname{Tot}\textbf{R}\Gamma(\+^{}_\mathbf{x}\mathrm{P}^\bullet_\mathbf{y}X^{\operatorname{KN}}_{\leq n},\underline{\mathbb{Q}}),$$ where the object $\!_\mathbf{x}\mathrm{P}^\bullet_\mathbf{y}X^{\operatorname{KN}}_{\leq n}$ is the diagram obtained from $\!_\mathbf{x}\mathrm{P}^\bullet_\mathbf{y}X^{\operatorname{KN}}$ by restriction to the sets of cardinality $\leq n$. Based on this result, we will now describe in a closer way to the classical point of view, the Betti realization of the filtered pieces of the path space based at tangential basepoints. This will be useful to compute the periods of the motivic fundamental groupoid. The main reference here is \cite[Section 3]{Javier}, see also \cite{enriquez2023naturaltransformationsrelatinghomotopy,enriquez2021homologicalinterpretationprounipotentcompletion}.

\begin{defi}
    Let $n$ be an nonegative integer, $\mathbf{x}, \mathbf{y}$ two tangential points on $X$ a divisorial log scheme over $k$. For $0 \leq i \leq n$, we denote by $\underline{Y_i}$ the image in $\underline{X}^n$ of the i-th coface map (which is a virtual morphism of log schemes when $i = 0,n$) of the cosimplicial based path space $\+_\mathbf{x}\mathrm{P}_\mathbf{y}X$ $$\partial_i^n: X^{n-1} \rightarrow X^n.$$ Note that $\underline{\partial^n_i}$ is an immersion of schemes, so that we can endow $\underline{Y_i}$ with the log structure induced by the isomorphisms $\underline{X}^{n-1} \simeq \underline{Y_i}$. It comes equipped with a virtual morphism $$Y_i \rightarrow X^n.$$
\end{defi}

\begin{rem}
    The scheme $\underline{Y_i}$ can be described explicitly. If we write $X^n$ in coordinates $(x_1,\dots, x_n)$, then  $\underline{Y_i}$ is \begin{itemize}
        \item $\{x_1 = y \}$ if $i = 0$,
        \item $\{x_i = x_{i+1}\}$ if $0 < i < n$,
        \item $\{x_n = x\}$ if $i = n$.
    \end{itemize} \end{rem}

 Naturally, the virtual morphism $Y_i \rightarrow X^n$ is \begin{itemize}
        \item The product of the virtual point $\mathbf{y}$ and the identity on $X^{n-1}$ if $i = 0$.
        \item The ordinary morphism of log schemes induced by the i-th coface of $\!_\mathbf{x}\mathrm{P}_\mathbf{y}X$.
        \item The product of the identity on $X^{n-1}$ and the virtual point $\mathbf{x}$ if $i= n$.
    \end{itemize}

\smallskip

For $\mathrm{I} \subsetneq \{0,\dots, n\}$, we let $\underline{Y_\mathrm{I}} = \bigcap_{i\in \mathrm{I}} \underline{Y_i}$. If we denote by $\mathrm{I}^c:= \Delta^n\setminus \mathrm{I}$, there are isomorphisms $$\underline{Y_\mathrm{I}} \simeq \underline{X}^{|\mathrm{I}^c|-1},$$ obtained by deleting redundant coordinates, which induces a log structure on $\underline{Y_\mathrm{I}}$.

\smallskip 

For $i \in \mathbb{N}$, we define $$Y^i:= \bigsqcup_{\mathrm{I} \subsetneq \{0,\dots,n\}, |\mathrm{I}| = i}Y_\mathrm{I}$$
and we denote by $(X^n,Y^\bullet)$ the family thus defined. Note that for $i>0$, there are $i$ natural morphisms \begin{center}
\begin{tikzcd}
    Y^i  \arrow[r, draw=none, "\raisebox{+1.5ex}{\vdots}" description] \arrow[r, shift left = 3] \arrow[r, shift right = 3] & Y^{i-1} 
\end{tikzcd}
\end{center}
given by the diagram for $\mathrm{I} \subsetneq \{0,\dots,n\}$, and $k \in \mathrm{I}$ \begin{center}
    \begin{tikzcd}
    Y_\mathrm{I} \arrow[r, "\sim"] \arrow[d,dotted] &  X^{|\mathrm{I}^c|-1} \arrow[d]\\
     Y_{\mathrm{I}\setminus \{k\}} \arrow[r,"\sim"] & X^{|(\mathrm{I}\setminus \{k\})^c|-1} = X^{|\mathrm{I}^c|}
    \end{tikzcd}
\end{center}
where the right morphism is the $k$-th coface of $\!_\mathbf{x}\mathrm{P}_\mathbf{y}X$.

\smallskip

These morphisms endow $(X^n,Y^\bullet)$ with a natural structure of an augmented symmetric semi-simplicial object in $\!^{\mathbf{v}}\mathbf{Div}_k$, that is a diagram in $\!^{\mathbf{v}}\mathbf{Div}_k$ indexed by the category of finite sets, with injective morphisms between them (see \cite[Section 3]{Chan_2021}). 

\begin{defi}\label{finfin}
We let $\mathbf{\Gamma}_\mathrm{B}(X^n,Y^\bullet)$ be the complex in $\mathcal{M}\mathrm{od}_\mathbb{Q}$ associated with the augmented symmetric semi-simplicial variety with log corners $(X^n,Y^\bullet)$ after applying $\mathbf{\Gamma}_\mathrm{B}^\otimes$.
\end{defi}

\begin{rem}
If we denote by $Y:= \cup_i Y_i$, then it is a classical fact (see \cite{ASENS_2005_4_38_1_1_0}) that the complex $\mathbf{\Gamma}_\mathrm{B}(X^n,Y^\bullet)$ is closely related to the relative Betti cohomology of the pair $(X^n,Y)$, that we denote by $\mathbf{\Gamma}_\mathrm{B}(X^n,Y)$. More precisely, there is a fiber sequence $$\underline{\mathbb{Q}}_{|Y_{\Delta^n}}[-n-1] \rightarrow \mathbf{\Gamma}_\mathrm{B}(X^n,Y) \rightarrow \mathbf{\Gamma}_\mathrm{B}(X^n,Y^\bullet).$$ Note that $Y_{\Delta^n}$ is either a point if $\mathbf{x} = \mathbf{y}$, or the empty set if $\mathbf{x} \neq \mathbf{y}$. Hence the fiber sequence induces an isomorphism $$\mathbf{\Gamma}_\mathrm{B}(X^n,Y) \simeq \mathbf{\Gamma}_\mathrm{B}(X^n,Y^\bullet)$$ when $\mathbf{x} \neq \mathbf{y}$.
\end{rem}

In general, the complex $\mathbf{\Gamma}_\mathrm{B}(X^n,Y^\bullet)$ is used to compute the filtered pieces of the bar construction. Indeed, with \cite[3.297, Lemma 3.298]{Javier} we know that there are natural quasi-isomorphisms of complexes $$ \+^{}_\mathbf{x}\mathrm{P}^\mathrm{B}_\mathbf{y}X_{\leq n} \simeq \mathbf{\Gamma}_\mathrm{B}(\+^{}_\mathbf{x}\mathrm{P}^\bullet_\mathbf{y}X_{\leq n}) \simeq  \mathbf{\Gamma}_\mathrm{B}(X^n,Y^\bullet)[n].$$

The quasi-isomorphisms can be described this way: The complex $\mathbf{\Gamma}_\mathrm{B}(X^n,Y^\bullet)$ can be represented by the total complex of the double complex \begin{equation}\label{complex 1}
    \bigoplus_{\mathrm{I} \subsetneq \{0,\dots, n\},\\ |\mathrm{I}| = \bullet} C^*(Y_\mathrm{I}^\mathrm{KN},\mathbb{Q}),
\end{equation} where the horizontal differentials are given by the semi-simplicial structure of $(X^n,Y^\bullet)$ (see a full description in \cite[A.8.3]{Javier}).

The complex $\mathbf{\Gamma}_\mathrm{B}(\+^{}_\mathbf{x}\mathrm{P}^\bullet_\mathbf{y}X_{\leq n})$ is obtained as the total complex of the double complex \begin{equation}\label{complex 2}
    C^*(\+^{}_\mathbf{x}\mathrm{P}^\bullet_\mathbf{y}X^\mathrm{KN}_{\leq n})
\end{equation} where the horizontal differentials are given by the cosimplicial structure of $\+^{}_\mathbf{x}\mathrm{P}^\bullet_\mathbf{y}X^\mathrm{KN}$.

\smallskip

The complexes (\ref{complex 1}) and (\ref{complex 2}) are isomorphic up to a shift of degree $n$, and the isomorphism is given explicitely in degree $k$ by: $$\begin{array}{ccc}
     \bigoplus_{0 \leq i \leq n} C^{k+i}((X^\mathrm{KN})^i,\mathbb{Q}) & \rightarrow & \bigoplus_{0 \leq i \leq n} \bigoplus_{\mathrm{I} \subsetneq \{0,\dots, n\}, |\mathrm{I}| = i} C^{k+n-i}(Y_\mathrm{I}^\mathrm{KN},\mathbb{Q}) \\
   (\omega_i)_{0 \leq i \leq n}  & \mapsto  & (\omega_\mathrm{I})_{\mathrm{I} \subsetneq \Delta^n}
\end{array}$$ where $(\omega_\mathrm{I})$ is defined by $$\omega_\mathrm{I} = \left\{
    \begin{array}{ll}
        (-1)^{n|\mathrm{I}^c|}\epsilon(\mathrm{I}^c)\omega_{i} & \mbox{if } \mathrm{I} = \{0,\dots,i\}\\
        0 & \mbox{else}
    \end{array} \right. $$
where we let $\epsilon(\mathrm{I}):= \prod_{i \in \mathrm{I}} (-1)^i$, and where $\omega_i = \omega_{|\mathrm{I}^c|-1}$ as a cochain on $(X^\mathrm{KN})^{|\mathrm{I}^c|-1}$, is viewed as a cochain on $Y_\mathrm{I}^\mathrm{KN}$.

\begin{rem}\label{signe}
The signs appearing above are conventions, and they are constrained by the signs chosen for the differentials of the complexes involved. We followed here the conventions from \cite{Javier}.
\end{rem}

Using this representation, the isomorphism given by the Beilinson lemma is classically described as follows: \\

For $i\in \mathbb{N}$, let us denote by \begin{equation}\label{model}
    \Delta^i:= \{(t_1,\dots,t_i)\in \mathbb{R}^i, 0 \leq t_i \leq \dots \leq t_1 \leq 1\}
\end{equation} the standard $i$-simplex. Let $\gamma: [0,1] \rightarrow X^\mathrm{KN}$ be a continuous path based at $\mathbf{x},\mathbf{y}$. It induces a singular chain \begin{center}
        $\begin{array}{cccc}
        \sigma_{\gamma,i}:     &  \Delta^i & \rightarrow & (X^\mathrm{KN})^i\\
             & (t_1,\dots, t_i) & \mapsto  & (\gamma(t_1),\cdots,\gamma(t_i))
        \end{array}$
    \end{center}
    
Moreover, for $\mathrm{I} \subsetneq \{0,\dots, n \}$, $\sigma_{\gamma,|\mathrm{I}^c|-1}$ defines a map $$\sigma_{\gamma,\mathrm{I}}: \Delta^{|\mathrm{I}^c|-1} \rightarrow (X^\mathrm{KN})^{|\mathrm{I}^c|-1} \simeq Y^\mathrm{KN}_\mathrm{I}$$

\begin{prop}\cite[Example 3.8]{ASENS_2005_4_38_1_1_0},\cite[Section 3.6.4]{Javier}
    The isomorphism $$\mathrm{R}^0_\mathrm{B}(\mpspace{\mathbf{x}}{m}{\mathbf{y}})_{\leq n} \simeq  (\mathbb{Q}[\pi_1(X^\mathrm{KN},\mathbf{x},\mathbf{y})]/I^{n+1})^\vee$$
    is dual to the map obtained via $$\begin{array}{ccc}
        \pi_1(X,\mathbf{x},\mathbf{y}) & \rightarrow & (\mathbf{\Gamma}_\mathrm{B}(X^n,Y^\bullet)^n)^\vee\\
        \gamma & \mapsto & ((-1)^\frac{|\mathrm{I}|(|\mathrm{I}|+1)}{2}(-1)^\frac{n(n-1)}{2}\epsilon(\mathrm{I}^c)\sigma_{\gamma,\mathrm{I}})_{\mathrm{I}\subsetneq [n]}.  \end{array}$$
\end{prop}
 \begin{rem}
 The signs are again conventions fixed in a coherent way with the conventions chosen above.
 \end{rem}

\subsection{The de Rham realization}\label{deRhamreal2}
\subsubsection{The de Rham fundamental groupoid at tangential basepoints}\;\smallskip

We now give a definition of the de Rham fundamental torsor at tangential basepoints, and compare it with the de Rham realization of the motivic pointed path space based at tangential basepoints. We suppose $k$ of characteristic $0$ in this section. Let us take again $X\in \mathbf{Div}_k$, with tangential basepoints $\mathbf{x},\mathbf{y}$. We will emphasize the point of view, taken from section \ref{virtual}, on $\mathbf{x},\mathbf{y}$ as virtual morphisms $* \rightarrow X$.\\
Recall that the sheaf of log de Rham differentials $\Omega^1_{-/k}$ is functorial with respect to virtual morphism with reduced source \cite[Lemma 5.9]{dupont2024logarithmicmorphismstangentialbasepoints}. Therefore the log de Rham cohomology $\textbf{R}\Gamma(X,\Omega_{X/k}) \in \operatorname{\CAlg}(\Mod_k)$ comes with canonical augmentations defined by functoriality $$\textbf{R}\Gamma(X,\Omega_{X/k}) \xrightarrow{\mathbf{x,y}} k.$$
\begin{defi}\label{larefimportante}
    The de Rham path space of $X$ based at $\mathbf{x},\mathbf{y}$ is defined by $$\+^{}_\mathbf{x}\mathrm{P}^\text{dR}_\mathbf{y}X:= \!_\mathbf{x}\Baar_\mathbf{y}\textbf{R}\Gamma(X,\Omega_{X/k}) \in \Mod_k.$$
    The de Rham fundamental torsor at $\mathbf{x},\mathbf{y}$ of $X$ is defined by $$\pi^\text{dR}_1(X,\mathbf{x},\mathbf{y}):= \operatorname{Spec}\operatorname{H}^0(\+^{}_\mathbf{x}\mathrm{P}^\text{dR}_\mathbf{y}X).$$
\end{defi}

\begin{rem}
As before, it follows from theorem \ref{bar hopf}, that the family $\{\+^{}_\mathbf{x}\mathrm{P}^\mathrm{dR}_\mathbf{y}X\}_{\mathbf{x},\mathbf{y}}$ is endowed with the structure of a homotopy Hopf algebroid, so that the family $\{\pi^\text{dR}_1(X,\mathbf{x},\mathbf{y})\}_{\mathbf{x},\mathbf{y}}$ has the structure of a groupoid. We call it the \textbf{de Rham fundamental groupoid}.
\end{rem}

\subsubsection{The de Rham realization functor}\;\smallskip

We can now move to the identification of the de Rham realization of the motivic pointed path space based at tangential basepoints. Let us first build the log de Rham realization functor.

\begin{prop}\label{de rham real}
 There exists a log de Rham realization functor $$\mathrm{R}^\otimes_\mathrm{dR}: \DAlog(k,k)^\otimes \rightarrow \Mod_k^\otimes$$ such that for $X \in \mathbf{lSm}_k$, there is a natural equivalence $$\mathrm{R}_\mathrm{dR}(h(X)) \simeq \mathbf{R}\Gamma(X,\Omega_{X/k})$$ and such that the following diagram commutes: \begin{center}
\begin{tikzcd}[ column sep = 0 cm]
    \DAlog(k, k)^\otimes \arrow[rr,"\sim"] \arrow[rdd,"\mathrm{R}^\otimes_\mathrm{dR}"] && \mathcal{DA}(k, k)^\otimes\arrow[ddl,"\mathrm{R}^\otimes_\mathrm{dR}"]\\
    &&\\
   & \Mod^\otimes_k &
\end{tikzcd}        
    \end{center}
\end{prop}
\begin{defi} \label{deRham}
    The log de Rham $\Mod_k$-valued log mixed Weil theory over $k$ is the symmetric monoidal functor 
    \begin{center}
    $\begin{array}{ccccc}
\mathbf{\Gamma}^\otimes_{\operatorname{dR}} &: & \mathbf{lSm}^{op,\times}_k & \to & \Mod_k^\otimes\\
& & X & \mapsto & \mathbf{R}\Gamma(X,\Omega_{X/k})
\end{array}$
\end{center}
\end{defi}
This indeed defines a symmetric monoidal $\infty$-functor: from \cite[Proposition 3.5]{blumberg2015uniquenessmultiplicativecyclotomictrace}, we see that it suffices to show that the functor of 1-categories $$\begin{array}{ccccc}
  \mathbf{\Gamma}&: &\mathbf{lSm}_k^\mathrm{op}  & \rightarrow & \mathbf{D}(\mathbf{Mod}_k)  \\
    & & X  & \mapsto & \mathbf{R}\Gamma(X,\Omega_{X/k})
\end{array}  $$
where $\mathbf{R}\Gamma(-,\Omega_{-/k})$ denotes by an abuse of notation a symmetric monoidal resolution of complexes of sheaves functor (e.g. the Thom-Whitney normalization of the Godement resolution \cite{Nava}) is such that \begin{itemize}
    \item it is lax symmetric monoidal,
    \item $k \rightarrow \mathbf{\Gamma}(\mathrm{Spec}(k))$ is a quasi-isomorphism,
    \item $\mathbf{\Gamma}(X) \otimes \mathbf{\Gamma}(Y) \rightarrow \mathbf{\Gamma}(X\times Y)$ is a quasi-isomorphism for all $X,Y \in \mathbf{lSm}_k.$
\end{itemize} These conditions are straightforward properties of the log de Rham cohomology. The proposition cited above defines a symmetric monoidal functor of $\infty$-categories.

\begin{proof}[Poof of proposition \ref{de rham real}.]

    The functor $\mathbf{\Gamma}_{\operatorname{dR}}$ is $\text{ver}$-invariant by \cite[Proposition 4.2.5]{Ogus_2018}, hence it satisfies $\mathrm{dNis}$-descent ($\omega: X \mapsto X\setminus \partial X$ sends $\mathrm{dNis}$ covers to $\mathrm{Nis}$ covers) since the classical de Rham cohomology satisfies Nisnevich descent. Moreover, $\mathbf{\Gamma}_\mathrm{dR}(\mathbb{G}_m^\mathrm{log}) \simeq \mathbf{\Gamma}_\mathrm{dR}(\mathbb{G}_m)$ so that it sends $\mathbb{G}^\mathrm{log}_m/1$ on a $\otimes$-invertible object, and naturally $ \mathbf{\Gamma}_\mathrm{dR}(\mathbb{A}^1)\simeq k$.\\ 
    Finally, it takes values in $\mathrm{Perf}(k)$ since the de Rham cohomology has finite dimension.
\end{proof}

\begin{rem}
    Since the log de Rham cohomology of the log point $*_\mathbb{N}$ and of the log scheme $\mathbb{A}_\mathbb{N}$ are computed by applying the functor of global sections of the corresponding log de Rham complexes, it is easy to see that the morphism $$\mathbf{\Gamma}_\mathrm{dR}(\mathbb{A}_\mathbb{N})\rightarrow\mathbf{\Gamma}_\mathrm{dR}(*_\mathbb{N})$$ is an equivalence.
\end{rem}

\subsubsection{Computation of the de Rham realization of the motivic fundamental groupoid}\;\smallskip

Let us now compute the realization of the motivic pointed path space:

\begin{thm}\label{deRhamreal}
Let $X\in \mathbf{Div}_k.$ For $\mathbf{x}, \mathbf{y}$ tangential basepoints on $X$, there is a natural equivalence
    $$\mathrm{R}_{\mathrm{dR}}(\mpspace{\mathbf{x}}{m}{\mathbf{y}}) \simeq \+^{}_\mathbf{x}\mathrm{P}^\mathrm{dR}_\mathbf{y}X.$$
     In particular, when $X$ is such that $M(X)$ is a mixed Tate motive, there is a natural isomorphism of groupoid induced by $$\mathrm{R}_\mathrm{dR}\pi_1^m(X,\mathbf{x},\mathbf{y}) \simeq \pi^{\operatorname{dR}}_1(X,\mathbf{x},\mathbf{y}).$$
\end{thm}
\begin{proof}
    Once again, using the fact that $\mathrm{R}_{\mathrm{dR}}$ is colimit preserving, it commutes with the bar construction, so that $$\mathrm{R}_{\mathrm{dR}}(\mpspace{\mathbf{x}}{m}{\mathbf{y}}) \simeq \!_{\mathrm{R}_{\mathrm{dR}}(\mathbf{x})}\Baar_{\mathrm{R}_{\mathrm{dR}}(\mathbf{y})}\mathbf{R}\Gamma(X,\Omega_{X/k}).$$ Now, the morphism $\mathrm{R}_\mathrm{dR}(\mathbf{x})$ (similarly for $\mathrm{R}_\mathrm{dR}(\mathbf{y})$) is defined after the choice of a homotopy inverse to the equivalence $$\mathbf{R}\Gamma(\mathrm{N}_x^\mathrm{log}(\underline{X},\underline{D}),\Omega_{\mathrm{N}_x^\mathrm{log}(X,D)/k}) \rightarrow \mathbf{R}\Gamma(x^\mathrm{log},\Omega_{x^\mathrm{log}/k}).$$
    Such a homotopy equivalence is found by functoriality applied to the virtual retraction $$\mathrm{N}_x^\mathrm{log}(\underline{X},\underline{D}) \rightarrow x^\mathrm{log}$$ which is defined trivially on the underlying schemes, and by the obvious morphism of sheaves in monoids $$\mathcal{M}_{X|_{\mathbf{x}}} \rightarrow \mathcal{M}_{\mathrm{N}_{x}^\mathrm{log}(\underline{X},\underline{D})}$$ (note that this is indeed a retraction of the natural morphism $-^\mathrm{log} \rightarrow \mathrm{N}_{\mathbf{x}}^\mathrm{log}(X,D)$) so that the composition $$* \rightarrow \mathrm{N}_x^\mathrm{log}(\underline{X},\underline{D}) \rightarrow x^\mathrm{log} \rightarrow X$$ is nothing but the virtual morphism defining $h(\mathbf{x}),h(\mathbf{y})$.
    This concludes the proof.
\end{proof}

Arguing as in the previous subsection, we see that the de Rham path space can be canonically obtained from the cosimplicial log scheme introduced in subsubsection \ref{The cosimplicial pointed path space} \begin{center}
    \begin{tikzcd}
        \+^{}_\mathbf{x}\mathrm{P}^\bullet_\mathbf{y}X \arrow[r] \arrow[d] \arrow[dr, phantom, "\usebox\pullback" , very near start, color=black] & X^{[0,1]} \arrow[d]\\
        * \arrow[r,"\mathbf{x,y}"] & X\times X
    \end{tikzcd}
\end{center}
after applying the functor $\textbf{R}\Gamma(-,\Omega_{-/k})$ and taking the associated total complex.

\medskip

As in the Betti case, we can go a little further and give a more detailed description of the de Rham fundamental groupoid after extension of the scalar to $\mathbb{C}$. In order to compute the periods of the motivic fundamental groupoid in the next sections, we interpret the de Rham fundamental groupoid via the de Rham theory of \textit{manifolds with log corners}.

\smallskip

 Recall from \cite{dupont2024regularizedintegralsmanifoldslog} that a manifold with log corners $(\Sigma,\mathcal{M}_\Sigma)$ is a manifold with corners $\underline{\Sigma}$ endowed with a \textit{positive log structure} (this notion is developped in \cite{gillam2015logdifferentiablespacesmanifolds}): a sheaf of monoids $\mathcal{M}_\Sigma$ on $\underline{\Sigma}$ and a morphism $\alpha: \mathcal{M}_\Sigma \rightarrow \mathcal{C}^{\infty,\geq 0}_\Sigma$ where $\mathcal{C}^{\infty,\geq 0}_\Sigma$ is the sheaf of non-negative smooth functions on $\underline{\Sigma}$, with the property that it induces an isomorphism $\alpha^{-1}(\mathcal{C}^{\infty,> 0}_\Sigma) \xrightarrow{\sim} \mathcal{C}^{\infty,> 0}_\Sigma$. We furthermore ask $X$ to be locally isomorphic (as a manifodl with corners equipped with a positive log structure) to an open subset of the standard corner $[0)^k\times [0,\infty)^n$.

The latter is defined by the underlying manifold with corners $[0,\infty)^n$ equipped with the sheaf $$\mathcal{M}_{[0)^k\times [0,\infty)^n}:= \mathcal{C}^{\infty,> 0}_{[0,\infty)^n}t_1^\mathbb{N}\dots t_k^\mathbb{N}r_1^\mathbb{N}\dots r_n^\mathbb{N}$$ where the $r_i$ are coordinates on $[0,\infty)^n$ and the $t_i$ are formal ``phantom" coordinates evaluated to $0$ in $\mathcal{C}^{\infty,\geq 0}_{[0,\infty)^n}.$ See \cite[Section 3.2]{dupont2024regularizedintegralsmanifoldslog} for details.
A manifold with log corners is locally obtained as the restriction at a closed submanifold with corners of the manifold with log corners $[0,\infty)^n$ for a non-negative integer $n$ (e.g. $[0) \times [0,\infty)$ is obtained as the restriction of the positive log structure $[0,\infty)^2$ on the line $r_1 = 0$. The coordinate $r_1$ is turned into the phantom coordinate $t_1$).

\begin{exmp}
\begin{itemize}
    \item The manifold with corners $[0,\infty)$ with coordinate $r$ is equipped with the structure of a manifold with log corners via the positive log structure $$\alpha: \mathcal{C}^{\infty,>0}_{[0,\infty)}r^\mathbb{N} \hookrightarrow \mathcal{C}^{\infty,\geq 0}_{[0,\infty)}.$$
    \item The Kato--Nakayama space of the log scheme $\mathbb{A}_\mathbb{N}$ is the oriented real blow-up of $\mathbb{R}^2$ at $0$. It is then endowed with the structure of a manifold with log corners $S^1 \times [0,\infty)$ where $S^1$ is endowed with the trivial positive log structure.
\end{itemize}
\end{exmp}

In general, for $X\in \mathbf{Div}_k$, $X^\mathrm{KN}$ has the structure of a manifold with log corners \cite[Proposition 8.4]{dupont2024regularizedintegralsmanifoldslog}, and virtual points $* \rightarrow X$ define virtual points of $X^\mathrm{KN}$, that is virtual morphisms in the sense of \cite[Definition 4.1]{dupont2024regularizedintegralsmanifoldslog} (the definition is similar to definition \ref{virtualmo}). Such a virtual point keeps track of the whole datum of the tangential basepoint, not only its direction, so that when $X$ is defined over $\mathbb{C}$, there is a bijection between tangential basepoints of $X$ and virtual points of $X^\mathrm{KN}$ \cite[Proposition 8.8]{dupont2024regularizedintegralsmanifoldslog}. Hence, this fills the gap between tangential basepoints on $X$, and points on $X^\mathrm{KN}$ explained at the end of subsection \ref{pointed}.

\begin{exmp}
A virtual morphism from the point to $\mathbb{A_N}^\mathrm{KN}$ landing on the infinitesimal unit circle at $0$ is the same as the choice of a direction $u$ and of the pullback of the phantom coordinate $t$ associated with $r = 0$ in $\mathcal{C}^{\infty,>0}_* \simeq \mathbb{R}_+^*$. The last datum allows to differentiate between tangential basepoints induced by $\lambda$ or $a\lambda$ for $a$ a strictly positive constant, see example \ref{yes}.
\end{exmp}

 A manifold with log corners $\Sigma$ is endowed with a de Rham complex of logarithmic differential forms $\mathcal{A}_\Sigma$, functorial with respect to virtual morphisms (\cite[Section 6.2]{dupont2024regularizedintegralsmanifoldslog}), which adds to the complex of classical differential forms, formal logarithmic differential forms $\mathrm{dlog}(r_i),\mathrm{dlog}(t_j)$. This construction is similar to the construction of the algebraic logarithmic de Rham complex so that for $X \in \mathbf{Div}_k$ there are canonical quasi-isomorphisms (\cite[Sections 8.5,6]{dupont2024regularizedintegralsmanifoldslog}) $$\mathbf{R}\Gamma(X,\Omega_{X/k})\otimes_k \mathbb{C}\simeq \mathbf{R}\Gamma(X^\mathrm{KN},\mathcal{A}_{X^\mathrm{KN}})\otimes_\mathbb{R}\mathbb{C} \simeq \mathcal{A}({X^\mathrm{KN}})\otimes_\mathbb{R}\mathbb{C}$$ where we denoted $\mathcal{A}(\Sigma):= \Gamma(\Sigma,\mathcal{A}_\Sigma)$.  In particular, for $\mathbf{x,y}$ tangential basepoints on $X$, the de Rham path space based at $\mathbf{x},\mathbf{y}$ satisfies $$\+^{}_\mathbf{x}\mathrm{P}^\mathrm{dR}_\mathbf{y}X \otimes_k \mathbb{C} \simeq \+_\mathbf{x}\Baar_\mathbf{y}(\mathcal{A}(X^{\mathrm{KN}})) \otimes_\mathbb{R} \mathbb{C}.$$ Where $\mathbf{x},\mathbf{y}$ are viewed as virtual morphisms $*\rightarrow X^\mathrm{KN}$.
 
Note that $\mathcal{A}(X^\mathrm{KN})$ is a commutative differential graded algebra, so that we can use the description from example \ref{classical}: $\;_\mathbf{x}\Baar_\mathbf{y}(\mathcal{A}(X^{\mathrm{KN}}))$ is interpreted as a complex of formal iterated integrals; its elements are denoted by $$[\omega_1|\dots| \omega_k]:= \omega_1 \otimes \dots \otimes \omega_k,\; \omega_i \in \mathcal{A}^{n_i}(X^\mathrm{KN})$$ in degree $\sum_i n_i - k$, and the differentials are described for instance in \cite[3.250]{Javier}.

\smallskip

The equivalence above can be refined at the filtered level: $$\+^{}_\mathbf{x}\mathrm{P}^\mathrm{dR}_\mathbf{y}X_{\leq n} \otimes_k \mathbb{C} \simeq \+_\mathbf{x}\Baar_\mathbf{y}(\mathcal{A}(X^{\mathrm{KN}}))_{\leq n} \otimes_\mathbb{R} \mathbb{C}.$$
The complex on the right is the subcomplex $\!_\mathbf{x}\Baar_\mathbf{y}(\mathcal{A}(X^{\mathrm{KN}})) \otimes_\mathbb{R} \mathbb{C}$ spanned by elements $[\omega_1|\dots| \omega_k]$ with $k \leq n$. As before, we can see that there is an isomorphism of complexes $$\!_\mathbf{x}\Baar_\mathbf{y}(\mathcal{A}(X^\mathrm{KN}))_{\leq n} \simeq \mathrm{Tot}\mathcal{A}(\+^{}_\mathbf{x}\mathrm{P}^\bullet_\mathbf{y}X^\mathrm{KN}_{\leq n}).$$

We will use in the following subsections the description of this isomorphim in \cite[3.284]{Javier} that we specialize into the following:

\begin{lem}\label{de rham explicit}
The isomorphism is given by \begin{center}$\begin{array}{ccc}
    \!_\mathbf{x}\Baar_\mathbf{y}(\mathcal{A}(X^\mathrm{KN}))_{\leq n} & \rightarrow & \mathrm{Tot}\mathcal{A}(\+_\mathbf{x}\mathrm{P}^\bullet_\mathbf{y}X^\mathrm{KN}_{\leq n})\\
    \omega_1,\dots,\omega_k & \mapsto & (\omega'_i)_{0 \leq i \leq n}
\end{array}$ \end{center}
where $$\omega'_i = \left\{
    \begin{array}{ll}
        (-1)^{\sum(k-i)\mathrm{deg}(\omega_i)}\mathrm{pr}_1^*\omega_1 \wedge \cdots \wedge \mathrm{pr}_k^*\omega_k  & \mbox{if } i = k\\
        0 & \mbox{else.}
    \end{array} \right. $$
\end{lem}

\begin{rem}
For $[\omega_1|\dots|\omega_k] \in \+_\mathbf{x}\Baar_\mathbf{y}(\mathcal{A}(X^\mathrm{KN}))_{\leq n}$ with $k\leq n$ and $deg(\omega_i) = 1$, we see that the outcome of the previous isomorphism is $(-1)^{\frac{k(k-1)}{2}}\mathrm{pr}_1^*\omega_1 \wedge \cdots \wedge \mathrm{pr}_k^*\omega_k$ in position $k$.
\end{rem}

\subsection{The comparison isomorphism}\label{compa}\;\smallskip

We now describe the Betti--de Rham comparison isomorphism applied to the motivic pointed path space based at tangential basepoints. This will allow us in the next subsection to compute the periods of the motivic fundamental groupoid at tangential basepoints. Let us remind that the log Betti and log de Rham cohomologies are, as in the classical setting, compared in a canonical way by the work of Kato and Nakayama \cite{10.2996/kmj/1138044041}. Namely, for $X$ an \textit{ideally log smooth log scheme over} $k$, there is a natural isomorphism $$\mathrm{H}^*(X,\Omega_{X/k})\otimes \mathbb{C} \simeq \mathrm{H}^*(X^\mathrm{KN},\mathbb{C}).$$
\begin{rem}\label{logptcomp}
The definition of ideally log smooth log schemes over a field $k$, as given in \cite{Ogus_2018}, characterizes log schemes étale locally modeled on $\mathrm{Spec}(k[P]/(\Sigma))$ where $P$ is an fs monoid, and $\Sigma$ is the ideal of $k[P]$ generated by an ideal $\Sigma$ of $P$. This definition includes in particular the standard log point $*_\mathbb{N} \simeq \mathrm{Spec}(k[\mathbb{N}]/\mathbb{N}^*)$.
\end{rem}

The comparison theorem refines into:

\begin{thm}
    For $\Lambda$ a subring of $\mathbb{C}$, there is an equivalence of symmetric monoidal $\infty$-functors $$\mathrm{R}^\otimes_\mathrm{B}\otimes_\mathbb{Q}\mathbb{C} \simeq \mathrm{R}^\otimes_\mathrm{dR}\otimes_k \mathbb{C} $$
\end{thm}
\begin{proof}

   By the result of lemma \ref{linear yoneda}, to prove the theorem it suffices to compare the associated log mixed Weil theories. Note that by the construction of these functors and after \cite[Proposition 3.5]{blumberg2015uniquenessmultiplicativecyclotomictrace}, we only need to compare the associated functors of $1$-categories $$\mathbf{lSm}_k^{\times,op} \rightarrow \mathbf{C}(\mathbf{Mod}_\Lambda)^\otimes.$$
   The log Betti-log de Rham comparison theorem \cite[0.2]{10.2996/kmj/1138044041} gives a weak equivalence of lax monoidal functors $$\mathbf{R}\Gamma((-)^{\operatorname{KN}},\mathbb{Q})\otimes\mathbb{C} \simeq \mathbf{R}\Gamma(-,\Omega_{-/k})\otimes\mathbb{C}$$ which directly shows the expected result. 
\end{proof}

It then suffices to apply the Betti--de Rham comparison to the motivic pointed path space to see that for $X\in \mathbf{Div}_k$, and $\mathbf{x,y}$ tangential basepoints on $X$, there is natural isomorphisms of schemes compatible with the groupoid structure
    $$\pi_1^{\mathrm{un}}(X^{\mathrm{KN}},\mathbf{x},\mathbf{y})\times_{\mathbb{Q}}\mathbb{C} \simeq \pi_1^\mathrm{B}(X,\mathbf{x},\mathbf{y})\times_{\mathbb{Q}} \mathbb{C} \simeq \pi_1^\mathrm{dR}(X,\mathbf{x},\mathbf{y})\times_{k} \mathbb{C}.$$
The only thing that needs to be checked is that the Betti and de Rham realizations of the motivic augmentations associated with $\mathbf{x}$ and $\mathbf{y}$ are naturally isomorphic. This is trivial using remark \ref{logptcomp} and lemma \ref{log pt real}.

This result can be refined by applying the Betti--de Rham comparison to the filtered pieces of the motivic pointed path space to obtain an isomorphism of $k$-algebras for all $n\in \mathbb{N}$
  $$\mathcal{O}(\pi_1^\mathrm{B}(X,\mathbf{x},\mathbf{y}))_{\leq n} \otimes \mathbb{C} \simeq \mathcal{O}(\pi_1^\mathrm{dR}(X,\mathbf{x},\mathbf{y}))_{\leq n} \otimes \mathbb{C}$$.

Recall now that for $\Sigma$ a manifold with log corners (see the end of subsection \ref{deRhamreal2}), its de Rham and singular cohomologies are compared via a log Poincaré lemma. Namely, the canonical inclusion of complex of sheaves $\Sigma$ $$\underline{\mathbb{R}} \hookrightarrow \mathcal{A}_{\Sigma}$$ is a quasi-isomorphism (\cite[Corollary 6.9]{dupont2024regularizedintegralsmanifoldslog}). This is compatible with the algebraic log de Rham, log Betti comparison theorem in the sense that for $X\in \mathbf{Div}_k$, the comparison isomorphism between its log de Rham and log Betti cohomologies is given after tensoring by $\mathbb{C}$ the quasi-isomorphism $$C^*(X^\mathrm{KN},\mathbb{R}) \simeq \mathcal{A}(X^\mathrm{KN}) $$ in a natural way. From the descriptions at the end of the last two sections, we see that the filtered comparison isomorphism applied to the motivic path space based at $\mathbf{x},\mathbf{y}$ is given after tensoring by $\mathbb{C}$ and taking the $0$-th homology on the following quasi-isomorphism of complexes:

     $$\mathrm{Tot}\+C^*(\+^{}_\mathbf{x}\mathrm{P}^\bullet_\mathbf{y}X^\mathrm{KN}_{\leq n},\mathbb{R}) \simeq \mathrm{Tot}\mathcal{A}(\+^{}_\mathbf{x}\mathrm{P}^\bullet_\mathbf{y}X^\mathrm{KN}_{\leq n}). $$

\subsection{Periods of the motivic fundamental groupoid}\label{periods}\;\smallskip

With the previous interpretation of the comparison isomorphism, we can now compute the periods of the motivic fundamental groupoid at tangential basepoints via \textit{regularized integration}. This will allow us to prove in general a variant of Chen's theorem (see theorem \ref{end}) for tangential basepoints relating the prounipotent completion of the fundamental group with the right notion of iterated integrals. In this last section, $k \subset \mathbb{C}$.

\subsubsection{Regularized integration}\;\smallskip

Let us first recall some facts on regularized integrals introduced in full generality in \cite{dupont2024regularizedintegralsmanifoldslog}. We refer the reader to the introduction to positive log geometry at the end of subsection \ref{deRhamreal2}.

\medskip

For $\Sigma$ a manifold with log corners and for $i > 0$, we denote by $\partial^i\Sigma$, and we call i-th boundary of $\Sigma$ the manifold with log corners described locally by iterating on the following $$\partial([0,\infty)^n\times[0)^k) = \bigsqcup_{i=0}^n[0,\infty)^{i-1}\times[0)\times[0,\infty)^{n-i}\times [0)^k.$$
We omit the $i$ in the notation $\partial^i\Sigma$ when $i = 1$, and we let $\partial^0 \Sigma := \Sigma$.

\smallskip

If $i >0$, there are $i$ natural ordinary morphisms of manifold with log corners \begin{center}
\begin{tikzcd}
    \partial^i \Sigma  \arrow[r, draw=none, "\raisebox{+1.5ex}{\vdots}" description] \arrow[r, shift left = 3] \arrow[r, shift right = 3] & \partial^{i-1} \Sigma.
\end{tikzcd}
\end{center}
This way, the family $(\Sigma, \partial^\bullet\Sigma)$ defines an augmented symmetric semi-simplicial manifold with log corners, as defined in the discussion introducing the definition \ref{finfin}.

\begin{exmp}
The topological $n$-simplex $\Delta^n$ (with the model described in \ref{model}) is endowed with the structure of a manifold with corners hence of a manifold with log corners. Recall that one needs to be careful in the choice of the associated positive log structure. It is not true in general that, when equipped with the trivial positive log structure $\mathcal{C}^{\infty,>0}_\Sigma$, a manifold with corners $\Sigma$ is a manifold with log corners. We can instead use \cite[Section 3.2.3]{dupont2024regularizedintegralsmanifoldslog} where the choice consists in pulling back the positive log structures on each local chart $U \rightarrow [0,\infty)^n$.
   \begin{itemize}
       \item When $n = 2$, the boundary of $\Delta^2$ can be described in a simple way as the disjoint union of its faces, each isomorphic as manifold with log corners to $\Delta^1\times [0)$.
    The picture in this case is:
    
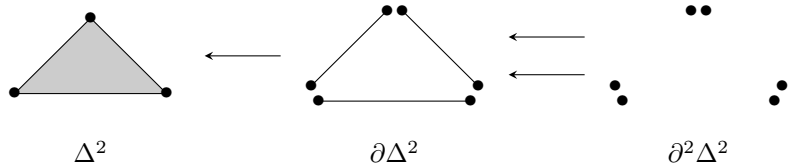
\begin{center}
    \begin{tikzpicture}
    \draw  (-5, -1) -- (-4, 0);
    \draw  (-4,0) -- (-3,-1);
    \draw  (-5,-1) -- (-3,-1);
    \draw (-5, -1) node {$\bullet$};
    \draw (-4,0) node {$\bullet$};
    \draw (-3,-1) node {$\bullet$};
    \draw [-stealth] (-1.5,-0.5)--(-2.5,-0.5);
    \draw  (-1.1,-0.9) -- (-0.1,0.1);
    \draw (0.1,0.1) -- (1.1,-0.9);
    \draw (-1,-1.1) -- (1,-1.1);
    \draw [-stealth] (2.5,-0.75) -- (1.5,-0.75);
    \draw [-stealth] (2.5,-0.25) -- (1.5,-0.25);
    \draw (0.1,0.1) node {$\bullet$};
    \draw (-1.1,-0.9) node {$\bullet$};
    \draw (1.1,-0.9) node {$\bullet$};
    \draw (-1,-1.1) node {$\bullet$};
    \draw (1,-1.1) node {$\bullet$};
    \draw (-0.1,0.1) node {$\bullet$};
    
    \draw (3.9,0.1) node {$\bullet$};
    \draw (2.9,-0.9) node {$\bullet$};
    \draw (5.1,-0.9) node {$\bullet$};
    \draw (3,-1.1) node {$\bullet$};
    \draw (5,-1.1) node {$\bullet$};
    \draw (4.1,0.1) node {$\bullet$};
    \node[below] at (-4,-1.5) {$\Delta^2$};
    \node[below] at (0,-1.5) {$\partial\Delta^2$};
    \node[below] at (4,-1.5) {$\partial^2\Delta^2$};
    
    \fill [fill=black,opacity = 0.2]
     (-5, -1)  
  -- (-4,0) 
  -- (-3,-1);
    
\end{tikzpicture}
\end{center}
   The second boundary $\partial^2\Delta^2$ is a disjoint union of two copies of each vertex of $\Delta^2$ endowed with the positive log structure $[0)^2$. The two rightmost morphisms are the natural immersion, and the morphism deduced via the action of the symmetric group $\mathfrak{S}_2$ with exchanges copies of the same vertex. 
   Note that $\partial^i \Delta^2$ is empty for $i > 2$.

    \item In general, $\partial^i\Delta^n$ is a disjoint union of $i!$ copies of each $(n-i)$-dimensional face in $\Delta^n$, and there is an action of the symmetric group $\mathfrak{S}_i$ permuting them.

   \end{itemize} 
\end{exmp}

For $\Sigma$ a manifold with log corners, we can consider the associated basic manifold with log corners $\Sigma^\mathrm{bas}$, defined locally by the property that $$([0,\infty)^n \times [0)^k)^\mathrm{bas} := [0,\infty)^n. $$ In other words, the underlying manifolds with log corners of $\Sigma$ and its basic version are the same, and the positive log structure of $\Sigma^\mathrm{bas}$ is obtained from the original on $\Sigma$ by restricting to the submonoid of functions whose germ at each point of $\Sigma$ is not sent to zero by the germ of the morphism $\mathcal{M}_\Sigma \rightarrow \mathcal{C}_\Sigma^{\infty,\geq 0}.$ There is a morphism of manifolds with log corners $$\Sigma \rightarrow \Sigma^\mathrm{bas}$$ induced by the inclusion $\mathcal{M}_{\Sigma^\mathrm{bas}}\hookrightarrow \mathcal{M}_{\Sigma}.$ We say that $\Sigma$ is basic if $\Sigma^\mathrm{bas} \simeq \Sigma$.

\medskip

In \cite[Section 3.8.6]{Javier}, the definition of the regularized integral of a differential form on the projective line with logaritmic poles at $\{0,1,\infty\}$ is based on the choice of tangential basepoints at these points. Indeed, it is defined \cite[Definition 3.364]{Javier} as the residue of a \textit{logarithmic asymptotic expansion} of the integral along a path whose starting and endpoints converge to removed points in the direction prescribed by the tangential datum. In higher dimension, the full definition appears in \cite{dupont2024regularizedintegralsmanifoldslog}. The choice of tangential basepoints at a finite set of points is replaced by the definition of the \textit{regularization of a manifold with log corners}. Although this definition, and the definition of regularized integration work in a very general setting, we focus on the case of $\Sigma = \Delta^n$.

\medskip

\begin{defi}
A nondegenerate regularization on $\Delta^n$ is defined as a positive section of the normal bundle of the immersion of the underlying manifold with corners $$\underline{\partial\Delta}^n \hookrightarrow \underline{\Delta}^n.$$
\end{defi}
This can be viewed as the datum of a family of tangential basepoints on the boundary. Tangential basepoints in this setting are defined in a similar way as in the algebraic setting, by the data of a point and of a positive normal vector at this point (outside the boundary it is just a point) see \cite[Definition 4.8]{dupont2024regularizedintegralsmanifoldslog} for details.

By \cite[Proposition 4.9]{dupont2024regularizedintegralsmanifoldslog}, we know that there is still a bijection between tangential basepoints on a basic manifold with log corners, and virtual morphisms from the point, equipped with the trivial positive log structure.

\begin{exmp} \begin{itemize}
    \item A virtual morphism $$* \rightarrow [0,\infty)$$ to $0$ is determined by the  pullback $\lambda \in \mathbb{R}_{>0}$ of the coordinate $r \in [0,\infty)$. As in the algebraic case, this is associated to the normal vector $\lambda \partial_{r}|_0.$
    \item More generally, a virtual morphism $$[0,\infty) \rightarrow [0,\infty)^2$$ which identifies the underlying manifold with corners $[0,\infty[$ with the component $r_2 = 0$ of the boundary of $[0,\infty)^2$ (where we let $(r_1,r_2)$ be coordinates for $[0,\infty)^2$) is the same thing as a family of tangential basepoints on the component $r_2 = 0$, determined by the pullback of $r_2$ on $\mathcal{M}_{[0,\infty)}$.
    Note that this datum is equivalent to the datum of a virtual morphism $$ [0,\infty) \simeq ([0,\infty)\times [0))^\mathrm{bas} \rightarrow [0,\infty)\times [0),$$ which is a section of the morphism $[0,\infty)\times [0) \rightarrow ([0,\infty)\times [0))^\mathrm{bas}$ and where the tangential datum is now given by the pullback of the phantom coordinate $t_2.$
\end{itemize}

\end{exmp}

In the same fashion as this last example, if we let $\partial_i\Delta^n$ be the i-th face of $\Delta^n$ with the induced positive log structure, a family of positive normal vectors to $\underline{\partial_i\Delta}^n$ in $\underline{\Delta}^n$ is the same as the datum of a section $$(\partial_i\Delta^n)^\mathrm{bas} \rightarrow \partial_i\Delta^n$$ to the morphism $\partial_i\Delta^n \rightarrow (\partial_i\Delta^n)^\mathrm{bas}$, the normal vector being determined by the pullback of the coordinate vanishing on the face $\partial_i\Delta^n$ to $\mathcal{M}_{(\partial_i\Delta^n)^\mathrm{bas}}$ ($t_{i+1}-t_i$ if $0 < i <n$, $t_0-1$ if $i = 0$ and $t_n$ if $i = n$).

Therefore, there is a commutative diagram of manifolds with log corners
\begin{center}
\begin{tikzcd}
\Delta^n \arrow[d,equal] & \partial\Delta^n \arrow[l]\\
(\Delta^n)^\mathrm{bas} & (\partial\Delta^n)^\mathrm{bas} \arrow[l], \arrow[u]
\end{tikzcd}
\end{center}
More generally, it is explained in \cite[Proposition 7.4]{dupont2024regularizedintegralsmanifoldslog}, that the datum of a nondegenerate regularization $s$ of $\Delta^n$ induces, by pulling back $s$ along all the $k$ different immersions $\underline{\partial^k \Delta^n} \hookrightarrow \underline{\partial \Delta^n}$ a family of positive normal vectors on to $\underline{\partial^k \Delta^n}$, and hence a section $$(\partial^k \Delta^n)^\mathrm{bas} \rightarrow \partial^k \Delta^n.$$ The data of these sections for $k \in \mathbb{N}$ induces in a unique way a structure of augmented symmetric simplicial manifold with log corners, with virtual morphisms, $((\Delta^n)^\mathrm{bas},(\partial^\bullet\Delta^n)^\mathrm{bas})$ (note that the operation $-^\mathrm{bas}$ is not functorial, see for instance the inclusion $[0) \hookrightarrow [0,\infty)$), and a morphism of augmented symmetric semi-simplicial manifolds with log corners $$s : ((\Delta^n)^\mathrm{bas},(\partial^\bullet\Delta^n)^\mathrm{bas}) \rightarrow (\Delta^n,\partial^\bullet \Delta^n).$$

We refer the reader to \cite[Subsection 7.1]{dupont2024regularizedintegralsmanifoldslog} for details. Note that we will not need to describe explicitely the structure of augmented symmetric semi-simplicial manifolds with log corners of $((\Delta^n)^\mathrm{bas},(\partial^\bullet\Delta^n)^\mathrm{bas})$.

\begin{defi}
   Let us consider the section of the normal bundle of the immersion $\underline{\partial\Delta}^n \hookrightarrow \underline{\Delta}^n$ obtained as \begin{center}\begin{itemize}
       \item $-\frac{\partial}{\partial t_1}\big|_{\underline{\partial_1\Delta}^n}$ on $\underline{\partial_1\Delta}^n$,
       \item $(\frac{\partial}{\partial t_i} -\frac{\partial}{\partial t_{i+1}})\big|_{\underline{\partial_i\Delta}^n}$ on $\underline{\partial_i\Delta}^n$ if $0 <i <n$,
       \item $\frac{\partial}{\partial t_n}\big|_{\underline{\partial_n\Delta}^n}$ on $\underline{\partial_n\Delta}^n$.
   \end{itemize}\end{center} We denote by $$s:((\Delta^n)^\mathrm{bas},(\partial^\bullet\Delta^n)^\mathrm{bas}) \rightarrow (\Delta^n,\partial^\bullet \Delta^n)$$ the induced morphism of augmented symmetric semi-simplicial manifolds with log corners where we made implicit the structure of augmented symmetric semi-simplicial manifold with log corners of $((\Delta^n)^\mathrm{bas},(\partial\Delta^n)^\mathrm{bas})$
\end{defi}

With this regularization at hand, for $\omega \in \mathcal{A}^{n}(\Delta^n)$ a logarithmic differential form on $\Delta^n$, the regularized integral of $\omega$ on $\Delta^n$ denoted $$\int_{(\Delta^n,s)}\omega$$ is defined as the classical integral of the cohomology class $[s^*\omega]$ in $$\mathrm{H}^n_\mathrm{dR}((\Delta^n)^\mathrm{bas},(\partial^\bullet\Delta^n)^\mathrm{bas}) \simeq \mathrm{H}^n_\mathrm{dR}(\underline{\Delta}^n,\underline{\partial^\bullet\Delta}^n)$$ on $\underline{\Delta}^n$. More generally, for $(X,Y^\bullet)$ an augmented symmetric semi-simplicial manifold with log corners, and $\omega \in \mathcal{A}^n(X)$ a logarithmic differential form of degree $n$ vanishing on $Y^\mathbf{1}$, if there is a morphism $$\phi: ((\Delta^n)^\mathrm{bas},(\partial^\bullet\Delta^n)^\mathrm{bas};s) \rightarrow (X,Y^\bullet)$$ of symmetric semi-simplicial manifolds with log corners, then the regularized integral of $\omega$ on $\phi$ is defined by $$\int_{\phi}\omega := \int_{(\Delta^n,s)}\phi^*\omega.$$ By functoriality of the comparison theorem we have $$\langle \phi_*[\underline \Delta^n], [\omega] \rangle = \int_{(\Delta^n,s)}\phi^*\omega,$$ via the pairing $$\mathrm{H}_*^\mathrm{Sing}(X,Y^\bullet)\otimes\mathrm{H}_\mathrm{dR}^*(X,Y^\bullet) \rightarrow \mathbb{R}$$ see \cite[Proposition 8.11]{dupont2024regularizedintegralsmanifoldslog}.

\subsubsection{Regularized iterated integration}\;\smallskip

Let us now compute the periods of the motivic fundamental groupoid in terms of \textit{regularized iterated integrals}. We first define this notion in definition \ref{regularized}. Let $X \in \mathbf{Div}_k$ and $\mathbf{x},\mathbf{y}$ tangential basepoints on $X$.

\medskip

Let $\gamma:\, (0,1)\, \rightarrow \underline{X}(\mathbb{C})\setminus \underline{D}(\mathbb{C})$ be a smooth path of $\underline{X}(\mathbb{C})\setminus \underline{D}(\mathbb{C})$ based at the tangential basepoints $\mathbf{x},\mathbf{y}$ (definition \ref{enfin}). Recall that since $\gamma$ can be continuously extended to a map $[0,1] \rightarrow X^\mathrm{KN}$, we can define a continuous map $$\sigma_{\gamma,i}: \Delta^i \rightarrow (X^\mathrm{KN})^i$$ for $i \in \mathbb{N}$ as in the discussion following remark \ref{signe}.

\begin{rem}\label{coordinates} Since $\gamma$ is smooth, it induces a morphism of manifolds with corners $[0,1] \rightarrow X^\mathrm{KN}$.
In coordinates, locally around $\mathbf{x}$ or $\mathbf{y}$, $X^\mathrm{KN} \simeq (\mathbb{R}_{+}\times S^1)^k\times (\mathbb{R}_+^{*})^{k'}$ so that locally around $0$ for instance, $\gamma(t) = (r_1(t), \theta_1(t), \dots, r_k(t),\theta_k(t), x'_1(t), \dots, x'_{k'}(t))$ and $r_i$ satisfies $r_i(t) \sim_0 t\lambda_i$ with $\lambda_i > 0$. Therefore, using the condition specified in the discussion following \cite[Definition 3.14]{dupont2024regularizedintegralsmanifoldslog}, we see that $\gamma$ is uniquely lifted into an ordinary morphism of manifolds with log corners $$[0,1] \rightarrow X^\mathrm{KN}.$$ 
\end{rem}

More generally the following lemma is proven in the same fashion working locally in coordinates:

\begin{lem}
    For $i \in \mathbb{N}$ and under the hypotheses on $\gamma$, the map $\sigma_{\gamma,i}$ has the structure of an ordinary morphism of manifolds with log corners.
\end{lem}

The following proposition will allow us to define regularized iterated integration. Recall that we defined at the end of subsection \ref{Bettikato} an augmented symmetric semi-simplicial object in $\!^\mathbf{v}\mathbf{Div}_k$ denoted $(X^k,Y^\bullet).$ Since the Kato-Nakayama space is functorial with respect to virtual morphisms, this defines an augmented symmetric semi-simplicial manifold with log corners (with virtual morphisms between the objects) $((X^\mathrm{KN})^k,(Y^\bullet)^\mathrm{KN})$.
\begin{prop}\label{structure}
For $\gamma$ defined as above, $\sigma_{\gamma,k}$ defines an ordinary morphism of augmented symmetric semi-simplicial manifolds with log corners $$(\Delta^k,(\partial^\bullet\Delta^k)^\mathrm{bas};s) \rightarrow ((X^\mathrm{KN})^k,(Y^\bullet)^\mathrm{KN}).$$
\end{prop}

First, note that since for $0 \leq i \leq k$, the i-th face of the standard $k$-simplex is given by \begin{itemize}
        \item $\{t_1 = 1\}$ if $i = 0$ 
        \item $\{t_i = t_{i+1} \}$ if $0 < i < k $
        \item $\{t_k = 0\}$ if $i = k$
    \end{itemize}
    we remark that $\sigma_{\gamma,k}$ induces a morphism of manifolds with corners  $$(\sigma_{\gamma,k})_{|\partial_i\Delta^k}: \underline{\partial_i\Delta}^k \rightarrow \underline{Y_i^\mathrm{KN}}.$$ This map induces a morphism of manifolds with log corners by the following identification \begin{center}
\begin{tikzcd}
    (\partial_i\Delta^k)^\mathrm{bas} \arrow[r,"\sim"] \arrow[d,"(\sigma_{\gamma,k})_{|\partial^i\Delta^k}"] & \Delta^{k-1} \arrow[d,"\sigma_{\gamma,k-1}"]\\
    Y_i^\mathrm{KN} \arrow[r,"\sim"] & (X^\mathrm{KN})^{k-1}
\end{tikzcd}
    \end{center}
    
Iterating this construction, we see that for $\mathrm{I} = (i_1,\dots, i_j) \subset [k]$ with $i_1 < \dots < i_k$, $\sigma_{\gamma,k}$ induces by restriction a morphism of manifolds with log corners $$(\partial_{i_1}\dots\partial_{i_k}\Delta^k)^\mathrm{bas} \rightarrow Y_\mathrm{I}^\mathrm{KN}$$ identified with the morphism $\sigma_{\gamma,k-j}.$

Therefore, there is a diagram of manifold with log corners 

\begin{center}
        \begin{tikzcd}
        (X^{KN})^k & (Y^1)^\mathrm{KN} \arrow[l] & (Y^2)^\mathrm{KN} \arrow[l,shift left = 2] \arrow[l,shift right = 2] & (Y^3)^\mathrm{KN} \arrow[l,shift left = 3] \arrow[l] \arrow[l,shift right = 3] & \dots\\
        \Delta^k \arrow[u] & (\partial^1 \Delta^k)^\mathrm{bas} \arrow[l] \arrow[u] & (\partial^{2} \Delta^k)^\mathrm{bas} \arrow[l,shift left = 2] \arrow[u] \arrow[l,shift right = 2] & (\partial^3\Delta^k)^\mathrm{bas} \arrow[l,shift left = 3] \arrow[u] \arrow[l] \arrow[l,shift right = 3]  & \dots
        \end{tikzcd}
    \end{center}
    
such that the underlying diagram of underlying manifolds with corners commutes. Let us prove the proposition \ref{structure}.
\begin{proof}
    We can work by induction on $k$. When $k = 0$ nothing needs to be done.\\
    Suppose that the statement is true for $k - 1 \geq 0$. Then we only need to show that the following square of manifolds with log corners commutes
    \begin{center}
        \begin{tikzcd}
        (X^k)^\mathrm{KN} & (Y^1)^\mathrm{KN} \arrow[l] \\
         \Delta^k \arrow[u] & (\partial^1 \Delta^k)^\mathrm{bas}.  \arrow[l] \arrow[u]
        \end{tikzcd}
    \end{center}
    We work on each face separately, we wish to show that for $0 \leq i \leq k$.  the diagram \begin{center}
        \begin{tikzcd}
        (X^k)^\mathrm{KN} & Y_i^\mathrm{KN} \arrow[l] \\
         \Delta^k \arrow[u] & (\partial_i \Delta^k)^\mathrm{bas}.  \arrow[l] \arrow[u]
        \end{tikzcd}
    \end{center}
    commutes. Note that by the so called ``continuity principle" \cite[Proposition 4.3]{dupont2024regularizedintegralsmanifoldslog}, we only need to check that for each section of $\mathcal{M}_{(X^k)^\mathrm{KN}}$ pulled back to the zero function in $\mathcal{C}^{\infty,\geq 0}_{(\partial_i \Delta^k)^\mathrm{bas}}$ by the down left composition $$f: (\partial_i\Delta^k)^\mathrm{bas} \rightarrow \Delta^k \rightarrow (X^k)^\mathrm{KN},$$ its pull back in $\mathcal{M}_{(\partial_i \Delta^k)^\mathrm{bas}}$ by $f$ and by the upper right composition $$g:(\partial_i\Delta^k)^\mathrm{bas} \rightarrow Y_i \rightarrow (X^k)^\mathrm{KN}$$ coincide. The rest of the sections behave according to the morphism of manifolds with corners structure.
    \begin{itemize}
        \item If $0 < i < k$, by working locally in coordinates we see that no sections of $\mathcal{M}_{(X^k)^\mathrm{KN}}$ are pulled back to the zero section by $f$.
        \item If $i = k$, then $(\partial_i \Delta^k)^\mathrm{bas} = \{t_1 = 0\}^\mathrm{bas}$ so that $\sigma_{\gamma,k}$ when restricted is the morphism $$(t_1 = 0,t_2, \dots, t_k) \rightarrow (\mathbf{x},\gamma(t_2),\dots, \gamma(t_k)).$$
        In coordinates, close to  $\{\mathbf{x}\}\times (X^{k-1})^\mathrm{KN}$ (we use the same coordinates as in the remark \ref{coordinates}), we can write in a neighborhood of $t_1 = 0$ $$\sigma_{\gamma,k}(t_1,\dots,t_k) = (r_1(t_1), \theta_1(t_1), \dots, r_k(t_1),\theta_k(t_1), x'_1(t_1), \dots, x'_{k'}(t_1), \gamma(t_2),\dots,\gamma(t_k))$$ we then see clearly that the only coordinates pulled back on the zero section are $r_1,\dots, r_k$.\\
        
        Let $\lambda_i:= \mathbf{x}^*r_i$. On one hand we see that $\sigma_{\gamma,k}^*r_i = \lambda_i$, which readily shows that $g^*r_i = \lambda_i$.
        On the other hand, since $\gamma$ is a path based at the tangential basepoint $\mathbf{x}$ in zero, we have that $r'_i(0) = \lambda_i$. The pullback of $r_i$ in $\mathcal{M}_{(Y_i)^\mathrm{KN}}$ by the morphism $Y_i \rightarrow (X^k)^\mathrm{KN}$ is  $r_i(t_1)$, and the pullback of $r_i(t_1)$ by the virtual morphism $(\partial_i\Delta^k)^\mathrm{bas} \rightarrow \Delta^k$ in $\mathcal{M}_{(\partial_i\Delta^k)^\mathrm{bas}}$ is $r_1'(0)$ since this morphism is determined by the virtual morphism $\partial_{t_1 = 0}$ by the choice of the section of the normal bundle. Hence the result.
   
    \item If $i = 0$ the situation is the same as the previous one.   
    \end{itemize}
    \end{proof}

We can then make the following definition:

\begin{defi}\label{regularized}
Let $X\in \mathbf{Div}_k$, $\mathbf{x},\mathbf{y}$ tangential basepoints on $X$, and $\gamma: \,(0,1)\, \rightarrow \underline{X}(\mathbb{C})\setminus \underline{D}(\mathbb{C})$ a smooth path of $\underline{X}(\mathbb{C})\setminus \underline{D}(\mathbb{C})$ based at $\mathbf{x},\mathbf{y}$. Let $\omega_1,\dots,\omega_k$ be logarithmic differential $1$-forms. We define $$\int_\gamma \omega_1\dots\omega_k:= \int_{\sigma_{\gamma,k}} \mathrm{pr}_1^*\omega_1 \wedge \cdots \wedge \mathrm{pr}_k^*\omega_k = \int_{(\Delta^k,s)}\sigma_{\gamma,k}^*(\mathrm{pr}_1^*\omega_1 \wedge \cdots \wedge \mathrm{pr}_k^*\omega_k)
$$ the \textbf{regularized iterated integral} of $\omega_1,\dots,\omega_k$ along $\gamma$.
\end{defi}

\begin{rem}
When $\omega_1,\dots,\omega_k$ are smooth forms of the manifold with corners $\underline{X^\mathrm{KN}}$ (for instance for smooth algebraic forms), we know from \cite[Proposition 7.2]{dupont2024regularizedintegralsmanifoldslog} that this definition coincide with the usual definition of iterated integral introduced by Chen, see \cite{bams/1183539443}. This is true in particular when $\mathbf{x}, \mathbf{y}$ are rational basepoints of $\underline{X}\setminus \underline{D}$.
\end{rem}

\subsubsection{The periods of the motivic fundamental groupoid}\;\smallskip

We are now able to completely determine the periods of the motivic fundamental groupoid with tangential basepoints:

\begin{thm}\label{end}
    The isomorphism of algebras $$\mathcal{O}(\pi_1^\mathrm{B}(X,\mathbf{x},\mathbf{y}))_{\leq n}\otimes_\mathbb{Q} \mathbb{C} \simeq \mathcal{O}(\pi_1^\mathrm{dR}(X,\mathbf{x},\mathbf{y}))_{\leq n} \otimes_k \mathbb{C}$$ is induced by the pairing  \begin{center}
        $\begin{array}{ccc}
            \mathbb{C}[\pi_1(\underline{X}(\mathbb{C})\setminus \underline{D}(\mathbb{C}),\mathbf{x},\mathbf{y})] \otimes_\mathbb{R} \mathrm{H}^0\big(\+_\mathbf{x}\Baar_\mathbf{y}(\mathcal{A}(X^\mathrm{KN}))\big) & \rightarrow & \mathbb{C} \\
            \gamma \otimes [\omega_1|\dots|\omega_k] & \mapsto & \int_\gamma \omega_1\dots\omega_k
        \end{array}$
    \end{center}
    
\end{thm}
\begin{proof}
We can suppose $X$ to be connected and we know that in this case every class in $\mathrm{H}^0\big(\+_\mathbf{x}\Baar_\mathbf{y}(\mathcal{A}(X^\mathrm{KN}))\big)$ is represented by elements of the form $[\omega_1| \cdots | \omega_k]$ where the $\omega_i$ are $1$-forms by similar arguments as \cite[Lemma 3.263]{Javier}. We have seen in subsection \ref{compa} (joined with the results at the end of subsections \ref{deRhamreal2} and \ref{Bettikato}) that the comparison isomorphism applied to the filtered pieces of the fundamental groupoid with tangential basepoints is described by the following sequence of isomorphisms (where the upper terms are defined similarly as their Betti counterpart from the end of subsection \ref{Bettikato}): \begin{center}
    \begin{tikzcd}
   & \mathbf{\Gamma}_\mathrm{Sing}(\+_\mathbf{x}\mathrm{P}_\mathbf{y}X^\mathrm{KN}_{\leq n}) \arrow[r,"\sim"] & \mathbf{\Gamma}_\mathrm{Sing}((X^\mathrm{KN})^n,(Y^\bullet)^\mathrm{KN}) \\
     \+_\mathbf{x}\Baar_\mathbf{y}(\mathcal{A}(X^\mathrm{KN}))_{\leq n} \arrow[r, "\sim"] & \mathcal{A}(\+_\mathbf{x}\mathrm{P}_\mathbf{y}X^\mathrm{KN})_{\leq n} \arrow[u,"\rotatebox{90}{\(\sim\)}"] \arrow[r,"\sim"] & \mathcal{A}((X^\mathrm{KN})^n,(Y^\bullet)^\mathrm{KN}) \arrow[u,"\rotatebox{90}{\(\sim\)}"]
    \end{tikzcd}
\end{center}   
For $[\omega_1,\dots, \omega_k] \in \+_\mathbf{x}\Baar_\mathbf{y}(\mathcal{A}(X^\mathrm{KN}))_{\leq n}$ with $\omega_i$ $1$-forms, its image in $\mathcal{A}(\+_\mathbf{x}\mathrm{P}_\mathbf{y}X^\mathrm{KN}_{\leq n})$ is $(\omega'_i)_{0 \leq i \leq n}$ where $$\omega'_i = \left\{
    \begin{array}{ll}
        (-1)^{\frac{k(k-1)}{2}}\mathrm{pr}_1^*\omega_1 \wedge \dots \boxtimes \mathrm{pr}_k^*\omega_k  & \mbox{if } i = k\\
        0 & \mbox{else.}
    \end{array} \right.$$
which is sent to $(\omega'_\mathrm{I})_{\mathrm{I} \subset [n]}$ in $\mathcal{A}((X^\mathrm{KN})^n,(Y^\bullet)^\mathrm{KN}) \simeq \mathbf{\Gamma}_\mathrm{Sing}((X^\mathrm{KN})^n,(Y^\bullet)^\mathrm{KN})$, where  $$\omega'_\mathrm{I} = \left\{
    \begin{array}{ll}
        (-1)^{n|\mathrm{I}^c|}\epsilon(\mathrm{I}^c)(-1)^{\frac{k(k-1)}{2}}\mathrm{pr}_1^*\omega_1 \wedge \cdots \wedge \mathrm{pr}_k^*\omega_k & \mbox{si } \mathrm{I} = \{k+1,\dots,n\} = [k]^c\\
        0 & \mbox{sinon}
    \end{array} \right. $$
Note that the form $\mathrm{pr}_1^*\omega_1 \wedge \cdots \wedge \mathrm{pr}_k^*\omega_k$ restrict to $0$ on $Y^\mathbf{1}$.

After the description of the isomorphism $$\mathbb{C}[\pi_1(\underline{X}(\mathbb{C})\setminus \underline{D}(\mathbb{C}),\mathbf{x},\mathbf{y})]/I^{n+1} \xrightarrow{\sim} \mathbf{\Gamma}_\mathrm{Sing}((X^\mathrm{KN})^n,(Y^\bullet)^\mathrm{KN})^\vee,$$ the pairing applied to $[\gamma]$ and $[\omega_1|\dots|\omega_k]$ is $$\begin{array}{cc}
    & \langle((-1)^\frac{|\mathrm{I}|(|\mathrm{I}|+1)}{2}(-1)^\frac{n(n-1)}{2}\epsilon(\mathrm{I}^c)\sigma_{\gamma,\mathrm{I}})_{\mathrm{I}\subsetneq [n]}, (\omega'_\mathrm{I})_{\mathrm{I}\subsetneq [n]}\rangle     \\
      = &  \langle (-1)^\frac{|[k]^c|(|[k]^c|+1)}{2}(-1)^\frac{n(n-1)}{2}\epsilon([k])\sigma_{\gamma,[k]^c},\omega'_{|[k]|^c}\rangle\\
      = &  (-1)^\frac{|[k]^c|(|[k]^c|+1)}{2}(-1)^\frac{n(n-1)}{2}\epsilon([k])(-1)^{n|[k]|}\epsilon([k])(-1)^{\frac{k(k-1)}{2}} \langle \sigma_{\gamma,k},\mathrm{pr}_1^*\omega_1 \wedge \cdots \wedge \mathrm{pr}_k^*\omega_k\rangle
\end{array}$$
Note that the signs cancels so that we are left with $$\langle \sigma_{\gamma,k},\mathrm{pr}_1^*\omega_1 \wedge \cdots \wedge \mathrm{pr}_k^*\omega_k\rangle.$$

By proposition \ref{structure}, the map $\sigma_{\gamma,k}$ has a structure of a morphism of augmented symmetric semi-simplicial manifolds with log corners so that $$\langle \sigma_{\gamma,k},\mathrm{pr}_1^*\omega_1 \wedge \cdots \wedge \mathrm{pr}_k^*\omega_k\rangle = \int_\gamma \omega_1\dots\omega_k.$$
\end{proof}

\begin{rem}
If we apply \cite[Example 7.12]{dupont2024regularizedintegralsmanifoldslog}, we can see that for $\omega$ an algebraic $1$-form of $\mathbb{P}^1_\mathbb{C}\setminus \{0,1,\infty\}$ with logarithmic poles at $0$, $1$ or $\infty$, then the regularized integral of $\omega$ along the straight path from $0$ to $1$ denoted $\mathbf{dch}$ yields the same result as the regularization process defined in \cite[Section 3.8]{Javier}, which consists in introducing the restricted path $\mathbf{dch}_\eta : [\eta, 1-\eta] \rightarrow \mathbb{P}^1_\mathbb{C}(\mathbb{C})\setminus \{0,1,\infty\}$ for $\eta \in (0,1/2)$ and taking the residue of the logarithmic expansion of the function $$\eta \mapsto \int_{\mathbf{dch}_\eta} \omega.$$
More generally it is possible to prove that the regularized iterated integrals $$\int_\mathbf{dch} \omega_1\dots\omega_k$$ when $\omega_i \in \{\frac{dz}{1-z},\frac{dz}{z}\}$ yields the correct value of the corresponding regularized multiple zeta value as define in \cite[Section 1.7]{Javier}.
\end{rem}

 \newpage
\bibliography{main}
\bibliographystyle{amsalpha}

\end{document}